\documentclass[11pt]{article}
\usepackage[utf8]{inputenc}
\usepackage{parskip}
\usepackage[dvipsnames]{xcolor}
\usepackage[T1]{fontenc}
\usepackage{array}
\usepackage[utf8]{inputenc}
\usepackage[english]{babel}
\usepackage{pgfplots}
\usepackage{wrapfig}
\usepackage{graphicx,wrapfig,lipsum}
\usepackage{float}    
\usepackage{verbatim} 
\usepackage{amsmath, amssymb, amsthm, mathtools}
\usepackage{caption}
\usepackage{subcaption}
\usepackage[colorlinks,linkcolor=purple, citecolor=Green, urlcolor=magenta, anchorcolor=ForestGreen,bookmarks=False]{hyperref}
\usepackage{cleveref}
\usepackage{bookmark}
\usepackage{fullpage}
\usepackage{enumerate}
\usepackage{paralist}
\usepackage{xspace}
\usepackage{bbm}
\usepackage{bm}
\usepackage{url}
\usepackage{lipsum}
\usepackage{algorithm, algorithmic}
\usepackage{color}
\usepackage{microtype}
\usepackage[section]{placeins}
\usepackage{lscape}
\usepackage{multirow}
\usepackage{todonotes}
\usepackage{enumitem}
\usepackage{booktabs}
\usepackage{accents}
\setlist[enumerate,1]{label=(\roman*)}
\makeatletter

\DeclareMathOperator{\poly}{poly}

\newcommand{\norm}[1]{\left\lVert#1\right\rVert}

\newcolumntype{L}[1]{>{\raggedright\arraybackslash}p{#1}}
\newcolumntype{C}[1]{>{\centering\arraybackslash}p{#1}}
\definecolor{secgray}{gray}{0.88}

\newcommand{\cA}{\mathcal{A}}

\newcommand{\cK}{\mathcal{K}}

\newcommand{\cP}{\mathcal{P}}

\newcommand{\cR}{\mathcal{R}}
\newcommand{\cS}{\mathcal{S}}

\newcommand{\cW}{\mathcal{W}}
\newcommand{\cX}{\mathcal{X}}

\newcommand{\cZ}{\mathcal{Z}}

\newcommand{\bP}{\mathbb{P}}
\newcommand{\bQ}{\mathbb{Q}}
\newcommand{\bR}{\mathbb{R}}

\newtheorem{theorem}{Theorem}
\newtheorem{definition}{Definition}
\newtheorem{lemma}{Lemma}

\newtheorem{remark}{Remark}

\newtheorem{proposition}{Proposition}

\newtheorem{assumption}{Assumption}

\def \RR {\mathbb{R}}

\def \NN {\mathbb{N}}

\def \QQ {\mathbb{Q}}

\newcommand{\EE}{\mathbb{E}}

\newcommand{\OT}{\mathrm{OT}_c}

\DeclareMathOperator*{\argmax}{argmax}
\DeclareMathOperator*{\argmin}{argmin}

\title{Oracle-Based Distributionally Robust Optimization under \\ Optimal Transport Ambiguity Sets}

\date{}
\author{
Guixian Chen\\
University of Michigan\\
\texttt{gxchen@umich.edu}\\
\and
Salar Fattahi\\
University of Michigan\\
\texttt{fattahi@umich.edu}
\and 
Soroosh Shafiee\\
Cornell University\\
\texttt{shafiee@cornell.edu}
}

\pgfplotsset{compat=1.18}
\allowdisplaybreaks

\begin{document}

\maketitle


\begin{abstract}
Distributionally robust optimization (DRO) with optimal transport ambiguity sets is traditionally solved by reformulating the minimax problem into a single-level convex program. While theoretically tractable, these reformulations introduce numerous auxiliary variables and demanding conic constraints that scale poorly in practice. In this paper, we address this challenge by reducing the inner worst-case expectation problem exactly to a scalar budget allocation task. This structural insight yields an efficient algorithm that bypasses large lifted reformulations, alongside a fast post-processing scheme to recover an optimal worst-case distribution supported on at most $N+1$ points, where $N$ denotes the sample size. We embed this procedure within an oracle-based distributional best-response framework to directly compute an approximate primal-dual solution to the overall DRO problem. Furthermore, we extend our analysis to the dual DRO formulation, proving the existence of a least-favorable distribution supported on at most $\min\{N+n+1, KN\}$ atoms, where $n$ and $K$ denote the decision dimension and number of loss components, respectively, and provide an efficient convex programming reduction to extract it from the solution of the primal DRO. Numerical experiments demonstrate that the proposed approach significantly outperforms state-of-the-art reformulation-based solvers.
\end{abstract}


\section{Introduction}
\label{sec:intro}
Given data, distributionally robust optimization (DRO) replaces a single reference distribution by an ambiguity set $\cP$ of plausible distributions and solves the minimax problem
\begin{align}
\label{eq:dro}
    \inf_{x \in \cX} \; \sup_{\bQ \in \cP} \; \EE_{z \sim \bQ}[\ell(x,z)].
\end{align}
This problem can be viewed as a zero-sum game between a decision maker (also referred to as {\it primal player}) choosing $x \in \cX$ and an adversary (also referred to as {\it dual player}) choosing a worst-case distribution $\bQ \in \cP$. Associated with \eqref{eq:dro} is the maximin problem
\begin{align}
\label{eq:dual:dro}
    \sup_{\bQ \in \cP} \; \inf_{x \in \cX} \; \EE_{z \sim \bQ}[\ell(x,z)],
\end{align}
which we refer to as the dual DRO problem. 
When strong duality holds, problems \eqref{eq:dro} and \eqref{eq:dual:dro} share the same value. While the minimax formulation \eqref{eq:dro} has been studied extensively, its maximin counterpart \eqref{eq:dual:dro} is much less explored, especially from a computational viewpoint.

From an algorithmic viewpoint, the standard approach to solve~\eqref{eq:dro} is \emph{indirect}. One first reformulates the inner worst-case expectation problem as a finite-dimensional convex optimization problem, typically using duality and tools from robust optimization, and then solves the resulting single-level formulation with an off-the-shelf solver. This paradigm has led to important tractability results in DRO. At the same time, it often produces large lifted formulations with many auxiliary variables and constraints, and in several important cases the reformulated problems belong to more demanding conic classes, inevitably relying on general conic solvers. As a result, these methods can become computationally prohibitive even at moderate scale. In contrast, for the dual DRO problem~\eqref{eq:dual:dro}, even basic computational approaches are not yet well developed. Our goal in this paper is not to derive another reformulation of~\eqref{eq:dro}. Instead, we ask whether both~\eqref{eq:dro} and~\eqref{eq:dual:dro} can be solved \emph{directly} using tailored and scalable algorithms that take advantage of the unique structure of the problem. 

We answer this question for ambiguity sets constructed from optimal transport (OT) discrepancy around the empirical distribution. Our approach combines structural and algorithmic ingredients. On the structural side, we show that both the primal and dual DRO problems admit small-support worst-case distributions, albeit of different sizes. On the algorithmic side, we exploit this structure to design oracle-based first-order methods that avoid repeatedly solving large lifted reformulations. 

\subsection{Summary of Contributions}
The key contributions of the paper are summarized below. 

\begin{enumerate}[label=$\diamond$]
    \item We begin by revisiting the inner maximization in~\eqref{eq:dro} for a fixed primal decision $x$. When the empirical distribution is supported on $N$ points, existing DRO results guarantee an optimal worst-case distribution supported on at most $N+1$ points~\cite{yue2022linear,gao2023distributionally}. We efficiently recover such a distribution when the loss is convex-piecewise concave. Our key insight is that computing the worst-case distribution is equivalent to a scalar budget allocation problem, which allows the worst-case expectation problem to be solved by optimizing over scalar budget variables alongside small structured subproblems, avoiding large-scale conic optimization. Building on this, we propose an efficient algorithm that returns an $\epsilon$-approximate worst-case distribution supported on $N+1$ points using only $O(\poly\log(1/\epsilon))$ cheap oracle calls.

    \item We revisit the primal DRO problem with OT ambiguity sets through the lens of convex-concave saddle-point optimization and develop an oracle-based distributional best-response framework for solving it. At each iteration, the dual player uses an oracle to compute its best response, i.e., a worst-case distribution corresponding to the current primal decision, while the primal player performs a first-order oracle update against that distribution. This perspective works directly with the saddle-point structure of the primal DRO problem and avoids solving large-scale reformulations.
    As a special case, when the proposed budget-allocation-based method and projected gradient descent are used as the dual and primal oracles, respectively, the resulting algorithm computes an $\epsilon$-approximate saddle point in $O(1/\epsilon^2)$ iterations and directly returns both an approximately optimal primal decision and an associated worst-case distribution.

    \item Although the optimal worst-case distribution is supported on at most $N+1$ points for any fixed primal decision $x$, the proposed best-response algorithm computes a running average, meaning it converges to a worst-case distribution whose support size grows with the number of iterations. To address this issue, we turn to the dual formulation. In particular, we extend our finite-dimensional reformulation technique to the dual DRO problem~\eqref{eq:dual:dro}. In this setting, we prove the existence of an optimal least-favorable distribution supported on at most $\min\{N+n+1, KN\}$ points. We show this bound is tight, which reveals two previously unknown facts: (i) when the decision dimension $n$ is moderate, the least-favorable distribution is significantly sparser than the previously known bound $KN$; and (ii) least-favorable distributions with an optimal primal decision are provably denser than worst-case distributions with a fixed primal decision. We also propose a post-processing method that, given the output of the best-response algorithm, efficiently computes a least-favorable distribution supported on at most $N+n+1$~points.

    \item Finally, we conduct extensive numerical experiments and show that the proposed oracle-based algorithms significantly outperform reformulation-based approaches implemented in state-of-the-art solvers. These results demonstrate that exploiting the oracle structure of the inner problem can lead to substantial practical gains over generic reformulation-based methods.
\end{enumerate}

\subsection{Related Works}
\paragraph{Extremal Distributions and Support Bounds.}
Although OT problems can be computationally intractable even when one of the distributions is discrete \cite{tacskesen2023semi,tacskesen2023discrete}, the structure of extremal distributions in worst-case risk evaluation is by now fairly well understood \cite{gao2023distributionally,owhadi2017extreme,wozabal2012framework}. In particular, when the Wasserstein ball is centered at a discrete distribution supported on $N$ atoms, a sequence of works has progressively tightened the support bound for an optimal worst-case distribution: from $N+3$ atoms in \cite[Theorem~3]{wozabal2012framework}, to $N+2$ atoms in \cite[Theorem~2.3]{owhadi2017extreme}, and finally to $N+1$ atoms in \cite[Corollary~1]{gao2023distributionally}. However, an efficient method for computing such an $N+1$ point optimizer has remained unavailable. For convex-piecewise concave losses with $K$ pieces, \cite[Theorem~4.4]{mohajerin2018data} gives a finite convex reformulation that produces a worst-case distribution supported on at most $KN$ points. We close this gap by developing an efficient approach that constructs an optimal worst-case distribution with only $N+1$ support points. We further show that, for the dual DRO problem~\eqref{eq:dual:dro}, there exists an optimal least-favorable distribution supported on at most $\min\{N+n+1, KN\}$ points under the same loss assumption. This structural result strictly improves upon the generic $KN$-point bound established in \cite[Theorem~2]{shafiee2025nash} whenever the primal decision dimension $n$ is moderate.

\paragraph{Algorithms for OT-based DRO.}
Most computational approaches to OT-based DRO solve problem~\eqref{eq:dro} indirectly, that is, by first deriving a single-level reformulation through duality and then applying an algorithm to the resulting optimization problem. This includes specialized first-order methods for particular models such as distributionally robust logistic regression and support vector machines \cite{li2019first,li2020fast}, as well as distributed methods that exploit structure in the reformulated problem when the loss is convex-concave or convex-convex in the decision and uncertainty variables \cite{cherukuri2020cooperative,li2020data}. For general reference distributions, inexact stochastic gradient methods have also been developed based on reformulations of the DRO problem \cite{sinha2018certifying,blanchet2018optimal,shafiee2025nash}. In contrast, our approach does not rely on solving a reformulated single-level problem. Instead, we work directly with the minimax structure and develop a primal-dual method that updates the primal decision against an explicit worst-case distribution. The key ingredient is an efficient worst-case scenario oracle for convex-piecewise concave losses. While oracle-based ideas have also appeared in robust optimization \cite{ben2015oracle}, we use the oracle in a different way. Namely, the oracle helps construct an $N+1$-point worst-case distribution, which is then embedded into a distributional best-response scheme for the primal DRO problem and further extended to the dual DRO problem. 

\paragraph{Algorithms for Robust Optimization.} 
Our work is closely related to oracle-based robust optimization and online convex optimization. \cite{ben2015oracle} showed that robust optimization can be approached through online learning and repeated oracle calls, and \cite{ho2018online,ho2019exploiting} developed first-order frameworks that treat robust optimization as a semi-infinite problem and reduce it to iterative feasibility or separation computations. More recent methods by \cite{postek2024first} and \cite{tu2024max} pursue large-scale robust optimization through perspective or Lagrangian reformulations. Our work shares the same algorithmic objective, but avoids expensive one-shot reformulations and tailors the oracle-based viewpoint to DRO with OT ambiguity sets, where the adversary selects a probability distribution rather than a finite-dimensional uncertainty vector.

\paragraph{Dual DRO Problem.}
While the primal DRO problem~\eqref{eq:dro} is the dominant computational route to a robust decision~\cite{kuhn2025distributionally}, the dual~\eqref{eq:dual:dro} characterizes the least-favorable distribution. The main computational difficulty is that the dual DRO problem maximizes a concave pointwise infimum over an infinite-dimensional ambiguity~set. Nevertheless, least-favorable distributions have been characterized in mean square error estimation and Kalman filtering under both Wasserstein~\cite{nguyen2023bridging,shafieezadeh2018wasserstein} and information-theoretic divergence~\cite{levy2004robust,levy2012robust,zorzi2016robust,zorzi2017robustness} ambiguity sets. We study the dual DRO problem under the same structural assumptions in \cite{mohajerin2018data,shafiee2025nash}, and establish a refined theoretical bound on the minimal support size of least-favorable distributions by carefully analyzing the equilibrium conditions.

\paragraph{First-Order Methods for DRO Problems.}
Finite-dimensional minimax optimization problems can be solved efficiently using projection-based methods \cite{nedic2009subgradient,nemirovski2009robust,xu2023unified}, projection-free algorithms \cite{boroun2023projection,giang2026projection}, or online convex optimization~\cite{orabona2019modern}.
Infinite-dimensional problems over probability distributions have also been studied recently, both for standard minimization \cite{chizat2018global,chizat2022sparse,eftekhari2019sparse,kent2021modified,yu2025deterministic} and for minimax problems \cite{sheriff2025nonlinear,liu2025convergence,xu2024flow,lanzetti2022first,lanzetti2024variational}.
For $f$-divergence ambiguity sets with discrete support, the adversary's decision variable lies on a probability simplex, leading to finite-dimensional formulations that can be solved by primal-dual methods~\cite{namkoong2016stochastic,aigner2023data,qi2021online}. 
For hybrid $f$-divergence-OT ambiguity sets such as the Sinkhorn ambiguity set, stochastic gradient descent and Langevin-based primal-dual methods have been recently developed \cite{wang2025sinkhorn,azizian2023regularization,wang2024regularization,wang2025iterative,vincent2024texttt}.
For OT ambiguity sets, the same minimax viewpoint is substantially more challenging as the adversary optimizes over the infinite-dimensional probability space. Our contribution is to show that, for convex-piecewise concave losses, the best-response framework can be generalized to solve DRO problems efficiently.

\subsection{Notation and Outline}

The set of positive integers up to $n \in \mathbb{N}$ is denoted by $[n]$. We write $\cP(\cZ)$ for the family of Borel probability measures on $\cZ \subseteq \RR^m$. 
If $f$ is proper, convex, and lower semicontinuous, then its recession function $f^{\infty}: \RR^m \rightarrow \RR \cup \{\infty\}$ is defined by $f^{\infty}(z) = \lim_{\alpha\rightarrow \infty} \alpha^{-1}(f(z_0 + \alpha z) - f(z_0))$, where $z_0$ is any point in $\text{dom}(f)$ \cite[Theorem~8.5]{rockafellar1970convex}. The perspective of $f$ is the function $f^\pi: \RR^{m} \times \RR_+ \rightarrow \RR \cup \{\infty\}$ defined by $f^\pi(z, t) = t f(z/t)$ if $t > 0$, and $f^\pi(z, t) = f^{\infty}(z)$ if $t = 0$. One can show that $f^\pi$ is proper, convex, and lower semicontinuous \cite[page~67]{rockafellar1970convex}. Without loss of generality, we use $t f(z/t)$ to denote $f^\pi(z, t)$ even if $t=0$. Moreover, the conjugate of $f$~is~defined~as $f^*(y) = \sup_{z\in \RR^m}\{\langle y,z\rangle-f(z)\}$, which is proper, convex, and lower semicontinuous~\cite[page~104]{rockafellar1970convex}.

The remainder of the paper is organized as follows. In Section~\ref{sec:setup}, we introduce the technical preliminaries and assumptions. In Section~\ref{sec::worst_case}, we focus on a key component of our analysis: the worst-case expectation problem, i.e., the inner maximization problem in \eqref{eq:dro}. We establish its equivalence to the classical budget allocation problem. We leverage this equivalence to (i) characterize the size and structure of the support of the worst-case distribution, and (ii) design an efficient algorithm for computing it. 
In Section~\ref{sec::primal-DRO}, we show how this algorithm can be used as an oracle within a distributional best-response framework to efficiently solve the primal DRO problem~\eqref{eq:dro}. In Section~\ref{sec::dual_dro}, we further demonstrate how the resulting worst-case distribution can be sparsified by solving the equivalent dual DRO problem~\eqref{eq::dual_DRO}. All complexity results are reported using standard big-$O$ and $\widetilde{O}$ notations, where the latter suppresses logarithmic factors for clarity. Explicit constants are provided in the appendix. Finally, we conclude with numerical experiments in Section~\ref{sec::numerics}.


\section{Problem Setup and Assumptions}
\label{sec:setup}

Throughout the paper, we rely on the following notion of optimal transport discrepancy, which also specifies the standing conditions imposed on the transportation cost function.

\begin{definition}
\label{def::cost_OT}
    A transportation cost function is any lower semicontinuous function $c: \cZ \times \cZ \rightarrow \bR_+$ satisfying $c(z, z) = 0$ for all $z \in \cZ$, where $c(z, \hat{z})$ is convex in $z$ for every fixed $\hat{z} \in \cZ$. The optimal transport discrepancy $\OT: \cP(\cZ) \times \cP(\cZ) \rightarrow \RR_+$ associated with $c$ is defined as
    \begin{align*}
        \OT(\bP, \bQ)
        =
        \min_{\gamma \in \Gamma(\bP, \bQ)}
        \EE_{(z, \hat{z}) \sim \gamma} [c(z, \hat{z})],
    \end{align*}
    where $\Gamma(\bP, \bQ) := \{\gamma \in \cP(\cZ \times \cZ) : \gamma(\cdot \times \cZ) = \bP,\, \gamma(\cZ \times \cdot) = \bQ\}$ is the set of all couplings of $\bP$ and~$\bQ$.
\end{definition}
The real-valuedness assumption on $c$ is made only to simplify the exposition. It can be relaxed to extended-valued costs, provided each data point $\hat z_i$ satisfy the relative-interior and Slater-type conditions required for the convex reformulation results. 

Suppose we are given a dataset $\{\hat{z}_i\}_{i = 1}^{N} \subseteq \cZ$. Let $\delta_{\hat{z}_i}$ denote the Dirac measure at $\hat{z}_i$, and let 
\begin{align*}
    \hat{\bP} = \frac{1}{N} \sum_{i = 1}^{N} \delta_{\hat{z}_i}
\end{align*}
be the empirical distribution.
Given a radius $\rho > 0$, we define the OT ambiguity set around $\hat{\bP}$ as
\begin{align}
\label{eq::ambiguity_set}
    \cP
    :=
    \left\{
    \bP \in \cP(\cZ):
    \OT(\bP, \hat{\bP}) \leq \rho
    \right\}.
\end{align}

We impose the following standing assumptions.

\begin{assumption}[Regularity]
\label{asp::regularity}
The following conditions hold.
\begin{enumerate}[label=(\roman*)]
    \item \label{asp::regular::sets}
    The feasible region $\cX \subseteq \RR^n$ is nonempty and convex. There exist $\bQ_0 \in \cP$ and $x_0 \in \cX$ such that $\EE_{z\sim \bQ_0}[\ell(x, z)]$ is inf-compact in $x \in \cX$, and $\EE_{z\sim \bQ_0}[\ell(x_0, z)] > -\infty$. The support set $\cZ \subseteq \RR^m$ is nonempty, closed and convex. 

    \item \label{asp::regular::ell}
    The loss function takes the form $\ell(x, z) := \max_{k\in [K]} \ell_k(x, z)$, where for each fixed $k\in [K]$, the function $\ell_k: \bR^n \times \bR^m \to \bR$ is real-valued, convex in its first argument, and concave in its second. For notational convenience, we assume $K\geq2$; the case $K=1$ is covered by setting $\ell_2:=\ell_1$.

    \item \label{asp::regular::c}
    There exist a norm $\| \cdot \|$ on $\bR^m$ and an exponent $p \geq 1$ such that
    \begin{align*}
        c(z, \hat{z})
        \geq
        \| z - \hat{z} \|^p
        \qquad
        \forall z, \hat{z} \in \cZ.
    \end{align*}
    Moreover, if $p = 1$, then for every $x \in \cX$, there exist a constant $g > 0$, a reference point $\hat{z}_0 \in \cZ$, and an exponent $r \in [0, 1)$ such that $\ell(x, z) \leq g \left[ 1 + \| z - \hat{z}_0 \|^r \right]$.
\end{enumerate}
\end{assumption}

Assumption~\ref{asp::regularity}\ref{asp::regular::sets} imposes mild topological conditions on the decision and support sets. The inf-compactness requirement ensures that the sublevel sets $\{x \in \cX : \EE_{z\sim \bQ_0}[\ell(x, z)] \leq c\}$ are compact for all $c \in \RR$. This serves as a practical relaxation of requiring the entire feasible region $\cX$ to be bounded. Furthermore, this inf-compactness condition is readily satisfied whenever the expected loss is lower semicontinuous and coercive on $\cX$. That is, for any sequence $\{x_k\}_{k \in \NN} \subset \cX$ such that $\lim_{k\to \infty} \norm{x_k} = \infty$, we have $\lim_{k\to \infty} \EE_{z\sim \bQ}[\ell(x_k, z)] = \infty$.
Assumption~\ref{asp::regularity}\ref{asp::regular::ell} is the convex-piecewise concave structure introduced in~\cite[Assumption~4.1]{mohajerin2018data} for tractable DRO reformulations. Assumption~\ref{asp::regularity}\ref{asp::regular::c} guarantees well-posedness and attainment. Since each loss piece is real-valued and concave in $z$, it admits an affine upper bound in $z$. Hence, the loss grows at most linearly in the uncertainty. Therefore, when $p > 1$, this assumption implies that the transportation cost dominates the loss growth, thereby guaranteeing the optimal solution is well-posed and attained. The case $p = 1$ is more delicate because linear loss growth competes directly with linear transportation cost. In this case, the sublinear growth condition in Assumption~\ref{asp::regularity}\ref{asp::regular::c} ensures well-posedness and attainment. We note that, when $\cZ$ is compact, Assumption~\ref{asp::regularity}\ref{asp::regular::c} is automatically satisfied for any $r\in [0,1)$ and sufficiently large $g$. In Appendix~\ref{sec::piecewise-linear}, we provide a separate treatment of the case where both the transportation cost and loss function grow linearly, thereby violating Assumption~\ref{asp::regularity}\ref{asp::regular::c}.

It is worth noting that the choice $c(z,\hat{z}) = \|z - \hat{z}\|^p$, for some norm $\|\cdot\|$ on $\mathbb{R}^m$ and exponent $p \geq 1$, satisfies Assumption~\ref{asp::regularity}\ref{asp::regular::c} with equality. In this case, $\left(\OT(\bP, \bQ)\right)^{1/p}$ defines the so-called \emph{$p$-Wasserstein distance}, one of the most widely studied optimal transport discrepancies in the DRO literature. Beyond this canonical choice, Assumption~\ref{asp::regularity}\ref{asp::regular::c} also accommodates Mahalanobis-type costs $c(z,\hat{z}) = (z-\hat{z})^\top M (z-\hat{z})$ for a positive definite matrix $M$, which allow the geometry of
the ambiguity set to reflect the scale and correlation structure of the uncertainty \cite{blanchet2018optimal}.

Together, these assumptions guarantee that the primal and dual DRO problems are well defined and admit a saddle point. 

\begin{lemma}[Existence of Saddle Point]
\label{lem::existence_saddle}
    Under Assumption~\ref{asp::regularity}, the primal DRO problem~\eqref{eq:dro} and the dual DRO problem~\eqref{eq:dual:dro} have finite optimal values and admit a saddle point. In particular,
    \begin{align*}
        \min_{x \in \cX}
        \max_{\bQ \in \cP}
        ~ \EE_{z \sim \bQ}[\ell(x, z)]
        =
        \max_{\bQ \in \cP}
        \min_{x \in \cX}
        ~ \EE_{z \sim \bQ}[\ell(x, z)].
    \end{align*}
\end{lemma}

The proof follows by verifying the conditions of~\cite[Lemmas~3 \& 4]{shafiee2025nash}, which can be checked straightforwardly. For completeness, we provide a self-contained proof in Appendix~\ref{proof::lem::existence_saddle}.

Beyond well-posedness, we exploit Assumption~\ref{asp::regularity} to design efficient algorithms for the inner worst-case expectation problem in the primal DRO formulation. We note that this assumption places us in the same structural regime that yields finite-dimensional reformulations for DRO problems.


\section{Worst-case Expectation Problem}
\label{sec::worst_case}

In this section, we develop an efficient algorithm for solving the inner maximization problem in the primal DRO formulation~\eqref{eq:dro} over the OT ambiguity set~\eqref{eq::ambiguity_set}. We fix an arbitrary decision $x \in \cX$ and suppress its dependence in the notation, writing $\ell(z)$ in place of $\ell(x, z)$. The resulting worst-case expectation problem is therefore
\begin{align}
\label{eq::worst_case}
    \max_{\bQ \in \cP} ~
    \mathbb{E}_{z \sim \bQ} \left[\ell(z) \right]
\end{align}
Under Assumption~\ref{asp::regularity}, problem~\eqref{eq::worst_case} admits a finite convex reformulation based on perspective functions. This reformulation was first derived for the special case $c(z, \hat z) = \| z - \hat z \|$ in \cite[Theorem~4.4]{mohajerin2018data} and later extended to general convex transportation cost functions in \cite[Proposition~20]{zhen2023unified}. While this perspective reformulation is an important step toward tractability, it relies on generic off-the-shelf solvers that do not exploit the problem's structure.

To uncover and leverage this structure, we instead work with the following nonconvex reformulation. Such reformulations serve as a natural intermediate step in deriving perspective reformulations via the primal-worst dual-best principle \cite{beck2009duality,zhen2023unified}, and as we show, they also reveal the geometry needed to solve \eqref{eq::worst_case} more efficiently.

\begin{proposition}
\label{prop::nonconvex-equivalence}
    Under Assumption~\ref{asp::regularity}, the worst-case expectation problem~\eqref{eq::worst_case} is equivalent to the nonconvex program
    \begin{equation}
    \label{eq::equivalence_primal}
    \max
    \left\{
    \frac{1}{N} \sum_{i=1}^N \sum_{k=1}^K \alpha_{ik} \ell_k \left(\hat{z}_i + 
    v_ {ik} \right) :
    \begin{array}{l}
        \alpha_{ik} \in \bR_+, \ v_{ik} \in \bR^m, \ \hat{z}_i + v_{ik} \in \cZ, \ \forall i \in [N], \forall k \in [K] \\[1ex]
        \displaystyle \frac{1}{N} \sum_{i=1}^N \sum_{k=1}^K \alpha_{ik} c(\hat{z}_i + v_{ik}, \hat{z}_i) \leq \rho, \ \sum_{k=1}^K \alpha_{ik} = 1, \, \forall i\in [N]
    \end{array}
    \right\}.
    \end{equation}
    Given an optimal solution
    $\{\alpha_{ik}^\star, v_{ik}^\star\}$ to~\eqref{eq::equivalence_primal}, the discrete distribution
\begin{align}\label{eq::worst-case-Q}
    \QQ^\star := \frac{1}{N} \sum_{i=1}^N \sum_{k=1}^K \alpha_{ik}^\star \delta_{\hat{z}_i + v_{ik}^\star},
\end{align}
    belongs to the ambiguity set $\cP$ and attains the maximum in \eqref{eq::worst_case}.
    Furthermore, there exists an optimal solution in which at most $N+1$ of the weights $\{\alpha_{ik}^\star\}$ are nonzero. Consequently, there exists a worst-case distribution $\bQ^\star$ supported on at most $N + 1$ points.
\end{proposition}

\begin{proof}
Consider any coupling $\gamma \in \Gamma(\bQ, \hat{\bP})$. Since the empirical distribution is $\hat{\bP} = \frac{1}{N} \sum_{i=1}^N \delta_{\hat{z}_i}$, we can disintegrate $\gamma$ into conditional probability distributions $\bQ_i \in \cP(\cZ)$ associated with each sample~$\hat{z}_i$, such that $\bQ = \frac{1}{N} \sum_{i=1}^N \bQ_i$. The worst-case expectation problem \eqref{eq::worst_case} can then be written as maximizing $\frac{1}{N} \sum_{i=1}^N \mathbb{E}_{z \sim \bQ_i}[\ell(z)]$ subject to the budget constraint $\frac{1}{N} \sum_{i=1}^N \mathbb{E}_{z \sim \bQ_i}[c(z, \hat{z}_i)] \leq \rho$.

By Assumption~\ref{asp::regularity}\ref{asp::regular::ell}, the loss is $\ell(z) = \max_{k \in [K]} \ell_k(z)$. For each $i \in [N]$, we can partition the domain $\cZ$ into $K$ disjoint regions $\mathcal{R}_1, \dots, \mathcal{R}_K$ such that $\ell(z) = \ell_k(z)$ for $z \in \mathcal{R}_k$. Let $\alpha_{ik} = \bQ_i(\mathcal{R}_k)$ be the probability mass of region $k$, and let $\bQ_{ik}$ be the conditional distribution of $\bQ_i$ restricted to~$\mathcal{R}_k$ when $\alpha_{ik} > 0$ and be any arbitrary distribution in $\cP(\cZ)$ when $\alpha_{ik} = 0$. By the law of total expectation, the objective and cost contributions for sample $i$ are given by:
\begin{align*}
    \textstyle \mathbb{E}_{z \sim \bQ_i}[\ell(z)] = \sum_{k=1}^K \alpha_{ik} \mathbb{E}_{z \sim \bQ_{ik}}[\ell_k(z)], \qquad
    \mathbb{E}_{z \sim \bQ_i}[c(z, \hat{z}_i)] = \sum_{k=1}^K \alpha_{ik} \mathbb{E}_{z \sim \bQ_{ik}}[c(z, \hat{z}_i)].
\end{align*}
Define the conditional mean $\tilde{z}_{ik} = \mathbb{E}_{z \sim \bQ_{ik}}[z]$. Because $\ell_k$ is concave and $c(\cdot, \hat{z}_i)$ is convex (Assumption~\ref{asp::regularity}\ref{asp::regular::ell} and Definition~\ref{def::cost_OT}), Jensen's inequality implies:
\begin{align*}
    \mathbb{E}_{z \sim \bQ_{ik}}[\ell_k(z)] \leq \ell_k(\tilde{z}_{ik}), \qquad
    \mathbb{E}_{z \sim \bQ_{ik}}[c(z, \hat{z}_i)] \geq c(\tilde{z}_{ik}, \hat{z}_i).
\end{align*}
Therefore, replacing each arbitrary conditional distribution $\bQ_{ik}$ with a Dirac measure $\delta_{\tilde{z}_{ik}}$ placed at its mean can only increase the objective and decrease the transportation cost. Thus, restricting $\bQ_{ik}$ to $\delta_{\tilde{z}_{ik}}$ and optimizing over $\tilde{z}_{ik}$ do not change the problem's optimal value. Applying the change of variables $v_{ik} = \tilde{z}_{ik} - \hat{z}_i$ safely restricts the search space to finite-dimensional perturbations, directly yielding \eqref{eq::equivalence_primal} and the worst-case distribution \eqref{eq::worst-case-Q}.

To establish solvability and the support bound, note that under Assumption~\ref{asp::regularity}\,\ref{asp::regular::c}, the transportation cost strictly dominates the growth of the loss function. This coercivity ensures that no sequence of valid perturbations $v_{ik}$ can indefinitely increase the objective without violating the bounded transport budget $\rho$. Since the probability weights $\alpha_{ik}$ reside in a compact simplex and the effective perturbations $v_{ik}$ are bounded, the joint maximization over $(\alpha, v)$ in \eqref{eq::equivalence_primal} is attainable with a finite optimum. 
Recall also that joint maximization is equivalent to sequential maximization. As the problem is solvable over $v$, there exists an optimal $v^\star$. Evaluating the inner maximization at $v = v^\star$, the problem over $\alpha$ reduces to the following linear program (LP)
\begin{align*}
    \max_{\alpha \geq 0} \left\{ \frac{1}{N} \sum_{i=1}^N \sum_{k=1}^K \alpha_{ik} \ell_k(\hat{z}_i + v_{ik}^\star) : \sum_{k=1}^K \alpha_{ik} = 1 \; \forall i \in [N], \ \frac{1}{N} \sum_{i=1}^N \sum_{k=1}^K \alpha_{ik} c(\hat{z}_i + v_{ik}^\star, \hat{z}_i) \leq \rho \right\},
\end{align*}
which has $NK$ nonnegative variables and exactly $N+1$ constraints ($N$ normalization constraints and $1$ total transport budget constraint). By the fundamental theorem of linear programming~\cite[Theorem~2.7]{bertsimas1997introduction}, the optimal value is attained at a basic feasible solution (BFS) $\alpha^\star$ possessing at most $N+1$ strictly positive variables. Thus, the resulting worst-case distribution $\bQ^\star$ is supported on at most $N+1$ points.
\end{proof}

The above proposition shows that the worst-case distribution $\QQ^\star$ is composed of discrete atoms, where each atom has mass $\alpha_{ik}^\star$ and is located at the shifted empirical sample $\hat{z}_i + v_{ik}^\star$. Moreover, the transportation budget constraint in~\eqref{eq::equivalence_primal} is the only constraint coupling the decisions associated with different empirical samples.
This observation naturally suggests a two-level decomposition of the problem. At the upper level, we allocate a local budget $b_i$ to the transportation cost associated with each empirical sample $\hat z_i$. At the lower level, for a fixed value of $b_i$, we compute the best \emph{utility value} $V_i(b_i)$ that can be attained from $\hat z_i$ using only its assigned local budget.

The next theorem formalizes this decomposition and shows that the resulting upper-level problem reduces to a classical budget allocation problem over $N$ scalar variables.

\begin{theorem}
\label{thm::decomposition}
    Under Assumption~\ref{asp::regularity}, the worst-case expectation problem \eqref{eq::worst_case} is equivalent to
    \begin{equation}
    \label{eq::master}
        \max \left\{ \frac{1}{N} \sum_{i=1}^{N} V_i(b_i) : b \in \bR_+^N, \ \frac{1}{N} \sum_{i=1}^N b_i \leq \rho \right\}.
    \end{equation}
    Here, for every $i\in [N]$ and $b_i \geq 0$, the local utility $V_i(b_i)$ is defined by
    \begin{align}
    \label{eq::V_i}
        V_i(b_i) = \max_{1\leq k_1< k_2\leq K} V_i^{(k_1, k_2)}(b_i),
    \end{align}
    where, for $1 \leq k_1 < k_2 \leq K$, the pairwise utility function $V_i^{(k_1, k_2)}(b_i)$ is defined as
    \begin{align}
    \label{eq::U_i^k}
        V_i^{(k_1, k_2)}(b_i) 
        = \max
        \left\{
        \alpha_{1} \ell_{k_1} \left(\hat{z}_i + v_1\right) + \alpha_{2} \ell_{k_2} \left(\hat{z}_i + v_2\right)
        :
        \begin{array}{l}
            \alpha_{j}, \beta_j \in \bR_+, v_{j} \in \bR^m, \hat{z}_i + v_j \in \cZ, \\[2ex]
            \displaystyle \alpha_j c\left(\hat{z}_i + v_j, \hat{z}_i\right) \leq \beta_j, ~ \forall j \in [2] \\[2ex]
            \alpha_{1} + \alpha_{2} = 1,\ \beta_{1} + \beta_{2} = b_i
        \end{array}
        \right\}. 
    \end{align}
\end{theorem}

\begin{proof}
By Proposition~\ref{prop::nonconvex-equivalence}, the worst-case expectation problem \eqref{eq::worst_case} is equivalent to the finite-dimensional nonconvex program \eqref{eq::equivalence_primal}. The objective function and the constraints of this primal program are completely separable across the empirical samples $i \in [N]$, coupled only by the total transport budget constraint. To decouple this problem, we introduce local budget variables $b_i \geq 0$ representing the transportation cost allocated to each empirical sample $\hat{z}_i$. We define this allocation as:
\begin{align*}
    \textstyle b_i = \sum_{k=1}^K \alpha_{ik} c(\hat{z}_i + v_{ik}, \hat{z}_i), \quad \forall i \in [N].
\end{align*}
The total budget constraint becomes the average of these local allocations: $\frac{1}{N} \sum_{i=1}^N b_i \leq \rho$. We can then decompose \eqref{eq::equivalence_primal} into two levels. At the upper level, we optimize the budget allocation $\{b_i\}_{i=1}^N$ by solving~\eqref{eq::master}, where $V_i(b_i)$ denotes the optimal value of the lower-level problem associated with empirical sample $\hat z_i$ under a fixed local budget $b_i$:
\begin{align}
\label{eq::full_V_i}
V_i(b_i) = \max
\left\{ \sum_{k=1}^K \alpha_{k} \ell_k(\hat{z}_i + v_{k}) :
\begin{array}{l}
\alpha_{k} \in \bR_+, \ v_{k} \in \bR^m, \ \hat{z}_i + v_{k} \in \cZ, \ \forall k \in [K] \\[1ex]
\displaystyle \sum_{k=1}^K \alpha_{k} = 1, \ \sum_{k=1}^K \alpha_{k} c(\hat{z}_i + v_{k}, \hat{z}_i) \leq b_i
\end{array}
\right\}.
\end{align}
We next show that the lower-level problem can be further decomposed into $O(K^2)$ subproblems, each involving only two components. Fixing the perturbations at their optimal values $\{v_k^\star\}_{k=1}^K$ reduces problem~\eqref{eq::full_V_i} to an LP over the probability weights $\{\alpha_k\}_{k=1}^K$. This LP has $K$ nonnegative variables and exactly two constraints: the normalization constraint $\sum_{k=1}^K \alpha_k = 1$ and the local budget constraint $\sum_{k=1}^K \alpha_k c(\hat z_i + v_k^\star,\hat z_i)\leq b_i$.
Therefore, by the same argument used in the proof of Proposition~\ref{prop::nonconvex-equivalence}, there exists an optimal BFS with at most two strictly positive weights $\alpha_k$. Consequently, for any fixed sample $i$, at most two components of the loss function receive positive mass in an optimal solution.
This structural property allows us to evaluate $V_i(b_i)$ as follows: (1) enumerate all pairs of components $1\leq k_1 < k_2 \leq K$; (2) for each pair, compute the corresponding pairwise utility function $V_i^{(k_1,k_2)}(b_i)$ by solving~\eqref{eq::U_i^k}; and (3) select the pair attaining the largest objective value.
\end{proof}

Theorem~\ref{thm::decomposition} should be interpreted as a value-function reformulation of the worst-case expectation problem. It replaces the original distributional optimization problem with the budget allocation problem~\eqref{eq::master}, whose objective is expressed in terms of the local utility functions $V_i$. The theorem does not, by itself, prescribe how these utility functions should be evaluated, how the upper-level problem \eqref{eq::master} should be solved, or how one should extract a worst-case distribution from its optimal solution. These constructive questions are addressed in the remainder of this section.

We proceed in two steps. In Section~\ref{subsec::efficient_budget_allocation}, we present an efficient algorithm for solving the budget allocation problem~\eqref{eq::master}. Given an optimal budget allocation, the same local optimizers used to evaluate the utilities can also be used to construct a worst-case distribution. This direct construction is simple but not necessarily sparsest. Specifically, according to \eqref{eq::V_i} and \eqref{eq::U_i^k}, each local optimizer may use two active components, contributing two atoms to the worst-case distribution. Therefore, the final constructed worst-case distribution can be supported on as many as $2N$ points.
On the other hand, Proposition~\ref{prop::nonconvex-equivalence} guarantees the existence of an optimal worst-case distribution supported on at most $N+1$ points. Reducing the support size from the direct $2N$-point construction to an $(N+1)$-point construction requires an additional argument, which is developed in Section~\ref{subsec::support_compression}.

\subsection{Efficient Budget Allocation Algorithm and 2N-Point Construction}
\label{subsec::efficient_budget_allocation}

We first explain how the upper-level problem~\eqref{eq::master} can be solved efficiently. 
For a multiplier $\lambda \geq 0$, the corresponding Lagrangian of~\eqref{eq::master}, up to the constant term $\lambda \rho$, is given by
\begin{align*}
    L_\lambda(b)
    = \frac{1}{N} \sum_{i=1}^{N} V_i(b_i) - \lambda b_i.
\end{align*}
Thus, for any fixed value of $\lambda$, maximizing the Lagrangian separates across the empirical samples. This separability is the main algorithmic benefit of the budget allocation reformulation.

This Lagrangian viewpoint is justified by Everett's epsilon theorem~\cite[Theorem~3]{everett1963generalized}. In the present notation, the theorem implies that if, for some $\epsilon>0$ and $\lambda\geq 0$, an allocation $\bar b\in\bR_+^N$ satisfies $L_\lambda(\bar b) \geq \max_{b \in \bR_+^N} L_\lambda(b) - \epsilon,$ then $\bar b$ is an $\epsilon$-optimal solution of the constrained problem with budget level equal to its own resource usage. More precisely, letting $\bar\rho = \frac{1}{N}\sum_{i=1}^{N} \bar b_i$, we have
\begin{align}\label{eq::dual-opt}
    \textstyle \frac{1}{N}\sum_{i=1}^{N}V_i(\bar b_i)
    \geq \max \left\{
        \frac{1}{N}\sum_{i=1}^{N}V_i(b_i)
        : b\in\bR_+^N, \ \frac{1}{N} \sum_{i=1}^{N}b_i \leq \bar \rho
    \right\}
    - \epsilon .
\end{align}
Consequently, if one can find a multiplier $\lambda^\star$ whose Lagrangian maximizer uses the target budget, $\frac{1}{N}\sum_{i=1}^{N} b_i^\star=\rho,$ then $b^\star$ solves the upper-level problem~\eqref{eq::master}, up to the accuracy with which the Lagrangian subproblems are solved.

Equivalently, in the nondegenerate case where the budget constraint is active, the KKT conditions~\cite[Theorem~31.3]{rockafellar1970convex} reduce the computation to finding a multiplier $\lambda^\star\geq 0$ such that
\begin{align*}
    \textstyle b_i^\star
    \in
    \argmax_{b_i\geq 0}
    \left\{
        V_i(b_i)-\lambda^\star b_i
    \right\}, \ \forall i \in [N] \quad \text{and} \quad
    \frac{1}{N}\sum_{i=1}^{N} b_i^\star=\rho.
\end{align*}
This leads to a natural top-down interpretation of our algorithm. In the outer loop, we search for the multiplier $\lambda^\star$ by bisection. For each fixed value of $\lambda$, the Lagrangian subproblem separates across $i$, and we compute $b_i(\lambda) \in \argmax_{b_i\geq 0} \{ V_i(b_i)-\lambda b_i \}$ by a one-dimensional golden-section search over the local budget $b_i$. However, each evaluation of the objective $V_i(b_i)-\lambda b_i$ requires the value of $V_i(b_i)$. Computing $V_i(b_i)$ requires maximizing over all pairs $(k_1,k_2)$, and evaluating each pairwise value $V_i^{(k_1, k_2)}(b_i)$ requires a nested golden-section search over the local weight and budget allocations.

The implementation is therefore most naturally presented from the bottom up. We first describe how to evaluate the pairwise utility $V_i^{(k_1, k_2)}(b_i)$. We then use this pairwise evaluator to compute the full local utility $V_i(b_i)$. Finally, these local evaluations are used inside the outer budget allocation algorithm for solving~\eqref{eq::master}.

Our algorithms use the following local worst-case oracle as a primitive.

\begin{assumption}[Local Worst-Case Oracle]
\label{asp::inner_oracle}
    For any $k\in[K]$, nominal point $\hat z\in\cZ$, radius $u \geq 0$, and accuracy level $\epsilon > 0$, there exists an algorithm, denoted by
    \begin{align*}
        (\hat L, \hat v)
        \gets
        \mathsf{WCO}_{k}^{\epsilon}(\hat z,u),
    \end{align*}
    that runs in time $\mathsf{Cost}_{k,\epsilon}$ and returns a feasible perturbation $\hat v\in\bR^m$ satisfying $c(\hat z+\hat v,\hat z)\leq u$ and $\hat z+\hat v\in\cZ,$ together with its attained loss value $\hat L=\ell_k(\hat z+\hat v)$ satisfying
    \begin{align*}
        \hat L 
        \geq \max
        \left\{
            \ell_k(\hat z+v)
            :
            v \in \bR^m, \ 
            \hat z+v\in\cZ, \
            c(\hat z+v,\hat z) \leq u
        \right\}
        - \epsilon.
    \end{align*}
\end{assumption}

Assumption~\ref{asp::inner_oracle} is the computational primitive used by the algorithms below, consistent with oracle-based approaches in robust optimization~\cite{ben2015oracle}. The oracle complexity $\mathsf{Cost}_{k,\epsilon}$ depends on the structure of $\ell_k$: interior-point methods yield $O(m^\varrho \log(1/\epsilon))$ for conic-representable losses with $\varrho>0$ reflecting problem complexity \cite{nesterov1994interior}; projected subgradient methods give $O(m/\epsilon^2)$ for concave and Lipschitz losses~\cite[Theorem~3.2]{bubeck2015convex}; and projected gradient methods achieve $O(m/\epsilon)$ for concave and smooth losses~\cite[Theorem~3.7]{bubeck2015convex}, with faster rates under additional curvature assumptions~\cite[Theorems~3.9--3.10]{bubeck2015convex}. Furthermore, Appendix~\ref{sec::efficient_oracles} discusses several structured \emph{prox-friendly} loss classes that admit even more efficient tailored oracle implementations.

We now describe the utility-evaluation algorithm. The pairwise utility $V_i^{(k_1, k_2)}(b_i)$ involves two scalar allocation decisions: the mass split $\alpha_1+\alpha_2=1$ and the budget split $\beta_1+\beta_2=b_i$. Algorithm~\ref{alg::evaluate_pairwise} performs a golden-section search over the mass split. For each candidate value of $\alpha_1$, Algorithm~\ref{alg::solve_inner} performs a second golden-section search over the budget split. Each objective evaluation inside the inner search is computed using the local worst-case oracle in Assumption~\ref{asp::inner_oracle}. 
The oracle is called with an effective radius $\beta/\alpha$, 
obtained by normalizing the component budget by its mass. Specifically, for a component with mass $\alpha>0$ and budget $\beta\geq 0$, the constraint $\alpha c(\hat z_i+v,\hat z_i)\leq \beta$ is equivalent to $c(\hat z_i+v,\hat z_i)\leq \beta/\alpha$. Thus, Algorithm~\ref{alg::solve_inner} calls $(\hat L, \hat v) \gets \mathsf{WCO}_{k}^{\epsilon} (\hat z_i,{\beta}/{\alpha})$ and uses the weighted value $\alpha \hat L$. This normalization is well defined in Algorithm~\ref{alg::solve_inner}, because Algorithm~\ref{alg::evaluate_pairwise} only passes interior golden-section points $\alpha_1 \in (0,1)$; hence both $\alpha_1$ and $\alpha_2=1-\alpha_1$ are strictly positive. In the following, we use a common golden-section tolerance $\eta$ and oracle precision $\epsilon$ throughout.

\begin{figure*}[!t]
\begin{minipage}[t]{0.495\textwidth}
\begin{algorithm}[H]
    \small
    \caption{\texttt{PairEval}}
    \label{alg::evaluate_pairwise}
    \begin{algorithmic}
        \REQUIRE Sample $i$, pieces $k_1<k_2$, budget $b_i$, 
        \STATE \hspace{2.3em} tolerance $\eta>0$, oracle accuracy $\epsilon>0$ \\[1ex]
        \hspace{-1em}\textbf{Initialize:} $\phi \gets \frac{\sqrt{5}-1}{2}$ \\[0.5ex]
        \hspace{3.8em} $L_\alpha \gets 0$, $U_\alpha\gets 1$ \\[0.5ex]
        \hspace{3.8em} $\alpha^{(1)} \gets U_\alpha-\phi(U_\alpha-L_\alpha)$ \\[0.5ex]
        \hspace{3.8em} $\alpha^{(2)} \gets L_\alpha+\phi(U_\alpha-L_\alpha)$ \\[1ex]

        \WHILE{$U_\alpha-L_\alpha>\eta$} \vspace{1ex}
            \STATE $(V^{(1)},s^{(1)})\gets
            \texttt{InnerSolver}(i,k_1,k_2,b_i,\alpha^{(1)},\eta,\epsilon)$ \\[1ex]
            \STATE $(V^{(2)},s^{(2)})\gets
            \texttt{InnerSolver}(i,k_1,k_2,b_i,\alpha^{(2)},\eta,\epsilon)$ \\[1ex]

            \IF{$V^{(1)}<V^{(2)}$}\vspace{0.5ex}
                \STATE $L_\alpha\gets \alpha^{(1)}$ \\[0.5ex]
                \STATE $\alpha^{(1)}\gets \alpha^{(2)}$ \\[0.5ex]
                \STATE $\alpha^{(2)}\gets L_\alpha+\phi(U_\alpha-L_\alpha)$ \\[0.5ex]
            \ELSE
                \STATE $U_\alpha\gets \alpha^{(2)}$ \\[0.5ex]
                \STATE $\alpha^{(2)}\gets \alpha^{(1)}$ \\[0.5ex]
                \STATE $\alpha^{(1)}\gets U_\alpha-\phi(U_\alpha-L_\alpha)$ \\[0.5ex]
            \ENDIF \vspace{0.5ex}
        \ENDWHILE
        \vspace{1ex}

        $j^\star \in \argmax_{j \in [2]} V^{(j)}$ \\[0.8ex] 
        
        \STATE 
        $\hat V_i^{(k_1, k_2)}(b_i) \gets V^{(j^\star)}, \ \hat s \gets
        s^{(j^\star)}$ \\[1ex]

        \ENSURE $\hat V_i^{(k_1, k_2)}(b_i) \approx V_i^{(k_1, k_2)}(b_i)$ \& optimizer $\hat s$
    \end{algorithmic}
\end{algorithm}
\end{minipage}
\hfill
\begin{minipage}[t]{0.49\textwidth}
\begin{algorithm}[H]
    \small
    \caption{\texttt{InnerSolver}}
    \label{alg::solve_inner}
    \begin{algorithmic}
        \REQUIRE Sample $i$, pieces $k_1,k_2$, budget $b_i$, weight $\alpha_1 \in (0,1)$, tolerance $\eta > 0$, oracle accuracy $\epsilon > 0$ \\[1ex]
        \hspace{-1em}\textbf{Initialize:} $\phi \gets \frac{\sqrt{5}-1}{2}$, $\alpha_2\gets 1-\alpha_1$ \\[0.5ex]
        \hspace{3.8em} $L_\beta\gets 0$, $U_\beta\gets b_i$ \\[0.5ex]
        \hspace{3.8em} $\beta^{(1)}\gets U_\beta-\phi(U_\beta-L_\beta)$ \\[0.5ex]
        \hspace{3.8em} $\beta^{(2)}\gets L_\beta+\phi(U_\beta-L_\beta)$ \\[1ex]

        \WHILE{$U_\beta-L_\beta>\eta$} \vspace{0.5ex}
            \STATE $\big(L_1^{(j)}, v_1^{(j)}\big) \gets
            \mathsf{WCO}_{k_1}^{\epsilon}
            \left(\hat z_i, \frac{\beta^{(j)}}{\alpha_1}\right), \hfill \forall j \in [2]$ \\[0.5ex]
            
            \STATE $\big(L_2^{(j)}, v_2^{(j)}\big) \gets
            \mathsf{WCO}_{k_2}^{\epsilon}
            \left(\hat z_i,\frac{b_i-\beta^{(j)}}{\alpha_2} \right), \hfill \forall j \in [2]$ \\[0.5ex]

            \STATE $V^{(j)}\gets
            \alpha_1 L_1^{(j)}+\alpha_2 L_2^{(j)}, \hfill \forall j \in [2]$ \\[1ex]

            \IF{$V^{(1)} < V^{(2)}$} \vspace{0.5ex}
                \STATE $L_\beta \gets \beta^{(1)}$ \\[0.5ex]
                \STATE $\beta^{(1)}\gets \beta^{(2)}, \ \beta^{(2)}\gets L_\beta+\phi(U_\beta-L_\beta)$ \\[0.5ex]
            \ELSE \vspace{0.5ex}
                \STATE $U_\beta\gets \beta^{(2)}$ \\[0.5ex]
                \STATE $\beta^{(2)}\gets \beta^{(1)}, \ \beta^{(1)}\gets U_\beta-\phi(U_\beta-L_\beta)$ \\[0.5ex]
            \ENDIF \vspace{0.5ex}
        \ENDWHILE
        \vspace{1ex}

        \STATE 
        $j^\star \in \argmax_{j \in [2]} V^{(j)}$ \\[1ex] 
        
        \STATE $\hat V \gets V^{(j^\star)}, \ \hat s \gets
        \big(\alpha_1, \alpha_2, \beta^{(j^\star)}, b - \beta^{(j^\star)}, v_1^{(j^\star)}, v_2^{(j^\star)}\big)$ \\[0.9ex] 

        \ENSURE Value $\hat V$ \& optimizer $\hat s$
    \end{algorithmic}
\end{algorithm}
\end{minipage}
\end{figure*}

\begin{lemma}
\label{lem::eval_approx}
    Fix $i\in[N]$, $b_i \in \RR_+$, and $1\leq k_1<k_2\leq K$. Suppose Assumptions~\ref{asp::regularity} and~\ref{asp::inner_oracle} hold. Then Algorithm~\ref{alg::evaluate_pairwise} runs with oracle precision $\epsilon>0$ and golden-section tolerance $\eta > 0$, and returns a feasible local solution $\hat s$ and a value $\hat V_i^{(k_1, k_2)}(b_i)$ satisfying
    \begin{align*}
        \left| V_i^{(k_1, k_2)}(b_i) - \hat V_i^{(k_1, k_2)}(b_i) \right|
        \leq 
        \begin{cases}
        O\left(\epsilon + \eta^{\min\{1/p,\,1-1/p\}} \right) &\mathrm{if}\ p>1 \\
        O\left(\epsilon + \eta^{1-r} \right) &\mathrm{if}\ p=1
        \end{cases}
    \end{align*}
    with the running time complexity of 
    \begin{align*}
        \widetilde O \left( \left( \mathsf{Cost}_{k_1,\epsilon} + \mathsf{Cost}_{k_2,\epsilon} \right) \log^2 ({1}/{\eta}) \right).
    \end{align*}
\end{lemma}
The proof of Lemma~\ref{lem::eval_approx} is presented in Appendix~\ref{proof::lemma2and3}. Lemma~\ref{lem::eval_approx} separates the two numerical error sources. The precision $\epsilon$ controls the accuracy of each call to the local worst-case oracle, while the tolerance $\eta$ controls the accuracy of the two golden-section searches. Thus, to obtain a $\delta$-accurate evaluation of $V_i^{(k_1, k_2)}(b_i)$, it suffices to choose $\epsilon$ and~$\eta$ so that $\epsilon+\eta = O(\delta)$. For example, one may take $\epsilon=O(\delta)$ with $\eta=O(\delta^{\max\{p,\, p/(p-1)\}})$ when $p>1$ and $\eta=O(\delta^{1/(1-r)})$ when $p=1$.

We next evaluate the full local utility $V_i(b_i)$ using the pairwise relation~\eqref{eq::V_i}.
Algorithm~\ref{alg::local_eval}, presented below, simply enumerates all pairs of loss pieces, calls Algorithm~\ref{alg::evaluate_pairwise} for each pair, and returns the best value. Since the maximum is taken over approximate pairwise values, the accuracy of the local evaluation is inherited directly from Lemma~\ref{lem::eval_approx}.

\begin{lemma}
\label{lem::evaluate_vi}
    Fix $i\in[N]$ and $b_i \in \RR_+$. Suppose Assumptions~\ref{asp::regularity} and~\ref{asp::inner_oracle} hold. Then Algorithm~\ref{alg::local_eval} runs with oracle precision $\epsilon>0$ and golden-section tolerance $\eta > 0$, returns a feasible local solution $\hat s_i$ and a value $\hat V_i(b_i)$ satisfying
    \begin{align*}
        \left| V_i(b_i) - \hat V_i(b_i) \right| \leq 
        \begin{cases}
        O\left(\epsilon + \eta^{\min\{1/p,\,1-1/p\}} \right) &\mathrm{if}\ p>1 \\
        O\left(\epsilon + \eta^{1-r} \right) &\mathrm{if}\ p=1
        \end{cases}
    \end{align*}
    with the running time complexity of
    \begin{align*}
        \widetilde O \left(
        \big( 
            \textstyle \sum_{1\leq k_1 < k_2 \leq K}
            \mathsf{Cost}_{k_1,\epsilon}
            +
            \mathsf{Cost}_{k_2,\epsilon}
        \big)
        \log^2 ({1}/{\eta}) \right).
    \end{align*}
\end{lemma}

\begin{figure*}[!t]
\vspace{-1em}
\begin{center}
\begin{minipage}[t]{0.63\textwidth}
\begin{algorithm}[H]
    \small
    \caption{\texttt{LocalEval}}
    \label{alg::local_eval}
    \begin{algorithmic}
        \REQUIRE Sample $i$, budget $b_i$, tolerance $\eta>0$, oracle accuracy $\epsilon>0$ \\[1ex]
        \hspace{-1em}\textbf{Initialize:} $\hat V_i(b_i) \gets -\infty$ \vspace{0.5ex}

        \FOR{$1\leq k_1<k_2\leq K$} \vspace{0.5ex}
            \STATE $(V,s)\gets \texttt{PairEval}(i,k_1,k_2,b_i,\eta,\epsilon)$ \\[0.5ex]
            \IF{$V>\hat V_i(b_i)$} \vspace{0.5ex}
                \STATE $\hat V_i(b_i) \gets V, \ \hat s \gets s, \ (\hat k_1,\hat k_2) \gets (k_1,k_2)$ \\[0.5ex]
            \ENDIF \vspace{0.5ex}
        \ENDFOR
        \vspace{1ex}

        \ENSURE $\hat V_i(b_i) \approx V_i(b_i)$, optimizer $\hat s$ \& active pair $(\hat k_1,\hat k_2)$
    \end{algorithmic}
\end{algorithm}
\end{minipage}
\end{center}
\end{figure*}

With an efficient procedure for evaluating the local utilities $V_i$ in place, we now solve the upper-level problem~\eqref{eq::master}. For a fixed multiplier $\lambda\geq 0$, the Lagrangian-penalized local budget problem is 
\begin{align}
\label{eq::decoupled}
    \max_{0 \leq b_i \leq N \rho} ~ V_i(b_i)-\lambda b_i .
\end{align}
The multiplier $\lambda$ penalizes the use of transportation budget. Thus, for a fixed value of $\lambda$, the problem separates across empirical samples, and each local budget can be computed independently.
Algorithm~\ref{alg::budget_eval} solves this fixed-$\lambda$ problem by a one-dimensional golden-section search over $b_i$. Each evaluation of $V_i(b_i)-\lambda b_i$ calls Algorithm~\ref{alg::local_eval} to approximate $V_i(b_i)$ and store the corresponding local optimizer. The outer routine, Algorithm~\ref{alg::master_eval}, then searches for a multiplier whose induced local budgets satisfy the aggregate budget constraint. Since increasing $\lambda$ penalizes budget more heavily, the resulting aggregate budget is nonincreasing in $\lambda$, and bisection can be used.
In the following, we use the common tolerance parameter $\eta$ for both golden-section and bisection algorithms. 

\begin{figure*}[!t]
\begin{center}
\begin{minipage}{0.49\textwidth}
\begin{algorithm}[H]
    \small
    \caption{\texttt{BudgetEval}}
    \label{alg::budget_eval}
    \begin{algorithmic}
        \REQUIRE Sample $i$, multiplier $\lambda$, tolerances $\eta>0$, \\
        \hspace{2.2em} oracle precision $\epsilon>0$ \\[2.6ex]
        \hspace{-1em}\textbf{Initialize:} $\phi\gets\frac{\sqrt5-1}{2}$, $L_b\gets0$, $U_b\gets N\rho$ \\
        $b_i^{(1)}\gets U_b-\phi(U_b-L_b)$ \\
        $b_i^{(2)}\gets L_b+\phi(U_b-L_b)$ \\[0.5ex]

        $(\hat V^{(j)}, \dots) \gets \texttt{LocalEval}(i,b_i^{(j)},\eta,\epsilon), \hfill \forall j \in [2]$ \\[1ex]
        \WHILE{$U_b  - L_b > \eta$} \vspace{0.5ex}
            \IF{$\hat V^{(1)} - \lambda b_i^{(1)} < \hat V^{(2)} - \lambda b_i^{(2)}$} \vspace{0.5ex}
                \STATE $L_b\gets b_i^{(1)}$, $b_i^{(1)}\gets b_i^{(2)}$ \\[0.5ex]
                \STATE $b_i^{(2)}\gets L_b+\phi(U_b-L_b)$ \\[0.5ex]
                \STATE $\hat V^{(2)}\gets \texttt{LocalEval}(i,b_i^{(2)},\eta,\epsilon)$
            \ELSE \vspace{0.5ex}
                \STATE $U_b\gets b_i^{(2)}$, $b_i^{(2)}\gets b_i^{(1)}$ \\[0.5ex]
                \STATE $b_i^{(1)}\gets U_b-\phi(U_b-L_b)$ \\[0.5ex]
                \STATE $\hat V^{(1)}\gets \texttt{LocalEval}(i,b_i^{(1)},\eta,\epsilon)$
            \ENDIF \vspace{0.5ex}
        \ENDWHILE
        \vspace{1ex}

        \STATE $\hat b_i(\lambda) \gets L_b$ \\[0.5ex]
        \STATE $\big(\hat V_i(\hat b_i(\lambda)), \hat s, \hat k_1, \hat k_2) \gets \texttt{LocalEval}(i,\hat b_i(\lambda),\eta,\epsilon)$ \\[1ex]

        \ENSURE Budget $\hat b_i(\lambda) \approx b_i(\lambda)$ \& $\big(\hat V_i(\hat b_i(\lambda)\big), \hat s, \hat k_1, \hat k_2)$
    \end{algorithmic}
\end{algorithm}
\end{minipage}
\hfill
\begin{minipage}{0.495\textwidth}
\begin{algorithm}[H]
    \small
    \caption{\texttt{MasterEval}}
    \label{alg::master_eval}
    \begin{algorithmic}
        \REQUIRE Tolerances $\eta>0$, oracle precision $\epsilon>0$, and an upper bound $U_{\lambda}\geq \lambda^\star$ \\
        \hspace{-1em}\textbf{Initialize:} $\lambda^{(1)} \gets 0$, $\lambda^{(2)} \gets U_\lambda $ \\[0.5ex] 
        \STATE $\hat b_i^{(j)} \gets \texttt{BudgetEval} (i,\lambda^{(j)},\eta,\epsilon), \hfill \forall j \in [2], \forall i \in [N]$ \\[1ex]

        \WHILE{$\lambda^{(2)}-\lambda^{(1)}>\eta$} \vspace{0.5ex}
            \STATE $\displaystyle \bar \lambda \gets \frac{\lambda^{(1)}+\lambda^{(2)}}{2}$ \\[0.5ex]

            \STATE $\bar b_i \gets
            \texttt{BudgetEval}(i,\bar\lambda,\eta,\epsilon), \hfill \forall i \in [N]$ \\[1ex]

            \IF{$\sum_{i=1}^{N} \bar b_i > N \rho$} \vspace{1ex}
                \STATE $\lambda^{(1)}\gets\bar\lambda$ \\[0.5ex]
                \STATE $\hat b_i^{(1)} \gets \bar b_i, \hfill \forall i \in [N]$ \\[0.5ex]
            \ELSE \vspace{0.5ex}
                \STATE $\lambda^{(2)}\gets\bar\lambda$ \\[0.5ex]
                \STATE $\hat b_i^{(2)} \gets \bar b_i, \hfill \forall i \in [N]$ \\[0.5ex]
            \ENDIF \vspace{0.5ex}
        \ENDWHILE
        \vspace{1ex}

        \STATE $\displaystyle \theta \gets \frac{N \rho - \sum_{i=1}^{N} \hat b_i^{(2)}}{\sum_{i=1}^{N} \hat b_i^{(1)} - \sum_{i=1}^{N} \hat b_i^{(2)}}$ \\[1ex]
        \STATE $\bar b_i \gets \theta \hat b_i^{(1)} + (1-\theta) \hat b_i^{(2)}, \hfill \forall i \in [N]$ \\[0.5ex]

        \STATE $\big(\hat V_i(\bar b_i), \hat s_i, \hat k_{i1}, \hat k_{i2} \big) \!\gets\! \texttt{LocalEval}(i, \bar b_i, \eta, \epsilon), \hfill \forall i \in [N]$ \\[1ex]

        \ENSURE Budgets $\{\bar b_i\}_{i=1}^{N}$ \& $\{ \big(\hat V_i(\bar b_i), \hat s_i, \hat k_{i1}, \hat k_{i2} \big)\}_{i=1}^{N}$
    \end{algorithmic}
\end{algorithm}
\end{minipage}
\end{center}
\end{figure*}

Our next lemma shows that each of these maximization problems can be solved efficiently.

\begin{lemma}
\label{lem::Bi_eval}
    Fix $i\in[N]$ and $\lambda\geq0$. Suppose Assumptions~\ref{asp::regularity} and~\ref{asp::inner_oracle} hold. Then Algorithm~\ref{alg::budget_eval} runs with oracle precision $\epsilon>0$ and golden-section tolerance $\eta > 0$, returns $\big(\hat b_i(\lambda), \hat V_i(\hat b_i(\lambda))\big)$ satisfying
    \begin{align*}
        \max_{0\leq b\leq N\rho}
        \left\{
            V_i(b)-\lambda b
        \right\}
        -
        \left(
            \hat V_i(\hat b_i(\lambda))-\lambda\hat b_i(\lambda)
        \right)
        \leq
        \begin{cases}
        O\left(\epsilon + \eta^{\min\{1/p,\,1-1/p\}} \right) &\mathrm{if}\ p>1 \\
        O\left(\epsilon + \eta^{1-r} \right) &\mathrm{if}\ p=1
        \end{cases}
    \end{align*}
    with the running time complexity of
    \begin{align*}
        \widetilde O \!\left(
            K^2
            \cdot
            \mathsf{Cost}_{\epsilon}
            \cdot
            \log^3(1 / \eta)
        \right),
    \end{align*}
    where $\mathsf{Cost}_{k,\epsilon}\leq\mathsf{Cost}_{\epsilon}$ for all $k\in[K]$.
\end{lemma}

The proof of Lemma~\ref{lem::Bi_eval} is presented in Appendix~\ref{proof::lem::Bi_eval}.
It remains to control the outer search over the multiplier $\lambda$. Since each $V_i$ is nondecreasing in its budget argument, increasing $\lambda$ can only reduce the budgets selected by the Lagrangian search. Thus, the aggregate budget is nonincreasing, and bisection can be used to locate a multiplier whose induced allocation satisfies the budget constraint. Algorithm~\ref{alg::master_eval} maintains two multipliers: a lower multiplier $\lambda^{(1)}$ whose induced aggregate budget is above $\rho$, and an upper multiplier $\lambda^{(2)}$ whose induced aggregate budget is below $\rho$. The final interpolation step produces an allocation with aggregate budget exactly equal to $\rho$.

Our next theorem provides an end-to-end guarantee on both the running time and the accuracy of the proposed algorithm, together with an explicit procedure for constructing a worst-case distribution from the computed solution.

\begin{theorem}
\label{thm::error_analysis}
    Let $V^\star$ denote the optimal value of~\eqref{eq::master}.
    Suppose Assumptions~\ref{asp::regularity} and~\ref{asp::inner_oracle} hold. Then Algorithm~\ref{alg::master_eval} runs with oracle precision $\epsilon > 0$ and tolerance $\eta > 0$, returns a feasible budget allocation $\{\bar{b}_i\}_{i=1}^{N}$ satisfying
    \begin{align*}
        \left| V^\star
        -
        \frac{1}{N} \sum_{i=1}^{N} V_i(\bar b_i) \right|
        \leq
        \begin{cases}
        O\left(\epsilon + \eta^{\min\{1/p,\,1-1/p\}} \right) &\mathrm{if}\ p>1 \\
        O\left(\epsilon + \eta^{1-r} \right) &\mathrm{if}\ p=1
        \end{cases}
    \end{align*}
    with the running time complexity of
    \[
        \widetilde O \!\left(
            N 
            \cdot 
            K^2
            \cdot
            \mathsf{Cost}_{\epsilon}
            \cdot
            \log^4(1/ \eta)
        \right),
    \]
    where $\mathsf{Cost}_{k,\epsilon}\leq\mathsf{Cost}_{\epsilon}$ for all $k\in[K]$. Moreover, the corresponding $2N$-point distribution 
    \[
    \QQ_{2N} := \frac{1}{N} \sum_{i=1}^N \left(\hat\alpha_{i\hat{k}_{i1}} \delta_{\hat{z}_i + \hat v_{i\hat{k}_{i1}}}+\hat\alpha_{i\hat{k}_{i2}} \delta_{\hat{z}_i + \hat v_{i\hat{k}_{i2}}}\right)
    \]
    satisfies $\QQ_{2N}\in \cP$ and 
    \[
    \max_{\bQ \in \cP}\left\{\mathbb{E}_{z \sim \bQ} \left[\ell(z) \right] \right\} - \mathbb{E}_{z \sim \bQ_{2N}} \left[\ell(z) \right]\leq 
    \begin{cases}
    O\left(\epsilon + \eta^{\min\{1/p,\,1-1/p\}} \right) &\mathrm{if}\ p>1 \\
    O\left(\epsilon + \eta^{1-r} \right) &\mathrm{if}\ p=1.
    \end{cases}
    \]
\end{theorem}

The proof is provided in Appendix~\ref{proof::thm::error_analysis}. Fixing $p\geq 1$ and/or $0\leq r<1$, and choosing $\epsilon = O(\delta)$ alongside an appropriate polynomial tolerance $\eta$ (e.g., $\eta = O(\delta^{\max\{p, p/(p-1)\}})$ for $p>1$), the theorem implies that the proposed algorithm produces a $\delta$-optimal worst-case distribution within a running time that scales linearly with the number of samples $N$ and the local worst-case oracle complexity $\mathsf{Cost}_{\delta}$, quadratically with the number of components $K$, and only poly-logarithmically with $1/\delta$.
The dependence on $1/\delta$ is particularly significant, as it shows that the overall computational overhead of the proposed framework is modest relative to the complexity of the local worst-case oracle itself. In other words, up to logarithmic factors, the algorithm essentially preserves the accuracy dependence of the underlying local oracle.
Moreover, the algorithm constructs a worst-case distribution supported on at most $2N$ points with an optimality gap of at most $O(\delta)$. To the best of our knowledge, this is the first tailored algorithm for the worst-case expectation problem with such guarantees.

\begin{remark}
Algorithm~\ref{alg::master_eval} requires an upper bound on \(\lambda^\star\) as input. Although the algorithm leaves this bound unspecified, Lemma~\ref{lem::U_lambda} in Appendix~\ref{proof::lem::Bi_eval} shows that such a bound can be computed a priori using the supergradients of \(\ell_k\) evaluated at the empirical samples $\hat z_i$.
\end{remark}

\subsection{Compression to (N+1)-Point Worst-Case Distribution}
\label{subsec::support_compression}

Despite the computational efficiency of the proposed algorithm, the constructed worst-case distribution does not attain the sparsity guarantee of Proposition~\ref{prop::nonconvex-equivalence}. In this subsection, we develop a post-processing procedure that compresses the $2N$-point distribution returned by Algorithm~\ref{alg::master_eval} into a distribution supported on at most $N+1$ points, without sacrificing optimality and with only negligible additional computational cost.

Recall that Algorithm~\ref{alg::master_eval} constructs a distribution in which the probability mass associated with each empirical sample $\hat z_i$ is split between two perturbed points, namely $\hat z_i + \hat v_{i\hat k_{i1}}$ and $\hat z_i + \hat v_{i\hat k_{i2}}$, with corresponding masses $\hat\alpha_{i\hat k_{i1}}$ and $\hat\alpha_{i\hat k_{i2}}$. Since the worst-case expectation problem~\eqref{eq::worst_case} is equivalent to~\eqref{eq::equivalence_primal}, one may fix $v_{ik} = \hat v_{ik}$ for $i\in [N]$ and $k\in \{\hat k_{i1},\hat k_{i2}\}$, and set $v_{ik}=0$ for all remaining indices. The resulting problem~\eqref{eq::equivalence_primal} is then defined only over the weights $\{\alpha_{ik}\}$. This reformulation has two key properties: (1) it is an LP; and (2) by standard LP theory, it admits an optimal BFS with at most $N+1$ nonzero variables. Consequently, one can recover an optimal worst-case distribution supported on at most $N+1$ points without sacrificing its optimality.

While this post-processing approach is theoretically straightforward, its practical efficiency depends on recovering an optimal BFS efficiently. Generic LP solvers can be used, but they often fail to exploit the underlying structure of the problem and may therefore incur unnecessary computational overhead. Fortunately, in our setting, the resulting LP admits a much simpler characterization: it reduces to a Fractional Knapsack Problem. As a result, a simple greedy algorithm recovers the sparse optimal solution in time $O(N\log N)$. The next theorem formalizes this result.

\begin{theorem}
\label{thm::N+1-reconstruction}
There exists an algorithm that, given the output of Algorithm~\ref{alg::master_eval}, constructs an approximate worst-case distribution $\bQ_{N+1} \in \cP$ supported on at most $N+1$ atoms in time $O(N\log N)$ such that
\[
\mathbb{E}_{z \sim \bQ_{N+1}} \left[\ell(z)\right]
\geq
\mathbb{E}_{z \sim \bQ_{2N}} \left[\ell(z)\right].
\]
\end{theorem}

\begin{proof}
For each $i\in [N]$, define $\ell_i^+ := \ell_{\hat k_{i1}}(\hat z_i + \hat v_{i\hat k_{i1}})$, $\ell_i^- := \ell_{\hat k_{i2}}(\hat z_i + \hat v_{i\hat k_{i2}})$, $c_i^+ := c(\hat z_i + \hat v_{i\hat k_{i1}},\hat z_i)$, and $c_i^- := c(\hat z_i + \hat v_{i\hat k_{i2}},\hat z_i)$. Without loss of generality, assume that $c_i^+ \geq c_i^-$ for all $i\in [N]$.
Fix $v_{ik} = \hat v_{ik}$ for $i\in [N]$ and $k\in \{\hat k_{i1},\hat k_{i2}\}$, and set $v_{ik}=0$ for all remaining indices. Under this restriction, problem~\eqref{eq::equivalence_primal} reduces to
\[
\max \left\{
\frac{1}{N}\sum_{i=1}^{N} \left( \alpha_i^+ \ell_i^+ + \alpha_i^- \ell_i^- \right)
:
\alpha_i^+,\alpha_i^- \in \RR_+,\;
\alpha_i^+ + \alpha_i^- = 1,\ \forall i\in [N],\;
\sum_{i=1}^N \left( \alpha_i^+ c_i^+ + \alpha_i^- c_i^- \right) \leq N\rho
\right\}.
\]
Eliminating the variables $\alpha_i^-$ via the substitution $\alpha_i^- = 1-\alpha_i^+$, and defining $\tilde c_i := c_i^+ - c_i^- \geq 0$, $\tilde \ell_i := \ell_i^+ - \ell_i^-$, and $\tilde \rho := N\rho - \sum_{i=1}^N c_i^-\geq 0$, the above problem is equivalent to
\[
\left(\frac{1}{N}\sum_{i=1}^N \ell_i^- \right)
+
\max \left\{
\frac{1}{N}\sum_{i=1}^N \alpha_i^+ \tilde \ell_i
:
\alpha_i^+ \in [0,1],\ \forall i\in [N],\;
\sum_{i=1}^N \alpha_i^+ \tilde c_i \leq \tilde \rho
\right\}.
\]
The resulting optimization problem is precisely the Fractional Knapsack Problem and can be solved via a sorting-based greedy algorithm in time $O(N\log N)$ \cite[Chapter~16.2]{cormen2022introduction}. Moreover, the returned optimal solution has at most one fractional solution. Consequently, among the $2N$ candidate atoms, at most $N+1$ receive nonzero mass.
\end{proof}


\section{Primal DRO Problem}
\label{sec::primal-DRO}

Having established an efficient algorithm for solving the worst-case expectation problem~\eqref{eq::worst_case}, we now turn to the primal DRO problem
\begin{align}
\label{eq::primal_DRO}
    \min_{x \in \cX} \max_{\bQ \in \cP} \, \EE_{z \sim \bQ} [\ell(x, z)].
\end{align}
Throughout, we define $f(x, \bQ) := \EE_{z \sim \bQ} [\ell(x, z)]$.
Designing efficient first-order methods for~\eqref{eq::primal_DRO} is challenging since standard primal-dual algorithms are not well suited to optimization over probability measures. Specifically, gradient-based updates of the dual variable, corresponding to the worst-case distribution $\bQ$, generally fail to preserve feasibility with respect to the ambiguity set, and the resulting projection step is often computationally~intractable.

To address this issue, we propose the \emph{Distributional Best-Response Algorithm}, inspired by the best-response framework detailed in \cite[Algorithm~12.2]{orabona2019modern}. The main idea is to decouple the minimax interaction between the primal variable $x\in \cX$ and dual variable $\bQ\in \cP$. At each iteration, we first fix the current primal decision $x_t$ and allow the adversary to compute a worst-case distribution $\bQ_t$, namely its \emph{best response}. We then update the primal decision against this adversarial distribution using an online learning step. By \emph{freezing} the adversary's best response $\mathbb{Q}_t$ against the current decision $x_t$, we guarantee that the decision maker is always reacting to the maximum possible distributional shift admissible within the ambiguity set. The resulting procedure is formally described in Algorithm~\ref{alg::1}.

\begin{figure*}[!t]
\vspace{-1em}
\begin{center}
\begin{minipage}[t]{0.7\textwidth}
\begin{algorithm}[H]
    \small
    \caption{\texttt{Distributional Best-Response}}
    \label{alg::1}
    \begin{algorithmic}
        \REQUIRE Initial decision $x_1 \in \cX$ \\[1ex]
        \hspace{-1em}\textbf{Initialize:} $\bar{x}_0 = 0$, $\bar{\bQ}_0 = 0$ \vspace{0.5ex}

        \FOR{$t=1,2,\dots, T$} \vspace{0.5ex}
            \STATE $\bQ_t \approx \argmax_{\bQ \in \cP} f(x_t, \bQ)$; e.g., using Algorithm~\ref{alg::master_eval} and Theorem~\ref{thm::error_analysis} \\[0.5ex]
            \STATE $x_{t+1} \gets \mathcal{A}(x_t, f(\cdot, \bQ_t))$, where $\mathcal{A}$ is an online algorithm \\[0.5ex]
            \STATE $\bar{\bQ}_{t} \gets \frac{t-1}{t}\bar{\bQ}_{t-1} +\frac{1}{t}\bQ_t$ \\[0.5ex]
            \STATE $\bar{x}_{t} \gets \frac{t-1}{t}\bar{x}_{t-1}+\frac{1}{t}x_{t}$ \\[0.5ex]
        \ENDFOR
        \vspace{1ex}

        \ENSURE $(\bar x_T, \bar \bQ_T)$
    \end{algorithmic}
\end{algorithm}
\end{minipage}
\end{center}
\end{figure*}

For the update of the primal variable, rather than committing to a specific rule, we abstract the learning step to accommodate \emph{any} feasible online optimization algorithm $\mathcal{A}$. Depending on the geometry of $\cX$ and the analytical properties of the loss function, one can easily plug in methods such as Online Subgradient Descent~\cite{zinkevich2003online}, Online Mirror Descent~\cite{hazan2022introduction}, or projection-free Frank-Wolfe-type algorithms~\cite{hazan2012projection} without altering the fundamental structure of the framework.

Before proceeding to the formal analysis, we establish a criterion for evaluating the performance of Algorithm~\ref{alg::1}. Let $(x^\star, \bQ^\star)$ denote a saddle point of the minimax problem \eqref{eq::primal_DRO} satisfying
\begin{align*}
    f(x^\star, \bQ)\leq f(x^\star, \bQ^\star)\leq f(x, \bQ^\star), \quad \forall x\in \cX, \bQ\in \cP.
\end{align*}
One can only hope to numerically obtain a saddle point satisfying the above inequality up to a certain tolerance $\zeta$. We thus rely on the notion of $\zeta$-saddle point~\cite[Definition~12.10]{orabona2019modern}. 
\begin{definition}[$\zeta$-saddle point]
    Let $\zeta\geq 0$. A point $(\bar x, \bar \bQ)\in \cX\times \cP$ is called a $\zeta$-saddle point of $f$ if
    \begin{align*}
    f(\bar x, \bQ)-\zeta\leq f(\bar x,\bar \bQ)\leq f(x, \bar\bQ)+\zeta, \quad \forall x\in \cX, \bQ\in \cP.
\end{align*}
\end{definition}
Our goal is to ensure that the aggregated solution $(\bar{x}_T, \bar{\bQ}_T)$ produced by Algorithm~\ref{alg::1} is a $\zeta$-saddle point, for a sufficiently small $\zeta\geq 0$. To achieve this, it is common to analyze the \textit{duality gap}:
\[
    \mathrm{Gap}(\bar{x}_T, \bar{\bQ}_T):=\max_{\bQ \in \cP} f(\bar x_T, \bQ) - \min_{x\in \cX} f(x, \bar\bQ_T).
\]
According to \cite[Lemma 12.11]{orabona2019modern}, $\mathrm{Gap}(\bar{x}_T, \bar{\bQ}_T)\leq \zeta$ implies that $(\bar{x}_T, \bar{\bQ}_T)$ is a $\zeta$-saddle point. Therefore, it suffices to control the duality gap of the computed solution $(\bar{x}_T, \bar{\bQ}_T)$.

To rigorously bound this gap without tying our analysis to specific primal and dual updates, we make two general assumptions. The first assumes access to an oracle that computes an approximate solution to the worst-case expectation problem, corresponding to the dual update. The second assumes access to an online algorithm for the primal update that satisfies a suitable regret guarantee.

\begin{assumption}[Worst-case Expectation Oracle]\label{asp::worst-case-oracle}
    Let $\delta\geq 0$. There exists a worst-case expectation oracle that, for every $x_t$, returns a distribution $\bQ_t\in \cP$ that satisfies
    \begin{align*}
        \mathbb{E}_{z \sim \bQ_t} \left[\ell(x_t,z) \right]\geq \max_{\bQ\in \cP}\left\{\mathbb{E}_{z \sim \bQ} \left[\ell(x_t,z) \right]\right\}-\delta.
    \end{align*}
\end{assumption}
Algorithm~\ref{alg::master_eval} (or its enhanced variant in Section \ref{subsec::support_compression}) is specifically designed to satisfy the conditions of Assumption~\ref{asp::worst-case-oracle}. Consequently, it serves as an ideal choice for the worst-case expectation oracle.

\begin{assumption}[No-Regret Guarantee]
\label{asp::regret}
The sequence of decisions $\{x_t\}_{t=1}^T$ generated by the online algorithm $\mathcal{A}$ against the sequence of adversarial loss functions $\{f(\cdot, \bQ_t)\}_{t=1}^T$ achieves a sublinear regret $\cR_T$. That is, there exists a bound $\cR_T = o(T)$ such that:
\begin{align*}
    \sum_{t=1}^{T} f(x_t, \bQ_t) - \min_{x\in \cX} \sum_{t=1}^T f(x, \bQ_t) \leq \cR_T.
\end{align*}
\end{assumption}

A wide variety of online algorithms satisfy Assumption~\ref{asp::regret} under standard geometric and regularity conditions \cite{hazan2016introduction,orabona2019modern,shalev2025online}. The classical Projected Online Subgradient Method (POSM) achieves minimax optimal regret of $\cR_T = O(\sqrt{T})$ generally, which improves to $\cR_T = O(\log T)$ for strongly convex objectives \cite{zinkevich2003online,hazan2007logarithmic,abernethy2008optimal}. When Euclidean projections are computationally prohibitive, Online Mirror Descent (OMD) and Follow-the-Regularized-Leader (FTRL) offer geometry-adaptive Bregman projections with comparable regret rates \cite{beck2003mirror,shalev2025online}. Alternatively, for highly structured domains, projection-free Online Frank--Wolfe algorithms substitute projections with cheaper linear optimization oracles, attaining sublinear regret bounds of $\cR_T = O(T^{3/4})$ or $\cR_T = O(T^{2/3})$ depending on the smoothness of the loss \cite{hazan2012projection,hazan2020faster}.

With Assumptions~\ref{asp::worst-case-oracle} and~\ref{asp::regret} in place, we are now ready to establish the convergence of Algorithm~\ref{alg::1}.~The proof is based on decomposing the duality gap into dual and primal regret terms. Since our algorithm computes an approximate worst-case distribution at every iteration, the dual regret vanishes up to the computational error $\zeta$ introduced by the worst-case oracle. Consequently, the final duality gap is governed entirely by the primal regret of the online algorithm $\cA$ and the oracle error $\zeta$.

\begin{theorem}
\label{thm::convergence_primal}
Under Assumptions~\ref{asp::regularity},~\ref{asp::worst-case-oracle}, and~\ref{asp::regret}, Algorithm~\ref{alg::1} outputs a solution $(\bar{x}_T, \bar{\bQ}_T)\in \cX\times \cP$ satisfying
\[
\mathrm{Gap}(\bar{x}_T, \bar{\bQ}_T) 
:= \max_{\bQ \in \cP} f(\bar x_T, \bQ) - \min_{x\in \cX} f(x, \bar\bQ_T) 
\leq \frac{\cR_T}{T}+\delta,
\]
where $\cR_T$ is the regret of the chosen online algorithm $\mathcal{A}$, and $\delta$ is the error of the worst-case expectation oracle.
\end{theorem}

\begin{proof}
The duality gap of $(\bar{x}_T, \bar{\bQ}_T)$ can be written as
\begin{align*}
& \max_{\bQ \in \cP} f(\bar x_T, \bQ) - \min_{x\in \cX} f(x, \bar\bQ_T) \\
=& \left\{ \max_{\bQ \in \cP} f(\bar x_T, \bQ) - \frac{1}{T}\sum_{t=1}^{T} f(x_t, \bQ_t) \right\}
+ \left\{ \frac{1}{T}\sum_{t=1}^{T} f(x_t, \bQ_t) - \min_{x\in \cX} f(x, \bar\bQ_T) \right\}\\
\stackrel{(a)}{\leq}& \left\{ \max_{\bQ \in \cP}\frac{1}{T} \sum_{t=1}^{T} f(x_t, \bQ) - \frac{1}{T}\sum_{t=1}^{T} f(x_t, \bQ_t) \right\}
+ \left\{ \frac{1}{T}\sum_{t=1}^{T} f(x_t, \bQ_t) - \min_{x\in \cX} \frac{1}{T}\sum_{t=1}^{T} f(x, \bQ_t) \right\}\\
\stackrel{(b)}{\leq}& \left\{ \frac{1}{T} \sum_{t=1}^{T} \max_{\bQ \in \cP} f(x_t, \bQ) - \frac{1}{T}\sum_{t=1}^{T} f(x_t, \bQ_t) \right\}
+ \left\{ \frac{1}{T}\sum_{t=1}^{T} f(x_t, \bQ_t) - \min_{x\in \cX} \frac{1}{T}\sum_{t=1}^{T} f(x, \bQ_t) \right\}\\
\stackrel{(c)}{=}& \frac{1}{T}\sum_{t=1}^{T} \left(f(x_t, \bQ_t)+\delta\right) - \min_{x\in \cX} \frac{1}{T}\sum_{t=1}^{T} f(x, \bQ_t) \\
\stackrel{(d)}{\leq}& \frac{\cR_T}{T}+\delta,
\end{align*}
where $(a)$ is due to the convexity of $f(\cdot, \bQ)$ and the linearity of $f(x, \cdot)$; $(b)$ is obtained by swapping the order of the summation $\sum_{t=1}^{T}$ and the maximization $\max_{\bQ \in \cP}$; $(c)$ is obtained by realizing that $f(x_t, \bQ_t) \geq \max_{\bQ \in \cP} f(x_t, \bQ)-\delta$ due to Assumption~\ref{asp::worst-case-oracle}; $(d)$ directly follows from Assumption~\ref{asp::regret}.
\end{proof}

As a special case, when the proposed budget-allocation-based method and vanilla POSM are used as the dual and primal oracles, under appropriate choices of parameters, Algorithm~\ref{alg::1} finds a $\zeta$-saddle point in $O(1/\zeta^2)$ iterations.


\section{Dual DRO Problem}
\label{sec::dual_dro}

In this section, we investigate the structural properties of the dual DRO problem under the optimal transport ambiguity set defined in~\eqref{eq::ambiguity_set}. Resolving this formulation requires identifying a least-favorable distribution that solves the infinite-dimensional maximin problem:
\begin{align}
\label{eq::dual_DRO}
\max_{\bQ \in \cP} \min_{x \in \cX} \, \EE_{z \sim \bQ} [\ell(x, z)].
\end{align}
Characterizing this least-favorable distribution is fundamentally more challenging than computing a worst-case distribution for a fixed primal decision $x$ as in~\eqref{eq::worst_case}. By the strong duality guaranteed in Lemma~\ref{lem::existence_saddle}, our distributional best-response framework (Algorithm~\ref{alg::1}) naturally produces an approximate least-favorable distribution $\bar \bQ_T$. Nonetheless, since the algorithm relies on a running average of intermediate worst-case measures $\bQ_t$, the support of the returned distribution grows linearly with the number of iterations, resulting in up to $T(N+1)$ atoms. Although this procedure is algorithmically efficient, the massive support size of its output renders downstream evaluation, storage, or subsequent re-optimization computationally prohibitive.

This raises a natural question: can one analytically recover a least-favorable distribution with significantly fewer atoms? For convex-piecewise concave losses with $K$ pieces, existing literature establishes that the dual DRO problem admits a least-favorable distribution supported on at most $KN$ atoms \cite[Theorem~2]{shafiee2025nash}. However, we show that this construction is generally not tight. Specifically, by carefully analyzing the dual equilibrium conditions, we prove the existence of a least-favorable distribution supported on at most $\min\{N+n+1, KN\}$ points. Whenever the primal decision dimension is moderate ($n < (K-1)N - 1$), this represents a strict reduction in support size. 

Under Assumption~\ref{asp::regularity}, the following result demonstrates that the dual problem admits a finite-dimensional nonconvex reformulation that rigorously exposes this minimal support structure.

\begin{proposition}
\label{prop::dual-nonconvex-equivalence}
    Under Assumption~\ref{asp::regularity}, the dual DRO problem~\eqref{eq::dual_DRO} is equivalent to the nonconvex program
    \begin{equation}
    \label{eq::dual_DRO_reformulation}
        \left\{
        \begin{array}{cll}
            \max & \displaystyle - \frac{1}{N}\sum_{i=1}^{N} \sum_{k=1}^{K} \alpha_{ik} \ell_k^{*1}\left( y_{ik}, \hat{z}_i + v_{ik} \right) - \nu \sigma_{\cX}\left(\theta \right) \\[1.0ex]
            \mathrm{s.t.} & \theta \in \RR^n, \ \nu \in \RR_+, \ \alpha_{ik} \in \bR_+, \ v_{ik} \in \bR^m,\ y_{ik} \in \RR^n ,\ \hat{z}_i + v_{ik} \in \cZ \ \forall i \in [N], \forall k \in [K] \\
            & \displaystyle \frac{1}{N} \sum_{i=1}^{N} \sum_{k=1}^{K} \alpha_{ik} y_{ik}  + \nu \theta = 0, \ \frac{1}{N} \sum_{i=1}^N \sum_{k=1}^K \alpha_{ik} c(\hat{z}_i + v_{ik}, \hat{z}_i) \leq \rho, \  \sum_{k=1}^K \alpha_{ik} = 1 \ \forall i\in [N] \\
            &  
        \end{array}
        \right.
    \end{equation}
    where $\ell_k^{*1}(w,z)$ denotes the conjugate of $\ell_k(x,z)$ with respect to its first argument $x$ for fixed $z$, and $\sigma_{\cX}(\cdot)$ is the support function over $\cX$. Given an optimal solution $\left( \{\alpha_{ik}^\star, v_{ik}^\star, y_{ik}^\star\}, \theta^\star, \nu^\star \right)$ to~\eqref{eq::dual_DRO_reformulation}, the discrete distribution
    \begin{align}
    \label{eq::dual-DRO-Q}
    \QQ^\star_{\mathrm{dual}} := \frac{1}{N} \sum_{i=1}^N \sum_{k=1}^K \alpha_{ik}^\star \delta_{\hat{z}_i + v_{ik}^\star}
    \end{align}
    belongs to the ambiguity set $\cP$ and attains the maximum in \eqref{eq::dual_DRO}.
    Furthermore, there exists an optimal solution in which at most $\min\{N+n+1, KN\}$ of the weights $\{\alpha_{ik}^\star\}$ are nonzero. Consequently, there exists a least-favorable distribution $\QQ^\star_{\mathrm{dual}}$ of the dual problem~\eqref{eq::dual_DRO} supported on at most $\min\{N+n+1, KN\}$ points.
\end{proposition}

\begin{proof}
Consider any coupling $\gamma \in \Gamma(\bQ, \hat{\bP})$. Since the empirical distribution is $\hat{\bP} = \frac{1}{N} \sum_{i=1}^N \delta_{\hat{z}_i}$, we can disintegrate $\gamma$ into conditional probability distributions $\bQ_i \in \cP(\cZ)$ associated with each sample~$\hat{z}_i$. The dual DRO problem \eqref{eq::dual_DRO} can be thus reformulated as:
\begin{align*}
    \max_{\bQ_i \in \cP(\cZ)} \left\{ \min_{x \in \cX} ~  \frac{1}{N} \sum_{i=1}^N \mathbb{E}_{z \sim \bQ_i} \left[\max_{k \in [K]} \ell_k(x, z) \right] : \frac{1}{N} \sum_{i=1}^N \mathbb{E}_{z \sim \bQ_i}[c(z, \hat{z}_i)] \leq \rho \right\}.
\end{align*}

To resolve the pointwise maximum inside the expectation, we express it as a continuous maximization over the probability simplex $\Delta_K \subset \RR^K$. By the interchangeability principle for integration and maximization \cite[Theorem~14.60]{rockafellar1998variational}, we can pull this supremum outside the expectation by introducing measurable mappings $s_i : \cZ \to \Delta_K$. This enables to rewrite the DRO dual problem as:
$$
\max_{\bQ_i \in \cP(\cZ)} \, \min_{x \in \cX}\, \max_{s_i: \cZ \to \Delta_K}
\, \left\{ \frac{1}{N} \sum_{i=1}^N \sum_{k=1}^{K} \mathbb{E}_{z \sim \bQ_{i}}\left[s_{ik}(z) \ell_k(x, z)\right] : \frac{1}{N} \sum_{i=1}^N \sum_{k=1}^{K} \mathbb{E}_{z \sim \bQ_i}[s_{ik}(z) c(z, \hat{z}_i)] \leq \rho \right\},
$$
where $s_{ik}(\cdot)$ denotes the $k$-th coordinate of $s_i(\cdot)$. For any fixed $\bQ_i$, we apply Sion's Minimax Theorem~\cite{sion1958general} to interchange the operations $\min_{x}$ and $\max_{s_i}$. Notice that $\Delta_K$ is a compact and convex subset of $\RR^K$. By \cite[Theorem 2]{diestel1977remarks}, the space of measurable mappings $s_i: \cZ \to \Delta_K$ is weakly compact. Since the objective is convex in $x \in \cX$, and linear in $s_i$ with respect to this weak topology, the requirements of Sion's theorem are satisfied and the minimax equality holds. Interchanging the order yields:
$$
\max_{\bQ_i \in \cP(\cZ)} \, \max_{s_i: \cZ \to \Delta_K}
\, \min_{x \in \cX}\, \left\{ \frac{1}{N} \sum_{i=1}^N \sum_{k=1}^{K} \mathbb{E}_{z \sim \bQ_{i}}\left[s_{ik}(z) \ell_k(x, z)\right] : \frac{1}{N} \sum_{i=1}^N \sum_{k=1}^{K} \mathbb{E}_{z \sim \bQ_i}[s_{ik}(z) c(z, \hat{z}_i)] \leq \rho \right\}.
$$
We can now group the maximization over $\bQ_i$ and $s_i$ into a unified choice of probability weights and conditional distributions. Let $\alpha_{ik} = \mathbb{E}_{z \sim \bQ_i}[s_{ik}(z)]$ denote the overall probability mass assigned to component $k$, satisfying $\sum_{k=1}^K \alpha_{ik} = 1$. Let the conditional probability measure be defined as $\bQ_{ik}(dz) = \frac{s_{ik}(z)}{\alpha_{ik}} \bQ_i(dz)$ when $\alpha_{ik} > 0$ (and as an arbitrary valid distribution when $\alpha_{ik} = 0$). The problem is equivalent to:
$$ \max_{\bQ_{ik} \in \cP(\cZ)} \max_{\alpha_{i} \in \Delta_K} \min_{x \in \cX} \left\{ \frac{1}{N} \sum_{i=1}^N \sum_{k=1}^K \alpha_{ik} \EE_{z \sim \bQ_{ik}} \left[\ell_k(x, z)\right] : \frac{1}{N} \sum_{i=1}^N \sum_{k=1}^{K} \alpha_{ik} \mathbb{E}_{z \sim \bQ_{ik}}[c(z, \hat{z}_i)] \leq \rho \right\}. $$
For any fixed conditional distribution $\bQ_{ik}$, let $\tilde{z}_{ik} = \EE_{z \sim \bQ_{ik}}[z]$ be its conditional mean. Because the loss component $\ell_k(x, z)$ is concave in $z$, Jensen's inequality guarantees that $\EE_{z \sim \bQ_{ik}} [\ell_k(x, z)] \le \ell_k(x, \tilde{z}_{ik})$ for \emph{every} fixed $x \in \cX$. Consequently, taking the minimum over $x$ on both sides preserves the inequality:
$$ \min_{x \in \cX} \left\{ \frac{1}{N} \sum_{i=1}^N \sum_{k=1}^K \alpha_{ik} \EE_{z \sim \bQ_{ik}} [\ell_k(x, z)] \right\} \le \min_{x \in \cX} \left\{ \frac{1}{N} \sum_{i=1}^N \sum_{k=1}^K \alpha_{ik} \ell_k(x, \tilde{z}_{ik}) \right\}. $$
Furthermore, the convex transportation cost satisfies $c(\tilde{z}_{ik}, \hat{z}_i) \le \EE_{z \sim \bQ_{ik}}[c(z, \hat{z}_i)]$. Therefore, replacing any arbitrary conditional distribution $\bQ_{ik}$ with a Dirac measure $\delta_{\tilde{z}_{ik}}$ placed at its mean strictly improves or maintains the outer maximization objective while requiring less or equal transport budget. Applying the change of variables $v_{ik} = \tilde{z}_{ik} - \hat{z}_i$ safely restricts the adversary's search space to finite-dimensional spatial perturbations without altering the problem's optimal value.

With the distributions collapsed to Dirac measures, we arrive at:
$$
\max \left\{ \min_{x \in \RR^n} \frac{1}{N} \sum_{i=1}^N \sum_{k=1}^K \alpha_{ik} \ell_k(x, \hat{z}_i + v_{ik}) + \delta_{\cX}(x) : 
\begin{array}{l}
    \alpha_i \in \Delta_K, \ v_{ik} \in \bR^m, \ \hat z_i + v_{ik} \in \cZ \ \forall i \in [N] \\ \displaystyle \frac{1}{N} \sum_{i=1}^N \sum_{k=1}^{K} \alpha_{ik} \mathbb{E}_{z \sim \bQ_{ik}}[c(\hat{z}_i + v_{ik}, \hat{z}_i)] \leq \rho 
\end{array}
\right\},
$$
where $\delta_{\cX}(\cdot)$ is the indicator function for the domain $\cX$. We dualize this inner minimization using classical Fenchel duality for a sum of convex functions. Introducing dual variables $\lambda_{ik} \in \RR^n$ for each loss component and a global dual variable $\lambda_0 \in \RR^n$ for the indicator $\delta_{\cX}$, we apply the variable substitutions $\lambda_{ik} = \frac{\alpha_{ik}}{N} y_{ik}$ and $\lambda_0 = \nu \theta$ (with $\nu \ge 0$). This yields the dual conjugates $\frac{\alpha_{ik}}{N}\ell_k^{*1}(y_{ik}, \hat{z}_i + v_{ik})$ and $\nu \sigma_{\cX}(\theta)$. The Fenchel stationarity condition $\sum_{i, k} \lambda_{ik} + \lambda_0 = 0$ produces the equilibrium constraint $\frac{1}{N}\sum_{i,k} \alpha_{ik}y_{ik} + \nu\theta = 0$. 
Substituting this dual maximization back into the problem yields the reformulation~\eqref{eq::dual_DRO_reformulation} and constructs the least-favorable distribution~\eqref{eq::dual-DRO-Q}. By the same coercivity and compactness arguments utilized in Proposition~\ref{prop::nonconvex-equivalence}, this joint maximization attains a finite optimum.

To establish the minimal support bound, we use the property that joint maximization is equivalent to sequential maximization. By fixing all continuous parameters to their optimal values $\{v_{ik}^\star, y_{ik}^\star, \theta^\star, \nu^\star\}$, the remaining optimization over the probability weights $\{\alpha_{ik}\}$ reduces to an LP:
\begin{equation*}
\left\{\!\!
\begin{array}{cl}
    \max & \displaystyle  - \frac{1}{N}\sum_{i=1}^{N} \sum_{k=1}^{K} \alpha_{ik} \ell_k^{*1}\left( y_{ik}^\star, \hat{z}_i + v_{ik}^\star \right) - \nu^\star \sigma_{\cX}\left(\theta^\star \right)  \\
    \mathrm{s.t.} & \displaystyle 
    \alpha \in \bR_+^{N \times K}, 
    \frac{1}{N} \sum_{i=1}^{N} \sum_{k=1}^{K} \alpha_{ik} y_{ik}^\star \!+\! \nu^\star \theta^\star \!=\! 0, 
    \frac{1}{N} \sum_{i=1}^N \sum_{k=1}^K \alpha_{ik} c(\hat{z}_i \!+\! v_{ik}^\star, \hat{z}_i) \!\leq\! \rho, 
    \sum_{k=1}^K \alpha_{ik} \!=\! 1, \, \forall i \in [N]  
\end{array}
\right.
\end{equation*}
This LP possesses $KN$ nonnegative variables and exactly $N + n + 1$ structural constraints: $N$ normalization equalities, $1$ aggregate budget inequality, and $n$ linear equalities governing the subgradient equilibrium for $x^\star$. By the fundamental theorem of linear programming, the optimal value is attained at a basic feasible solution (BFS) with at most as many strictly positive variables as there are structural constraints. Thus, at most $\min\{N + n + 1, KN\}$ weights $\alpha_{ik}^\star$ can be strictly positive, limiting the support of the resulting worst-case distribution $\QQ^\star_{\mathrm{dual}}$ to at most $\min\{N + n + 1, KN\}$ unique atoms.
\end{proof}

Proposition~\ref{prop::dual-nonconvex-equivalence} establishes the existence of a least-favorable distribution supported on at most $\min\{N + n + 1, KN\}$ atoms for the dual DRO problem. Unlike its primal counterpart, the existence of such a sparse least-favorable distribution was previously unknown. One might speculate that this bound could be improved, particularly in light of the primal setting, where worst-case distributions supported on at most $N+1$ atoms are known to exist. The following lemma confirms that this bound is in fact tight.

\begin{lemma}
\label{lem::N+n+1_lower_bound}
For any $n\geq 1$ and $N\geq 3$,
there exists an instance of the dual DRO problem~\eqref{eq::dual_DRO} satisfying Assumption~\ref{asp::regularity} such that any optimal least-favorable distribution $\QQ^\star_{\mathrm{dual}}$ requires a support of at least $\min\{N + n + 1, KN\}$ distinct atoms.
\end{lemma}
The proof is provided in Appendix~\ref{app:N+n+1-proof}, and proceeds via an explicit, nontrivial construction of empirical samples and an associated loss function, carefully engineered so that every optimal least-favorable distribution must spread its mass across at least $\min\{N + n + 1,\, KN\}$ atoms. We believe this construction to be of independent interest. By constructing an explicit instance satisfying Assumption~\ref{asp::regularity}, the lemma shows that the coupling between the continuous decision variable $x$ and the adversarial distribution $\bQ$ necessitates exactly $N+n+1$ atoms when $K$ is sufficiently large, confirming that least-favorable distributions with an optimal primal decision are provably denser than worst-case distributions with a fixed primal decision.

\begin{remark}
    Our proposed best-response framework directly computes an approximate saddle-point pair $(x^\star, \QQ^\star_{\mathrm{dual}})$. One might wonder if an alternative two-step procedure is viable: first compute the optimal least-favorable distribution $\QQ^\star_{\mathrm{dual}}$ by solving the dual DRO problem~\eqref{eq::dual_DRO}, and then recover the robust primal decision $x^\star$ by simply solving $\min_{x \in \cX} \EE_{z \sim \QQ^\star_{\mathrm{dual}}}[\ell(x, z)]$. 
    In general, this sequential approach is flawed. While the saddle-point conditions guarantee that any robust primal minimizer must also minimize the expected loss under $\QQ^\star_{\mathrm{dual}}$, the converse is not necessarily true, that is,
    \begin{equation*}
        \argmin_{x \in \cX} \max_{\bQ \in \cP}
        \EE_{z \sim \bQ}[\ell(x, z)]
        \;\subseteq\;
        \argmin_{x \in \cX}
        \EE_{z \sim \QQ^\star_{\mathrm{dual}}}[\ell(x, z)].
    \end{equation*}
    Since this inclusion can be strict, minimizing against $\QQ^\star_{\mathrm{dual}}$ alone may yield solutions that are strictly suboptimal for the overall primal DRO problem. This two-step recovery is only mathematically guaranteed to succeed when the minimizer of $x \mapsto \EE_{z \sim \QQ^\star_{\mathrm{dual}}}[\ell(x, z)]$ is unique, a condition satisfied, for example, if the loss $\ell(\cdot, z)$ is strictly convex on $\cX$ for every $z \in \cZ$. This two-step procedure, together with the additional uniqueness requirement, is the predominant algorithmic framework in the DRO literature \cite{shafieezadeh2018wasserstein,nguyen2023bridging,taskesen2023distributionally,tacskesen2025optimality,sheriff2025nonlinear}. 
\end{remark}

\subsection{Compression to (N+n+1)-Point Least-Favorable Distribution}
\label{subsec::compression_dual}

Proposition~\ref{prop::dual-nonconvex-equivalence} guarantees the existence of a least-favorable distribution supported on at most $\min\{N + n + 1, KN\}$ atoms. However, computing it directly is challenging: existing constructions for $KN$-point distributions \cite[Theorem~2]{shafiee2025nash} rely on expensive large-scale conic reformulations, and achieving the tighter $N+n+1$ bound remains an open problem when $n < (K-1)N - 1$. To bypass these computational hurdles, we design an efficient post-processing step that compresses the time-averaged distribution $\bar{\bQ}_T$ from Algorithm~\ref{alg::1} into a sparse equivalent, without sacrificing the duality gap established in Theorem~\ref{thm::convergence_primal}.

Recall that at each iteration $t$ of Algorithm~\ref{alg::1}, the inner worst-case distribution $\bQ_t$ is computed and subsequently compressed to at most $N+1$ atoms (as detailed in Section~\ref{subsec::support_compression}). While that inner compression is highly efficient (reducing to a Fractional Knapsack Problem solvable via a simple $O(N \log N)$ greedy algorithm), the time-averaged distribution $\bar{\bQ}_T = \frac{1}{T}\sum_{t=1}^T\bQ_t$ still accumulates up to $O(TN)$ atoms over the algorithm's execution. To compress this massive aggregated distribution down to $N+n+1$ points, a greedy strategy is no longer sufficient due to the coupled primal equilibrium constraints. Instead, we must solve a structured convex program.

We achieve this by restricting the dual reformulation~\eqref{eq::dual_DRO_reformulation} so that the adversary can only allocate mass to previously discovered historical atoms. For each $i \in [N]$, let $\cS_i$ denote the finite set of all adversarial spatial locations generated for the empirical sample $\hat{z}_i$ over the entire execution of Algorithm~\ref{alg::1}. Introducing allocation weights $\alpha_{izk} \geq 0$ representing the fraction of probability mass transported from $\hat{z}_i$ to the historical atom $z \in \cS_i$ and evaluated against the $k$-th loss component, we obtain the restricted dual program:

\begin{equation}
\label{eq::restricted_dual_DRO}
\left\{
\begin{array}{cl}
    \max & \displaystyle - \frac{1}{N}\sum_{i=1}^{N} \sum_{z \in \cS_i} \sum_{k=1}^{K} \alpha_{izk} \ell_k^{*1}\left( y_{izk}, z \right) - \nu \sigma_{\cX}\left(\theta \right) \\[3ex]
    \mathrm{s.t.} & \alpha_{izk} \in \bR_+, \ y_{izk} \in \RR^n,\ \theta \in \RR^n, \ \nu \in \RR_+ \\[0.5ex]
    &\displaystyle \frac{1}{N} \sum_{i=1}^{N} \sum_{z \in \cS_i} \sum_{k=1}^{K} \alpha_{izk} y_{izk}  + \nu \theta = 0, \ \frac{1}{N} \sum_{i=1}^N \sum_{z \in \cS_i} \sum_{k=1}^K \alpha_{izk} c(z, \hat{z}_i) \leq \rho, \\
    & \displaystyle \sum_{z \in \cS_i} \sum_{k=1}^K \alpha_{izk} = 1, \, \forall i\in [N]
\end{array}
\right.
\end{equation}

Since the spatial locations $z$ inside the conjugate functions $\ell_k^{*1}$ are now fixed, this restricted formulation is readily convexified. Applying the variable substitutions $\lambda_{izk} = \alpha_{izk} y_{izk}$ and $\lambda_0 = \nu \theta$ linearizes the equilibrium constraint into $\frac{1}{N} \sum_{i=1}^N \sum_{z \in \cS_i} \sum_{k=1}^K \lambda_{izk} + \lambda_0 = 0$. The objective terms correspondingly transform into $-\alpha_{izk} \ell_k^{*1}(\lambda_{izk} / \alpha_{izk}, z)$ and $-\nu \sigma_{\cX}(\lambda_0 / \nu)$. As these are the negative perspective functions of the convex conjugates and the support function, they are jointly concave. While compressing the full maximin problem relies on an off-the-shelf convex (or linear) solver rather than a greedy heuristic, solving this finite-dimensional restricted program remains highly tractable and directly yields a highly compressed least-favorable distribution $\tilde{\bQ}$.

\begin{theorem}
\label{thm::dual_compression}
Let $(\bar{x}_T, \bar{\bQ}_T)$ be the output of Algorithm~\ref{alg::1}. Given an optimal solution $\{\alpha^\star_{izk}, y^\star_{izk}, \theta^\star, \nu^\star\}$ to the restricted dual reformulation~\eqref{eq::restricted_dual_DRO}, the discrete distribution
$$
\tilde{\bQ} := \frac{1}{N} \sum_{i=1}^N \sum_{z \in \cS_i} \sum_{k=1}^K \alpha_{izk}^\star \delta_{z}
$$
belongs to the ambiguity set $\cP$, and $(\bar{x}_T, \tilde{\bQ})$ does not increase the duality gap of $(\bar{x}_T, \bar{\bQ}_T)$. Furthermore, there exists an optimal solution of~\eqref{eq::restricted_dual_DRO} in which at most $N+n+1$ of the weights $\{\alpha_{izk}^\star\}$ are nonzero, restricting the support size of $\tilde{\bQ}$ to at most $N+n+1$.
\end{theorem}

\begin{proof}
The feasibility $\tilde{\bQ} \in \cP$ is guaranteed because the constraints in~\eqref{eq::restricted_dual_DRO} strictly enforce marginal matching for each empirical sample $\hat{z}_i$ and bound the total optimal transport cost by $\rho$.

To establish duality gap preservation, it suffices to show that $\min_{x \in \cX} f(x, \tilde{\bQ}) \geq \min_{x \in \cX} f(x, \bar{\bQ}_T)$. Let $\cP_{\cS} \subseteq \cP$ denote the ambiguity set restricted to the historical support $\bigcup_{i=1}^N \cS_i$. Analogous to Proposition~\ref{prop::dual-nonconvex-equivalence}, the optimal value of the restricted program~\eqref{eq::restricted_dual_DRO} is exactly equal to the optimal value of the restricted minimax problem $\max_{\bQ \in \cP_{\cS}} \min_{x \in \cX} f(x, \bQ)$. Observe that the time-averaged distribution $\bar{\bQ}_T$ is supported exclusively on this history, meaning $\bar{\bQ}_T \in \cP_{\cS}$. Consequently, its mass allocation can be mapped directly to a feasible set of weights $\{\bar{\alpha}_{izk}\}$ in~\eqref{eq::restricted_dual_DRO}. Because the restricted program~\eqref{eq::restricted_dual_DRO} evaluates the exact maximum over the restricted ambiguity set $\cP_{\cS}$, the optimal value achieved by $\tilde{\bQ}$ must be at least as large as the value attained by the feasible, sub-optimal allocation corresponding to $\bar{\bQ}_T$. Therefore, we have
$$
\min_{x \in \cX} f(x, \tilde{\bQ}) = \max_{\bQ \in \cP_{\cS}} \min_{x \in \cX} f(x, \bQ) \geq \min_{x \in \cX} f(x, \bar{\bQ}_T).
$$
Finally, the sparsity bound emerges directly from the geometric arguments used in Proposition~\ref{prop::dual-nonconvex-equivalence}. By fixing all continuous parameters to their optimal values $\{y_{izk}^\star, \theta^\star, \nu^\star\}$, the remaining maximization over the probability weights $\{\alpha_{izk}\}$ reduces to a linear program (LP). This LP features exactly $N$ equality constraints for the marginals, $n$ equality constraints for the dual decision variables $\{y_{izk}^\star\}$, and $1$ inequality constraint for the total transport budget. By the Fundamental Theorem of Linear Programming, there exists a Basic Feasible Solution for this system that possesses at most $N+n+1$ strictly positive variables $\alpha_{izk}^\star$.
\end{proof}

\begin{remark}
    One might be tempted to simply discard the aggregated distribution $\bar{\bQ}_T$ and compute an $N+1$-point worst-case distribution (best response) specifically against the final primal iterate $\bar{x}_T$. However, in a minimax game, a pure best response against a single primal decision is often highly exploitable by other decisions $x \in \cX$. Consequently, while this $N+1$ distribution correctly evaluates the worst-case risk at $\bar{x}_T$, it fails to act as a global least-favorable distribution, and replacing $\bar{\bQ}_T$ with it would severely degrade the theoretical duality gap. 
\end{remark}

To properly compress the adversary's strategy to at most $N+n+1$ points, the approach in Theorem~\ref{thm::dual_compression} relies on a two-step procedure: first, solving a large-scale convex program to evaluate the exact conjugate functions $\ell_k^{*1}$ and identify the optimal dual variables $\{s^\star_{izk}, y^\star_{izk}, \theta^\star, \nu^\star\}$; and second, fixing these dual variables to solve the resulting LP in~\eqref{eq::restricted_dual_DRO} that extracts the sparse mass allocation. This two-step process is computationally demanding because the first stage essentially aims to exactly model the inner minimization $\min_x f(x, \bar{\bQ}_T)$ in order to preserve the duality gap.

In the following, we further improve this computational cost by providing a direct one-stage approach when each loss component $\ell_k(\cdot, z)$ is continuously differentiable and $\kappa$-smooth for every $k \in [K]$. We further assume that the feasible set $\cX \subset \RR^n$ is compact, and let $D := \max_{x, x' \in \cX} \|x - x'\|_2$ denote the diameter of the domain. The key idea is that instead of solving a conic problem in the first stage to perfectly model $\min_x f(x, \bar{\bQ}_T)$, we use the algorithm's output $\bar{x}_T$ as an approximate solution and replace the first stage entirely. By substituting the loss with its first-order Taylor approximation at $\bar{x}_T$, we reduce the compression task to a much simpler single-stage problem.

Recall that $\cS_i$ denote the finite set of historical atoms generated for the empirical sample $\hat{z}_i$ during Algorithm~\ref{alg::1}. We formulate the direct tangent-based compression program as:

\begin{equation}
\label{eq::tangent_program}
\left\{
\begin{array}{cl}
    \max & \displaystyle \frac{1}{N}\sum_{i=1}^{N} \sum_{z \in \cS_i} \sum_{k=1}^K \alpha_{izk} \ell_k(\bar{x}_T, z) - \sigma_{\cX}(\theta) + \theta^\top \bar{x}_T \\[3ex]
    \text{s.t.} & \alpha_{izk} \in \bR_+, \ \theta \in \RR^n \\[0.5ex]
    &\displaystyle \displaystyle \theta + \frac{1}{N} \sum_{i=1}^{N} \sum_{z \in \cS_i} \sum_{k=1}^{K} \alpha_{izk} \nabla_x \ell_k(\bar{x}_T, z) = 0, \ \frac{1}{N} \sum_{i=1}^N \sum_{z \in \cS_i} \sum_{k=1}^K \alpha_{izk} c(z, \hat{z}_i) \leq \rho, \\
    &\displaystyle \sum_{z \in \cS_i} \sum_{k=1}^K \alpha_{izk} = 1, \, \forall i\in [N].
\end{array}
\right.
\end{equation}
Since $\sigma_{\cX}(\cdot)$ is convex, \eqref{eq::tangent_program} is a standard concave maximization problem. Crucially, when $\cX$ is a polyhedron defined by $\{x : Ax \le b\}$, we have $\sigma_{\cX}(\theta) = \min_{\mu \ge 0 : A^\top \mu = \theta} b^\top \mu$. Substituting this dual representation transforms the objective penalty $-\sigma_{\cX}(\theta) + \theta^\top \bar{x}_T$ into $\max_{\mu \ge 0} \{-\mu^\top(b - A\bar{x}_T)\}$, which reduces~\eqref{eq::tangent_program} directly into a pure LP.

\begin{theorem}
\label{thm::dual_compression_tangent}
    Suppose Assumption~\ref{asp::regularity} holds, the feasible set $\cX$ is compact with diameter $D$, and each $\ell_k(\cdot, z)$ is $\kappa$-smooth on $\cX$. Let $(\bar{x}_T, \bar{\bQ}_T)$ be the output of Algorithm~\ref{alg::1}. Given an optimal solution $(\alpha^\star, \theta^\star)$ to~\eqref{eq::tangent_program}, the discrete distribution
    $$
    \tilde{\bQ} := \frac{1}{N} \sum_{i=1}^N \sum_{z \in \cS_i} \left( \sum_{k=1}^K \alpha_{izk}^\star \right) \delta_{z}
    $$
    belongs to the ambiguity set $\cP$. Furthermore, $\tilde{\bQ}$ bounds the duality gap such that 
    $$ \mathrm{Gap}(\bar{x}_T, \tilde{\bQ}) \leq \min\left\{ D\sqrt{2\kappa \cdot \mathrm{Gap}(\bar{x}_T, \bar{\bQ}_T)}, \; \mathrm{Gap}(\bar{x}_T, \bar{\bQ}_T) + \frac{\kappa D^2}{2} \right\}.$$ 
    Finally, there exists an optimal solution possessing at most $N+n+1$ strictly positive weights $\alpha_{izk}^\star$, restricting the support size of $\tilde{\bQ}$ to at most $N+n+1$ atoms.
\end{theorem}

\begin{proof}
The marginal and budget constraints in~\eqref{eq::tangent_program} explicitly ensure that $\tilde{\bQ} \in \cP$. To establish the sparsity bound, consider the problem after the optimal dual variable $\theta^\star$ has been determined. Fixing $\theta = \theta^\star$, the remaining maximization over $\alpha \ge 0$ is an LP with exactly $N$ normalization rows, $1$ budget row, and $n$ equilibrium rows. By the Fundamental Theorem of Linear Programming, there exists an optimal basic feasible solution for this subproblem with at most $N+n+1$ strictly positive variables $\alpha^\star_{izk}$, thereby limiting the support of $\tilde{\bQ}$ to at most $N+n+1$ atoms.

To bound the duality gap, we must lower bound the worst-case risk $\min_{x \in \cX} f(x, \tilde{\bQ})$. For any fixed $x \in \cX$, the convexity of $\ell_k(\cdot, z)$ yields the standard gradient inequality:
$$
\ell_k(x, z) \ge \ell_k(\bar{x}_T, z) + \nabla_x \ell_k(\bar{x}_T, z)^\top (x - \bar{x}_T).
$$
Because $f(x, \tilde{\bQ}) = \frac{1}{N}\sum_{i,z} \left(\sum_k \alpha^\star_{izk}\right) \max_k \ell_k(x,z) \ge \frac{1}{N}\sum_{i,z,k} \alpha^\star_{izk} \ell_k(x, z)$, we apply the gradient inequality to obtain:
$$
f(x, \tilde{\bQ}) \ge \frac{1}{N} \sum_{i,z,k} \alpha_{izk}^\star \ell_k(\bar{x}_T, z) + \left( \frac{1}{N} \sum_{i,z,k} \alpha_{izk}^\star \nabla_x \ell_k(\bar{x}_T, z) \right)^{\!\!\top} \!(x - \bar{x}_T).
$$
Substituting the equilibrium constraint from~\eqref{eq::tangent_program}, we have $\frac{1}{N} \sum \alpha_{izk}^\star \nabla_x \ell_k(\bar{x}_T, z) = -\theta^\star$. Thus, the trailing term becomes $-\theta^{\star\top} (x - \bar{x}_T) = \theta^{\star\top} \bar{x}_T - \theta^{\star\top} x$. By definition of the support function, $\theta^{\star\top} x \le \sigma_{\cX}(\theta^\star)$, meaning $-\theta^{\star\top} x \ge -\sigma_{\cX}(\theta^\star)$ for all $x \in \cX$. Therefore:
$$
f(x, \tilde{\bQ}) \ge \frac{1}{N} \sum_{i,z,k} \alpha_{izk}^\star \ell_k(\bar{x}_T, z) - \sigma_{\cX}(\theta^\star) + \theta^{\star\top} \bar{x}_T =: J^\star,
$$
which implies $\min_{x \in \cX} f(x, \tilde{\bQ}) \ge J^\star$, where $J^\star$ is the optimal value of~\eqref{eq::tangent_program}. 

We now lower bound the optimal value $J^\star$. Let $\hat{\ell}_k(x, z) := \ell_k(\bar{x}_T, z) + \nabla_x \ell_k(\bar{x}_T, z)^\top (x - \bar{x}_T)$ denote the tangent loss, and let $\hat{\ell}(x, z) = \max_k \hat{\ell}_k(x, z)$. We define the expected tangent loss as $\hat{f}(x, \bQ) := \EE_{z \sim \bQ}[\hat{\ell}(x, z)]$. By substituting the definition of $\sigma_{\cX}(\theta) = \sup_{x \in \cX} \theta^\top x$ and the equilibrium constraint $\theta = -\frac{1}{N} \sum \alpha_{izk} \nabla_x \ell_k(\bar{x}_T, z)$ back into the objective, the penalty term evaluates exactly to $\min_{x \in \cX} \langle \frac{1}{N} \sum \alpha_{izk} \nabla_x \ell_k(\bar{x}_T, z), x - \bar{x}_T \rangle$. Thus, we may conclude
$$ J^\star = \max_{\alpha \in \cW} \min_{x \in \cX} \textstyle \left\{ \frac{1}{N}\sum_{i,z,k} \alpha_{izk} \hat{\ell}_k(x, z) \right\},$$
where $\cW := \{\alpha \ge 0 : \frac{1}{N} \sum_{i=1}^N \sum_{z \in \cS_i} \sum_{k=1}^K \alpha_{izk} c(z, \hat{z}_i) \leq \rho,\, \sum_{z \in \cS_i} \sum_{k=1}^K \alpha_{izk} = 1, \, \forall i\in [N]\}$. Since $\cW$ and $\cX$ are compact convex sets and the objective is bilinear, Sion's Minimax Theorem enables us to swap the operators to $\min_x \max_\alpha$. For a fixed $x$, the inner maximum assigns probability mass to the component $k$ that maximizes the tangent. This inner maximum evaluates exactly to $\max_{\bQ \in \cP_{\cS}} \hat{f}(x, \bQ)$. Since $\bar{\bQ}_T \in \cP_{\cS}$, this implies:
$$ J^\star = \min_{x \in \cX} \max_{\bQ \in \cP_{\cS}} \hat{f}(x, \bQ) \ge \min_{x \in \cX} \hat{f}(x, \bar{\bQ}_T).$$
We bound $\min_x \hat{f}(x, \bar{\bQ}_T)$ by sandwiching it against the true expected loss $f(x, \bar{\bQ}_T)$. Since $\hat{\ell}_k \le \ell_k \le \ell$, we globally have $\hat{f}(x, \bar{\bQ}_T) \le f(x, \bar{\bQ}_T)$. At the terminal point $\bar{x}_T$, the tangents are exact: $\hat{f}(\bar{x}_T, \bar{\bQ}_T) = f(\bar{x}_T, \bar{\bQ}_T)$. Furthermore, the $\kappa$-smoothness of $\ell_k(\cdot, z)$ yields the descent lemma bound $\ell_k(x, z) \le \hat{\ell}_k(x, z) + \frac{\kappa}{2}\|x - \bar{x}_T\|_2^2$, which after taking the maximum over $k$ and expectations gives 
$$
f(x, \bar{\bQ}_T) \le \hat{f}(x, \bar{\bQ}_T) + \frac{\kappa}{2}\|x - \bar{x}_T\|_2^2.
$$
Let $\zeta_T := \mathrm{Gap}(\bar{x}_T, \bar{\bQ}_T) = \max_{\bQ \in \cP} f(\bar{x}_T, \bQ) - \min_{x \in \cX} f(x, \bar{\bQ}_T)$. Let $\hat{x} \in \argmin_{x \in \cX} \hat{f}(x, \bar{\bQ}_T)$. For any parameter $s \in (0, 1]$, define $x_s = (1-s)\bar{x}_T + s\hat{x} \in \cX$. Using the sandwich bounds and the convexity of $\hat{f}(\cdot, \bar{\bQ}_T)$, we obtain:
\begin{align*}
    \min_{x \in \cX} f(x, \bar{\bQ}_T) \le f(x_s, \bar{\bQ}_T) &\le \hat{f}(x_s, \bar{\bQ}_T) + \frac{\kappa s^2}{2}\|\hat{x} - \bar{x}_T\|_2^2 \\
    &\le (1-s)\hat{f}(\bar{x}_T, \bar{\bQ}_T) + s\hat{f}(\hat{x}, \bar{\bQ}_T) + \frac{\kappa s^2 D^2}{2}.
\end{align*}
Since $f(\bar{x}_T, \bar{\bQ}_T) \le \max_{\bQ \in \cP} f(\bar{x}_T, \bQ)$, we have $\hat{f}(\bar{x}_T, \bar{\bQ}_T) = f(\bar{x}_T, \bar{\bQ}_T) \le \min_{x \in \cX} f(x, \bar{\bQ}_T) + \zeta_T$. Substituting this into the inequality and isolating $\hat{f}(\hat{x}, \bar{\bQ}_T) = \min_{x \in \cX} \hat{f}(x, \bar{\bQ}_T)$ yields:
$$
\min_{x \in \cX} \hat{f}(x, \bar{\bQ}_T) \ge \min_{x \in \cX} f(x, \bar{\bQ}_T) - \frac{1-s}{s}\zeta_T - \frac{\kappa sD^2}{2}.
$$
Consequently, the worst-case risk of the compressed distribution satisfies:
$$ \min_{x \in \cX} f(x, \tilde{\bQ}) \ge J^\star \ge \min_{x \in \cX} f(x, \bar{\bQ}_T) - \left( \frac{\zeta_T}{s} + \frac{\kappa sD^2}{2} \right) + \zeta_T. $$
Minimizing the subtractive expression over $s \in (0, 1]$ directly yields the bound stated in the theorem.
\end{proof}

\begin{remark}
    If the feasible region $\cX$ is not globally compact, Assumption~\ref{asp::regularity} guarantees that the expected loss is inf-compact. Consequently, the relevant sublevel sets containing the primal minimizers are strictly bounded. Without loss of generality, we can define a compact effective domain $\cX_T := \cX \cap \{x : \|x - \bar{x}_T\|_\infty \le R\}$ for a sufficiently large radius $R > 0$. Since restricting the problem to a region containing the global minimizers does not change the optimal value of the primal DRO problem, one can simply replace $\cX$ with $\cX_T$ in formulation~\eqref{eq::tangent_program}. The support function becomes $\sigma_{\cX_T}(\theta)$, the loss is restricted to be $\kappa$-smooth on $\cX_T$, and the diameter is bounded by $D \le 2R\sqrt{n}$. When $\cX$ is polyhedral, $\cX_T$ is also a polyhedron, ensuring the post-processing step remains a pure LP.
\end{remark}

\begin{remark}
    The $O(\sqrt{\mathrm{Gap}(\bar{x}_T, \bar{\bQ}_T)})$ bound in Theorem~\ref{thm::dual_compression_tangent} is not an artifact of the proof, but reflects the underlying curvature of the relaxed objective. However, when the loss components are piecewise-affine in $x$ (i.e., $\kappa=0$), the tangent model becomes exact. In this regime, the approximation error vanishes, and the single convex program~\eqref{eq::tangent_program} compresses the distribution to $N+n+1$ points while perfectly preserving the exact duality gap.
\end{remark}


\section{Numerical Experiments}
\label{sec::numerics}

In this section, we evaluate the runtime, accuracy, and scalability of our proposed algorithms on synthetic instances involving piecewise quadratic loss functions. Our numerical experiments consist of two parts. First, we isolate the inner worst-case expectation problem to demonstrate the efficiency of our budget allocation algorithm (Algorithm~\ref{alg::master_eval}) relative to state-of-the-art commercial solvers, namely Gurobi and MOSEK. Building on the scalability of this inner oracle, we then evaluate the Distributional Best-Response algorithm (Algorithm~\ref{alg::1}) on the primal DRO problem~\eqref{eq::primal_DRO}, and further examine the additional benefits of solving the dual DRO problem~\eqref{eq::dual_DRO} in improving both the accuracy and sparsity of the resulting worst-case/least-favorable distributions. The Python implementation of our algorithms and the code used to generate the results in this paper were run on a MacBook Pro (Apple M4 chip, 16GB RAM), and are publicly available at: \url{https://github.com/Christ1anChen/OT-DRO}

\subsection{Data Generation and Implementation Details}
\label{subsec::data_generation}

In all experiments, we take the uncertainty set to be $\mathcal{Z} = \mathbb{R}^m$ and the feasible region to be the $\ell_1$-norm ball $\mathcal{X} = \{x \in \mathbb{R}^{n}: \|x\|_1 \leq R\}$ with radius $R = 100$, and we set the decision and uncertainty dimensions equal, $m = n$. The transportation cost is the Euclidean distance $c(z, \hat{z}) = \|z - \hat{z}\|_2$, and the ambiguity radius is fixed at $\rho = 0.1$.

The empirical dataset comprises $N$ independent and identically distributed (i.i.d.) samples $\{\hat{z}_i\}_{i=1}^{N} \subset \mathcal{Z}$, drawn from a normal distribution with a randomly generated nonzero mean. Specifically, we first fix a global mean vector $\hat{\mu} \sim \mathcal{N}(0, I_m)$, and conditioned on $\hat{\mu}$, each empirical sample is generated as $\hat{z}_i \sim \mathcal{N}(\hat\mu, I_m)$. The empirical samples are redrawn for every run of the experiments, while the global mean vector is redrawn only when the dimension changes.

The component loss functions $\ell_k$ for the primal DRO problem take the following piecewise quadratic form for $k \in [K]$:
$$
\ell_k(x, z) = x^\top C_k x + z^\top B_k x - z^\top A_k z.
$$
To guarantee that each component loss $\ell_k(x,z)$ is convex in $x$ and concave in $z$, we randomly draw standard normal matrices $X_A, X_C \in \mathbb{R}^{m \times m}$ and construct the positive definite matrices as $A_k = \frac{1}{m} X_A^\top X_A + 0.01 I_m$ and $C_k = \frac{1}{m} X_C^\top X_C + 0.01 I_m$. The addition of $0.01 I_m$ ensures numerical stability during optimization. The bilinear coupling matrix is generated as $B_k = X_B$, where $X_B$ is an independent standard normal matrix used to shift the cross-terms. 

For isolated evaluations of the inner worst-case expectation solvers (Section~\ref{subsec::numeric_worst_case}), we evaluate the environment using a fixed nominal primal decision variable $x_{\text{nom}} \sim \mathcal{N}(0, I_m)$. Under this condition, the loss components reduce to quadratic forms $\ell_k(z) = c_k + b_k^\top z - z^\top A_k z$, with constants and linear coefficients defined as $c_k = x_{\text{nom}}^\top C_k x_{\text{nom}}$ and $b_k = B_k x_{\text{nom}}$, while $A_k$ remains unchanged.

\subsection{Worst-case Expectation Problem}
\label{subsec::numeric_worst_case}

We first evaluate the computational efficiency of the proposed budget allocation algorithm (Algorithm~\ref{alg::master_eval}) for solving the inner worst-case expectation problem~\eqref{eq::worst_case}. Across all experiments, the tolerance $\eta$ and oracle precision $\epsilon$ of Algorithm~\ref{alg::master_eval} are universally set to $10^{-3}$. We benchmark our method against state-of-the-art commercial solvers, Gurobi and MOSEK, which compute the approximate global optimum by reformulating the inner quadratic loss problem into the following second-order cone program (SOCP) \cite{mohajerin2018data}: 
$$
\left\{
\begin{array}{cl}
    \max & \displaystyle \frac{1}{N} \sum_{i=1}^N \sum_{k=1}^K \left[ \alpha_{ik} \left( c_k + b_k^\top \hat{z}_i - \hat{z}_i^\top A_k \hat{z}_i \right) + \left(b_k - 2 A_k \hat{z}_i \right)^\top q_{ik} - t_{ik} \right] \\[3.0ex]
    \text{s.t.} & \alpha_{ik}, t_{ik} \in \mathbb{R}_+, \ q_{ik} \in \mathbb{R}^m \\[0.0ex]
    &\displaystyle t_{ik}\alpha_{ik} \ge \norm{A_k^{1/2} q_{ik}}_2^2,\,\forall i\in [N], \forall k\in[K],\ \frac{1}{N} \sum_{i=1}^N \sum_{k=1}^K \norm{q_{ik}}_2 \leq \rho, \\[-1.0ex]
    &\displaystyle \sum_{k=1}^K \alpha_{ik} = 1, \, \forall i\in [N].
\end{array}
\right.
$$
To systematically assess scalability, we track the runtime and objective value accuracy across three experimental settings: (1) varying the number of components $K \in \{2, 4, 6, 8, 10\}$ with fixed $N=10, m=500$; (2) varying the sample size $N \in \{10, 50, 100, 500, 1000\}$ with fixed $K=3, m=500$; and (3) varying the dimension $m \in \{10, 50, 100, 500, 1000\}$ with fixed $K=3, N=100$.

\begin{figure}[htbp]
\centering
\includegraphics[width=0.95\linewidth]{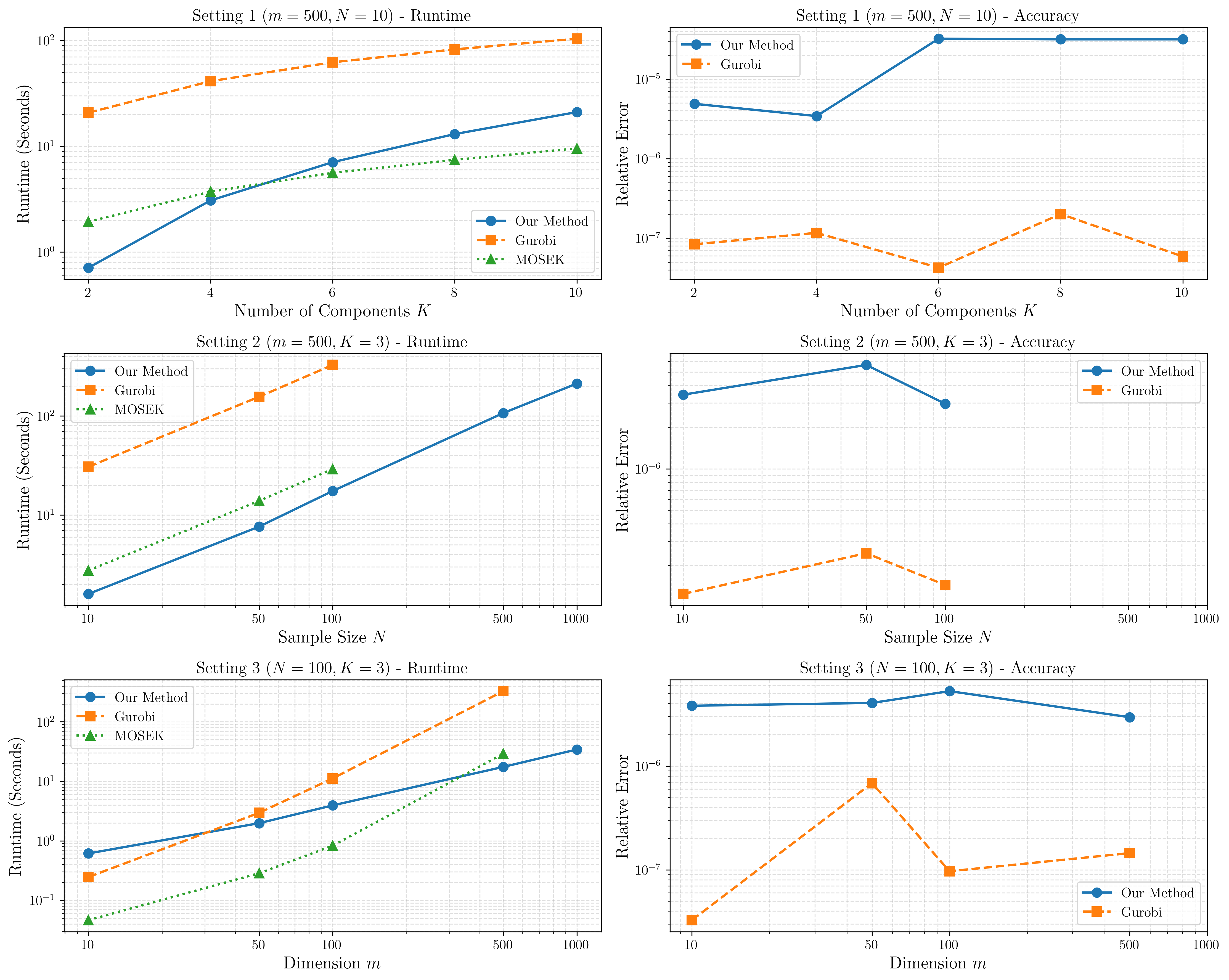}
\caption{Computational runtime and accuracy of the inner worst-case expectation solvers across varying problem parameters. \textbf{Left column:} Runtime (in seconds) for each method, shown on a logarithmic scale. \textbf{Right column:} Relative error of the optimal objective value relative to the baseline solver (MOSEK), shown on a logarithmic scale. Rows 1, 2, and 3 correspond to scaling with respect to the number of components $K$ ($m=500$, $N=10$), the sample size $N$ ($m=500$, $K=3$), and the data dimension $m$ ($N=100$, $K=3$), respectively. Missing lines indicate that MOSEK triggered an out-of-memory (OOM) error due to its dense conic lifting, or that Gurobi exceeded the $10$-minute wall-clock limit.}
\label{fig:inner_comp}
\end{figure}

The performance of the inner solvers is summarized in Figure~\ref{fig:inner_comp}. Here, our method leverages the budget allocation algorithm (Algorithm~\ref{alg::master_eval}) to solve the inner maximization~\eqref{eq::worst_case}, coupled with the sorting-based greedy algorithm (Section~\ref{subsec::support_compression}) to enforce post-processing support compression. Across all instances for which the solvers successfully terminate, we adopt MOSEK's output as the high-precision baseline to compute the relative error, defined as $|\text{Obj}_{\text{algo}} - \text{Obj}_{\text{MOSEK}}| / |\text{Obj}_{\text{MOSEK}}|$. As shown in the results, our algorithm computes an objective value that matches the baseline up to a relative error of approximately $10^{-5}$, verifying its strong numerical accuracy.

Regarding computational runtime, Setting 1 (Row 1) demonstrates that for moderate data dimensions ($m=500$) and a limited number of components ($K \leq 10$), the runtimes of MOSEK and our proposed method are comparable. Although MOSEK is highly optimized for these moderate-scale conic regimes, our algorithm remains highly competitive and well within practical limits. However, Settings 2 and 3 (Rows 2 and 3) highlight the critical impact of the overall problem scale, which acts as the primary computational bottleneck for commercial solvers. For large instances (e.g., $N=100$, $m=1000$ or $N=500$, $m=500$), commercial solvers fail entirely: MOSEK encounters an Out-Of-Memory (OOM) error due to the severe memory overhead required to construct and maintain $O(NKm^2)$ matrix blocks, while Gurobi exceeds the 10-minute wall-clock limit. In contrast, our proposed budget allocation algorithm completely avoids dense matrix lifting and scales gracefully with both the dimension $m$ and the sample size $N$. Crucially, for any fixed dual candidate $\lambda$, our budget allocation algorithm decomposes naturally across all $N$ empirical samples, an inherently parallelizable structure we introduce in Section~\ref{subsec::efficient_budget_allocation} and exploit in our implementation. By multi-threading across samples, our method maintains minimal memory overhead and efficiently solves large instances.

\subsection{Primal and Dual DRO Problems}

Having established the efficiency of the inner oracle, we now evaluate the practical efficacy of our Distributional Best-Response algorithm (Algorithm~\ref{alg::1}) for solving the full primal DRO problem~\eqref{eq::primal_DRO}. We set the number of loss components to $K=3$ and apply the identical tolerances to the inner budget allocation algorithm as specified in Section~\ref{subsec::numeric_worst_case}.

As a baseline for evaluation, we compute a nearly-exact global optimum of the primal DRO problem using MOSEK. Despite its accuracy, MOSEK relies on expensive-to-solve conic formulations, prohibiting its graceful scalability. Specifically, by dualizing the inner maximization problem, \cite{mohajerin2018data} show that the minimax problem can be reformulated as a single semidefinite program (SDP). To further enhance MOSEK's performance, we apply the Schur complement to convert the massive positive semidefinite constraints into a set of rotated second-order cones (RSOC). This yields the following equivalent formulation:
\begin{equation}
\label{eq::mosek_socp}
\left\{
\begin{array}{cl}
\min & \displaystyle \lambda \rho + \frac{1}{N} \sum_{i=1}^N s_i \\[3.0ex]
\text{s.t.} & \lambda,r_k \in \RR_+,\ s_i \in \RR,\ x, u \in \RR^n,\ w_{ik}\in \RR^n\\[1.5ex]
& x^\top C_k x \leq r_k, \, \forall k\in [K],\ 
\norm{w_{ik}}_2 \leq \lambda, \, \forall i \in [N], \forall k \in [K],\ -u \preceq x \preceq u, \ \sum_{j=1}^n u_j \leq R,\\[2.0ex]
& \frac{1}{4} \left( B_k x -2 A_k \hat{z}_i - w_{ik}\right)^\top A_k^{-1} \left( B_k x -2 A_k \hat{z}_i - w_{ik}\right) \leq s_i - r_k - \hat{z}_i^\top B_k x + \hat{z}_i^\top A_k \hat{z}_i.
\end{array}
\right.
\end{equation}
Against this baseline, we evaluate our Distributional Best-Response algorithm. The algorithm is initialized at the origin ($x_0 = \mathbf{0}$) and evaluated over $T = 50$ iterations, updating the primal decision variable $x$ via subgradient descent with a decaying learning rate of $\texttt{lr}_t = 0.2 / \sqrt{t}$. We employ a tail-averaging scheme where the first $20\%$ of the iterations act as a burn-in phase; the final output is extracted as the uniform average of the remaining trajectory, which prevents the high-variance early steps from degrading the final solution quality.
\begin{figure}[htbp]
    \centering
    \includegraphics[width=0.7\linewidth]{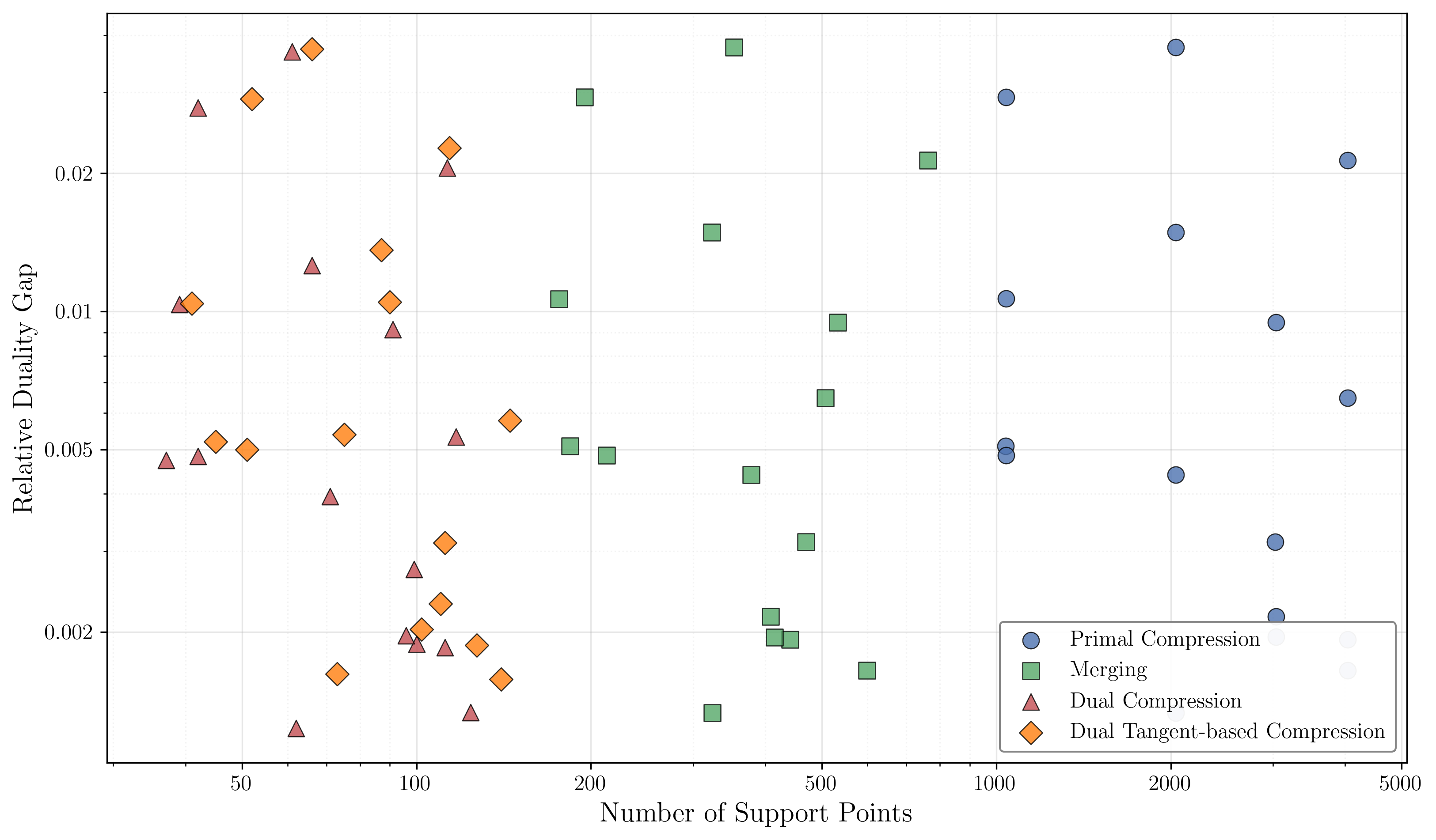}
    \caption{Relative duality gap versus support size across four post-processing methods: (1) primal compression (``Primal Compression''); (2) primal compression augmented by a merging heuristic (``Merging''); (3) dual compression (``Dual Compression''); and (4) dual tangent-based compression (``Dual Tangent-based Compression''). The figure displays results from $16$ distinct problem instances, generated using every combination of four sample sizes ($N \in \{25, 50, 75, 100\}$) and four dimensions ($n \in \{25, 50, 75, 100\}$). Each method is represented by $16$ scatter points, corresponding to its performance on each of these instances.}
    \label{fig:post_process_compare}
\end{figure}

To evaluate the performance of the compression methods proposed in Section~\ref{subsec::compression_dual}, we analyze the sparsity of the resulting distribution $\hat{\bQ}_T$ and its corresponding relative duality gap 
\[
\left| \frac{\mathrm{Gap}(\bar{x}_T, \hat{\bQ}_T)}{\min_{x \in \cX}\max_{\bQ \in \cP} \EE_{z\sim \bQ} [\ell(x,z)]} \right|
= \left| \frac{\max_{\bQ \in \cP} f(\bar x_T, \bQ) - \min_{x\in \cX} f(x, \hat{\bQ}_T)}{\min_{x \in \cX}\max_{\bQ \in \cP} f(x, \bQ)} \right|
\]
across four distinct post-processing methods: (1) primal compression based on the sorting-based greedy method discussed in Section~\ref{subsec::support_compression}; (2) primal compression augmented by a merging heuristic, which merges points within a $10^{-3}$ $\ell_2$-distance of each other; (3) dual compression by solving the restricted dual program~\eqref{eq::restricted_dual_DRO}; and (4) dual compression by solving the tangent-based formulation~\eqref{eq::tangent_program}.

We consider $16$ distinct problem instances, each solved via these four post-processing methods. These instances are generated using every combination of four sample sizes ($N \in \{25, 50, 75, 100\}$) and four dimensions ($n \in \{25, 50, 75, 100\}$). As shown in Figure~\ref{fig:post_process_compare}, the primal compression method yields support sizes of $1040, 2040, 3040, 4040$ for $N = 25, 50, 75, 100$, respectively. This matches our theory precisely: the Distributional Best-Response Algorithm (Algorithm~\ref{alg::1}) can produce distributions with up to $T(N+1)$ atoms. Figure~\ref{fig:post_process_compare} also demonstrates the benefit of the naive merging heuristic: the support sizes (averaged over the dimension $n$) reduce to $192, 344.5, 456.5, 577$ for $N = 25, 50, 75, 100$, respectively, substantially improving upon the previous method without worsening solution quality. As predicted by our theoretical analysis, the dual compression approaches yield the sparsest supports: the restricted dual program~\eqref{eq::restricted_dual_DRO} produces least-favorable distributions with average support sizes of $40, 65, 96.5, 116.5$ for $N = 25, 50, 75, 100$, while simultaneously slightly improving the duality gap. One downside of this method, however, is its computational cost: it requires first solving the primal DRO problem, whose solution is then used to formulate and solve the restricted dual compression program~\eqref{eq::restricted_dual_DRO}. For the largest instance ($N=100$, $n=100$), the primal DRO step takes $160.14$ seconds in $50$ iterations, and the subsequent dual compression step demands an additional $334.36$ seconds. The tangent-based formulation~\eqref{eq::tangent_program} eliminates this additional computational cost by reducing the compression step to a linear program, albeit with a slight deterioration in the optimality gap (as predicted by Theorem~\ref{thm::dual_compression_tangent}). For this same largest instance, the runtime of the compression step drops drastically to $0.29$ seconds.


\section*{Acknowledgments}
Salar Fattahi is supported, in part, by the NSF CAREER grant CCF-2337776 and ONR grant N00014-26-1-2074. Soroosh Shafiee is supported, in part, by the NSF CAREER grant ECCS-2541066.


\bibliography{references}

@article{diestel1977remarks,
  title={Remarks on weak compactness in L1 ($\mu$, X)},
  author={Diestel, J},
  journal={Glasgow Mathematical Journal},
  volume={18},
  number={1},
  pages={87--91},
  year={1977},
  publisher={Cambridge University Press}
}

@book{cormen2022introduction,
  title={Introduction to Algorithms},
  author={Cormen, Thomas H and Leiserson, Charles E and Rivest, Ronald L and Stein, Clifford},
  year={2022},
  publisher={MIT press}
}

@article{beck2003mirror,
  title={Mirror descent and nonlinear projected subgradient methods for convex optimization},
  author={Beck, Amir and Teboulle, Marc},
  journal={Operations Research Letters},
  volume={31},
  number={3},
  pages={167--175},
  year={2003},
  publisher={Elsevier}
}

@article{hazan2012projection,
  title={Projection-free online learning},
  author={Hazan, Elad and Kale, Satyen},
  journal={arXiv:1206.4657},
  year={2012}
}

@inproceedings{hazan2020faster,
  title={Faster projection-free online learning},
  author={Hazan, Elad and Minasyan, Edgar},
  booktitle={Conference on Learning Theory},
  pages={1877--1893},
  year={2020}
}

@article{hazan2007logarithmic,
  title={Logarithmic regret algorithms for online convex optimization},
  author={Hazan, Elad and Agarwal, Amit and Kale, Satyen},
  journal={Machine Learning},
  volume={69},
  number={2},
  pages={169--192},
  year={2007},
  publisher={Springer}
}

@inproceedings{zinkevich2003online,
  title={Online convex programming and generalized infinitesimal gradient ascent},
  author={Zinkevich, Martin},
  booktitle={International Conference on Machine Learning},
  pages={928--936},
  year={2003}
}

@article{shalev2025online,
  title={Online learning and online convex optimization},
  author={Shalev-Shwartz, Shai},
  journal={Foundations and Trends in Machine Learning},
  volume={4},
  number={2},
  pages={107--194},
  year={2012}
}

@book{boyd2004convex,
  title={Convex Optimization},
  author={Boyd, Stephen and Vandenberghe, Lieven},
  year={2004},
  publisher={Cambridge University Press}
}

@article{yue2022linear,
  title={On linear optimization over {W}asserstein balls},
  author={Yue, Man-Chung and Kuhn, Daniel and Wiesemann, Wolfram},
  journal={Mathematical Programming},
  volume={195},
  number={1},
  pages={1107--1122},
  year={2022},
  publisher={Springer}
}

@article{bubeck2015convex,
  title={Convex optimization: Algorithms and complexity},
  author={Bubeck, S{\'e}bastien},
  journal={Foundations and Trends in Machine Learning},
  volume={8},
  number={3-4},
  pages={231--357},
  year={2015},
  publisher={Now Publishers, Inc.}
}

@article{ben2015oracle,
  title={Oracle-based robust optimization via online learning},
  author={Ben-Tal, Aharon and Hazan, Elad and Koren, Tomer and Mannor, Shie},
  journal={Operations Research},
  volume={63},
  number={3},
  pages={628--638},
  year={2015},
  publisher={INFORMS}
}

@article{orabona2019modern,
  title={A Modern Introduction to Online Learning},
  author={Orabona, Francesco},
  journal={arXiv:1912.13213},
  year={2019}
}

@article{mohajerin2018data,
  title={Data-driven distributionally robust optimization using the {W}asserstein metric: Performance guarantees and tractable reformulations},
  author={Mohajerin Esfahani, Peyman and Kuhn, Daniel},
  journal={Mathematical Programming},
  volume={171},
  number={1},
  pages={115--166},
  year={2018},
  publisher={Springer}
}

@article{kuhn2025distributionally,
  title={Distributionally robust optimization},
  author={Kuhn, Daniel and Shafiee, Soroosh and Wiesemann, Wolfram},
  journal={Acta Numerica},
  volume={34},
  pages={579--804},
  year={2025},
  publisher={Cambridge University Press}
}

@article{shafiee2025nash,
  title={Nash Equilibria, Regularization, and Computation in Optimal Transport-Based Distributionally Robust Optimization},
  author={Shafiee, Soroosh and Aolaritei, Liviu and D{\"o}rfler, Florian and Kuhn, Daniel},
  journal={Operations Research},
  number={3},
  volume={74},
  pages={1689--1709},
  year={2026}
}

@article{aigner2023data,
    title={Data-driven distributionally robust optimization over time},
    author={Aigner, Kevin-Martin and B{\"a}rmann, Andreas and Braun, Kristin and Liers, Frauke and Pokutta, Sebastian and Schneider, Oskar and Sharma, Kartikey and Tschuppik, Sebastian},
    journal={INFORMS Journal on Optimization},
    volume={5},
    number={4},
    pages={376--394},
    year={2023},
    publisher={INFORMS}
}

@inproceedings{qi2021online,
    title={An online method for a class of distributionally robust optimization with non-convex objectives},
    author={Qi, Qi and Guo, Zhishuai and Xu, Yi and Jin, Rong and Yang, Tianbao},
    booktitle={Advances in Neural Information Processing Systems},
    pages={10067--10080},
    year={2021}
}

@article{ho2018online,
  title={Online first-order framework for robust convex optimization},
  author={Ho-Nguyen, Nam and K{\i}l{\i}n{\c{c}}-Karzan, Fatma},
  journal={Operations Research},
  volume={66},
  number={6},
  pages={1670--1692},
  year={2018},
  publisher={INFORMS}
}

@article{ho2019exploiting,
  title={Exploiting problem structure in optimization under uncertainty via online convex optimization},
  author={Ho-Nguyen, Nam and K{\i}l{\i}n{\c{c}}-Karzan, Fatma},
  journal={Mathematical Programming},
  volume={177},
  number={1},
  pages={113--147},
  year={2019},
  publisher={Springer}
}

@article{postek2024first,
  title={First-order algorithms for robust optimization problems via convex-concave saddle-point Lagrangian reformulation},
  author={Postek, Krzysztof and Shtern, Shimrit},
  journal={INFORMS Journal on Computing},
  volume={37},
  number={3},
  pages={557--581},
  year={2025},
  publisher={INFORMS}
}

@article{hazan2016introduction,
  title={Introduction to online convex optimization},
  author={Hazan, Elad},
  journal={Foundations and Trends in Optimization},
  volume={2},
  number={3-4},
  pages={157--325},
  year={2016},
  publisher={Emerald Publishing Limited}
}

@book{hazan2022introduction,
  title={Introduction to Online Convex Optimization},
  author={Hazan, Elad},
  year={2022},
  publisher={MIT Press}
}

@article{tu2024max,
  title={A Max-Min-Max Algorithm for Large-Scale Robust Optimization},
  author={Tu, Kai and Chen, Zhi and Yue, Man-Chung},
  journal={arXiv:2404.05377},
  year={2024}
}

@inproceedings{namkoong2016stochastic,
    title={Stochastic gradient methods for distributionally robust optimization with $f$-divergences},
    author={Namkoong, Hongseok and Duchi, John C},
    booktitle={Advances in Neural Information Processing Systems},
    pages={2216--2224},
    year={2016}
}

@article{nemirovski2009robust,
  title={Robust stochastic approximation approach to stochastic programming},
  author={Nemirovski, Arkadi and Juditsky, Anatoli and Lan, Guanghui and Shapiro, Alexander},
  journal={SIAM Journal on Optimization},
  volume={19},
  number={4},
  pages={1574--1609},
  year={2009},
  publisher={SIAM}
}

@book{rockafellar1998variational,
	title={Variational Analysis},
	author={Rockafellar, R Tyrrell and Wets, Roger JB},
	year={1998},
	publisher={Springer}
}

@article{nguyen2023bridging,
  title={Bridging {B}ayesian and minimax mean square error estimation via {W}asserstein distributionally robust optimization},
  author={Nguyen, Viet Anh and Shafieezadeh-Abadeh, Soroosh and Kuhn, Daniel and Mohajerin Esfahani, Peyman},
  journal={Mathematics of Operations Research},
  volume={48},
  number={1},
  pages={1--37},
  year={2023},
  publisher={INFORMS}
}

@article{shafieezadeh2019regularization,
    title={Regularization via Mass Transportation},
    author={Shafieezadeh-Abadeh, Soroosh and Kuhn, Daniel and Mohajerin Esfahani, Peyman},
    journal={Journal of Machine Learning Research},
    volume={20},
    number={103},
    pages={1--68},
    year={2019}
}

@inproceedings{shafieezadeh2018wasserstein,
    title={Wasserstein distributionally robust {K}alman filtering},
    author={Shafieezadeh-Abadeh, Soroosh and Nguyen, Viet Anh and Kuhn, Daniel and Mohajerin Esfahani, Peyman},
    booktitle={Advances in Neural Information Processing Systems},
    pages={8474--8483},
    year={2018}
}

@article{tacskesen2023semi,
  title={Semi-discrete optimal transport: Hardness, regularization and numerical solution},
  author={Ta{\c{s}}kesen, Bahar and Shafieezadeh-Abadeh, Soroosh and Kuhn, Daniel},
  journal={Mathematical Programming},
  volume={199},
  number={1},
  pages={1033--1106},
  year={2023},
  publisher={Springer}
}

@article{tacskesen2023discrete,
    title={Discrete optimal transport with independent marginals is \#{P}-hard},
    author={Ta{\c{s}}kesen, Bahar and Shafieezadeh-Abadeh, Soroosh and Kuhn, Daniel and Natarajan, Karthik},
    journal={SIAM Journal on Optimization},
    volume={33},
    number={2},
    pages={589--614},
    year={2023},
    publisher={SIAM}
}

@article{wozabal2012framework,
    title={A framework for optimization under ambiguity},
    author={Wozabal, David},
    journal={Annals of Operations Research},
    volume={193},
    number={1},
    pages={21--47},
    year={2012},
    publisher={Springer}
}

@article{zhen2023unified,
  title={A unified theory of robust and distributionally robust optimization via the primal-worst-equals-dual-best principle},
  author={Zhen, Jianzhe and Kuhn, Daniel and Wiesemann, Wolfram},
  journal={Operations Research},
  volume={73},
  number={2},
  pages={862--878},
  year={2025},
  publisher={INFORMS}
}

@article{blanchet2018optimal,
    title={Optimal transport-based distributionally robust optimization: Structural properties and iterative schemes},
    author={Blanchet, Jose and Murthy, Karthyek and Zhang, Fan},
    journal={Mathematics of Operations Research},
    volume={47},
    number={2},
    pages={1500--1529},
    year={2022}
}

@inproceedings{sinha2018certifying,
    title={Certifying Some Distributional Robustness with Principled Adversarial Training},
    author={Sinha, Aman and Namkoong, Hongseok and Duchi, John},
    booktitle={International Conference on Learning Representations},
    year={2018}
}

@article{gao2023distributionally,
  title={Distributionally robust stochastic optimization with {W}asserstein distance},
  author={Gao, Rui and Kleywegt, Anton},
  journal={Mathematics of Operations Research},
  volume={48},
  number={2},
  pages={603--655},
  year={2023},
  publisher={INFORMS}
}

@inproceedings{li2019first,
    title={A First-Order Algorithmic Framework for {W}asserstein Distributionally Robust Logistic Regression},
    author={Li, Jiajin and Huang, Sen and So, Anthony Man-Cho},
    booktitle={Advances in Neural Information Processing Systems},
    pages={3937--3947},
    year={2019}
}

@inproceedings{li2020fast,
    title={Fast epigraphical projection-based incremental algorithms for {W}asserstein distributionally robust support vector machine},
    author={Li, Jiajin and Chen, Caihua and So, Anthony Man-Cho},
    booktitle={Advances in Neural Information Processing Systems},
    pages={4029--4039},
    year={2020}
}

@article{cherukuri2020cooperative,
  title={Cooperative data-driven distributionally robust optimization},
  author={Cherukuri, Ashish and Cort{\'e}s, Jorge},
  journal={IEEE Transactions on Automatic Control},
  volume={65},
  number={10},
  pages={4400--4407},
  year={2019},
  publisher={IEEE}
}

@article{li2020data,
    title={Data assimilation and online optimization with performance guarantees},
    author={Li, Dan and Mart{\'\i}nez, Sonia},
    journal={IEEE Transactions on Automatic Control},
    volume={66},
    number={5},
    pages={2115--2129},
    year={2020}
}

@book{rockafellar1970convex,
    title={Convex Analysis},
    author={Rockafellar, R Tyrrell},
    year={1970},
    publisher={Princeton University Press}
}

@article{parikh2014proximal,
    title={Proximal Algorithms},
    author={Parikh, Neal and Boyd, Stephen},
    journal={Foundations and Trends in Optimization},
    volume={1},
    number={3},
    pages={127--239},
    year={2014},
    publisher={Now Publishers}
}

@book{beck2017first,
    title={First-Order Methods in Optimization},
    author={Beck, Amir},
    year={2017},
    publisher={SIAM}
}

@article{sion1958general,
    title={On general minimax theorems},
    author={Sion, Maurice},
    journal={Pacific Journal of Mathematics},
    volume={8},
    number={1},
    pages={171--176},
    year={1958}
}

@article{zorzi2016robust,
    title={Robust {K}alman filtering under model perturbations},
    author={Zorzi, Mattia},
    journal={IEEE Transactions on Automatic Control},
    volume={62},
    number={6},
    pages={2902--2907},
    year={2016},
    publisher={IEEE}
}

@article{zorzi2017robustness,
    title={On the robustness of the {B}ayes and {W}iener estimators under model uncertainty},
    author={Zorzi, Mattia},
    journal={Automatica},
    volume={83},
    pages={133--140},
    year={2017},
    publisher={Elsevier}
}

@article{levy2004robust,
    title={Robust least-squares estimation with a relative entropy constraint},
    author={Levy, Bernard C and Nikoukhah, Ramine},
    journal={IEEE Transactions on Information Theory},
    volume={50},
    number={1},
    pages={89--104},
    year={2004},
    publisher={IEEE}
}

@article{levy2012robust,
    title={Robust state space filtering under incremental model perturbations subject to a relative entropy tolerance},
    author={Levy, Bernard C and Nikoukhah, Ramine},
    journal={IEEE Transactions on Automatic Control},
    volume={58},
    number={3},
    pages={682--695},
    year={2012},
    publisher={IEEE}
}

@article{owhadi2017extreme,
    title={Extreme points of a ball about a measure with finite support},
    author={Owhadi, H and Scovel, C},
    journal={Communications in Mathematical Sciences},
    volume={15},
    number={1},
    pages={77--96},
    year={2017},
    publisher={International Press}
}

@article{tacskesen2025optimality,
  title={Optimality of Linear Policies in Distributionally Robust Linear Quadratic Control},
  author={Ta{\c{s}}kesen, Bahar and Iancu, Dan A and Ko{\c{c}}yi{\u{g}}it, {\c{C}}a{\u{g}}{\i}l and Kuhn, Daniel},
  journal={arXiv:2508.11858},
  year={2025}
}

@article{giang2026projection,
  title={Projection-Free Algorithms for Minimax Problems},
  author={Giang-Tran, Khanh-Hung and Shafiee, Soroosh and Ho-Nguyen, Nam},
  journal={arXiv:2603.29870},
  year={2026}
}

@inproceedings{boroun2023projection,
  title={Projection-free methods for solving nonconvex-concave saddle point problems},
  author={Boroun, Morteza and Yazdandoost Hamedani, Erfan and Jalilzadeh, Afrooz},
  booktitle={Advances in Neural Information Processing Systems},
  pages={53844--53856},
  year={2023}
}

@article{xu2023unified,
  title={A unified single-loop alternating gradient projection algorithm for nonconvex-concave and convex-nonconcave minimax problems},
  author={Xu, Zi and Zhang, Huiling and Xu, Yang and Lan, Guanghui},
  journal={Mathematical Programming},
  volume={201},
  pages={635--706},
  year={2023}
}

@article{nedic2009subgradient,
  title={Subgradient methods for saddle-point problems},
  author={Nedi{\'c}, Angelia and Ozdaglar, Asuman},
  journal={Journal of Optimization Theory and Applications},
  volume={142},
  number={1},
  pages={205--228},
  year={2009},
  publisher={Springer}
}

@article{beck2009duality,
  title={Duality in robust optimization: primal worst equals dual best},
  author={Beck, Amir and Ben-Tal, Aharon},
  journal={Operations Research Letters},
  volume={37},
  number={1},
  pages={1--6},
  year={2009},
  publisher={Elsevier}
}

@article{everett1963generalized,
  title={Generalized {L}agrange multiplier method for solving problems of optimum allocation of resources},
  author={Everett III, Hugh},
  journal={Operations Research},
  volume={11},
  number={3},
  pages={399--417},
  year={1963}
}

@book{nesterov1994interior,
  title={Interior-Point Polynomial Algorithms in Convex Programming},
  author={Nesterov, Yurii and Nemirovskii, Arkadii},
  year={1994},
  publisher={SIAM}
}

@article{wang2025sinkhorn,
  title={Sinkhorn distributionally robust optimization},
  author={Wang, Jie and Gao, Rui and Xie, Yao},
  journal={Operations Research},
  volume={74},
  number={3},
  pages={1581--1603},
  year={2026}
}

@article{azizian2023regularization,
  title={Regularization for {W}asserstein distributionally robust optimization},
  author={Azizian, Wa{\"\i}ss and Iutzeler, Franck and Malick, J{\'e}r{\^o}me},
  journal={ESAIM: Control, Optimisation and Calculus of Variations},
  volume={29},
  pages={33},
  year={2023},
  publisher={EDP Sciences}
}

@article{xu2024flow,
  title={Flow-based distributionally robust optimization},
  author={Xu, Chen and Lee, Jonghyeok and Cheng, Xiuyuan and Xie, Yao},
  journal={IEEE Journal on Selected Areas in Information Theory},
  volume={5},
  pages={62--77},
  year={2024},
  publisher={IEEE}
}

@article{wang2024regularization,
  title={Regularization for adversarial robust learning},
  author={Wang, Jie and Gao, Rui and Xie, Yao},
  journal={arXiv:2408.09672},
  year={2024}
}

@article{wang2025iterative,
  title={Iterative Sampling Methods for Sinkhorn Distributionally Robust Optimization},
  author={Wang, Jie},
  journal={arXiv:2512.12550},
  year={2025}
}

@article{vincent2024texttt,
  title={\texttt{skwdro}: a library for {W}asserstein distributionally robust machine learning},
  author={Vincent, Florian and Azizian, Wa{\"\i}ss and Iutzeler, Franck and Malick, J{\'e}r{\^o}me},
  journal={arXiv:2410.21231},
  year={2024}
}

@inproceedings{liu2025convergence,
  title={Convergence of Mean-Field Langevin Stochastic Descent-Ascent for Distributional Minimax Optimization},
  author={Liu, Zhangyi and Liu, Feng and Gao, Rui and Li, Shuang},
  booktitle={International Conference on Machine Learning},
  pages={38869--38893},
  year={2025}
}

@article{sheriff2025nonlinear,
  title={Nonlinear distributionally robust optimization},
  author={Sheriff, Mohammed Rayyan and Mohajerin Esfahani, Peyman},
  journal={Mathematical Programming},
  volume={213},
  number={1},
  pages={639--698},
  year={2025},
  publisher={Springer}
}

@inproceedings{kent2021modified,
  title={Modified {F}rank {W}olfe in probability space},
  author={Kent, Carson and Li, Jiajin and Blanchet, Jose and Glynn, Peter W},
  booktitle={Advances in Neural Information Processing Systems},
  pages={14448--14462},
  year={2021}
}

@article{yu2025deterministic,
  title={Deterministic and Stochastic {F}rank-{W}olfe Recursion on Probability Spaces},
  author={Yu, Di and Henderson, Shane G and Pasupathy, Raghu},
  journal={Mathematics of Operations Research (forthcoming)},
  year={2025},
  publisher={INFORMS}
}

@article{eftekhari2019sparse,
  title={Sparse inverse problems over measures: equivalence of the conditional gradient and exchange methods},
  author={Eftekhari, Armin and Thompson, Andrew},
  journal={SIAM Journal on Optimization},
  volume={29},
  number={2},
  pages={1329--1349},
  year={2019},
  publisher={SIAM}
}

@article{chizat2022sparse,
  title={Sparse optimization on measures with over-parameterized gradient descent},
  author={Chizat, Lenaic},
  journal={Mathematical Programming},
  volume={194},
  number={1},
  pages={487--532},
  year={2022},
  publisher={Springer}
}

@inproceedings{chizat2018global,
  title={On the global convergence of gradient descent for over-parameterized models using optimal transport},
  author={Chizat, Lenaic and Bach, Francis},
  booktitle={Advances in Neural Information Processing Systems},
  pages={3040--3050},
  year={2018}
}

@article{lanzetti2024variational,
  title={Variational analysis in the {W}asserstein space},
  author={Lanzetti, Nicolas and Terpin, Antonio and D{\"o}rfler, Florian},
  journal={arXiv:2406.10676},
  year={2024}
}

@article{lanzetti2022first,
  title={First-order conditions for optimization in the {W}asserstein space},
  author={Lanzetti, Nicolas and Bolognani, Saverio and D{\"o}rfler, Florian},
  journal={SIAM Journal on Mathematics of Data Science},
  volume={7},
  number={1},
  pages={274--300},
  year={2025},
  publisher={SIAM}
}

@book{bertsimas1997introduction,
  title={Introduction to Linear Optimization},
  author={Bertsimas, Dimitris and Tsitsiklis, John N},
  year={1997},
  publisher={Athena Scientific}
}

@inproceedings{taskesen2023distributionally,
  title={Distributionally robust linear quadratic control},
  author={Ta{\c{s}}kesen, Bahar and Iancu, Dan and Ko{\c{c}}yi{\u{g}}it, {\c{C}}a{\u{g}}{\i}l and Kuhn, Daniel},
  booktitle={Advances in Neural Information Processing Systems},
  pages={18613--18632},
  year={2023}
}

@inproceedings{abernethy2008optimal,
  title={Optimal strategies and minimax lower bounds for online convex games},
  author={Abernethy, Jacob and Bartlett, Peter L and Rakhlin, Alexander and Tewari, Ambuj},
  booktitle={Conference on Learning Theory},
  year={2008}
}


\appendix

\section{Omitted Proofs}

\subsection{Proof of Lemma~\ref{lem::existence_saddle}}
\label{proof::lem::existence_saddle}
We first show that the optimal values are finite. Fix $x \in \cX$. Since each function $\ell_k(x, \cdot)$ is real-valued and concave, it admits an affine upper bound. Since $K$ is finite, there exist constants $a_x, b_x < \infty$ and a point $\hat{z}_0 \in \cZ$ such that
\begin{align*}
    \ell(x, z)
    \leq
    a_x + b_x \| z - \hat{z}_0 \|
    \qquad
    \forall z \in \cZ.
\end{align*}
For any $\bQ \in \cP$, let $\gamma \in \Gamma(\bQ, \hat{\bP})$ be a coupling satisfying $\EE_{(z, \hat{z}) \sim \gamma}[c(z, \hat{z})] \leq \rho$.
By Assumption~\ref{asp::regularity}\ref{asp::regular::c}, we have $\EE_{(z, \hat{z}) \sim \gamma} [\| z - \hat{z} \|^p] \leq \rho$.
Therefore, for any $\bQ \in \cP$, we may conclude that
\begin{align*}
    \EE_{z \sim \bQ}[\ell(x, z)]
    \leq 
    a_x + b_x \, \EE_{z \sim \bQ}
    \left[
        \| z - \hat{z}_0 \|
    \right]
    &\leq
    a_x + b_x \, \EE_{(z, \hat{z}) \sim \gamma}
    \left[
        \| z - \hat{z} \|
    \right]
    +
    b_x \, \EE_{\hat{z} \sim \hat{\bP}}
    \left[
        \| \hat{z} - \hat{z}_0 \|
    \right] \\
    &\leq
    a_x + b_x \, \rho^{1/p}
    +
    \frac{b_x}{N}
    \sum_{i = 1}^{N}
    \| \hat{z}_i - \hat{z}_0 \|.
\end{align*}
Thus, $\sup_{\bQ \in \cP} \EE_{z \sim \bQ}[\ell(x, z)] < \infty$ for every $x \in \cX$.

To bound the problem from below, we leverage the inf-compactness assumption. By Assumption~\ref{asp::regularity}\ref{asp::regular::sets}, the function $\EE_{z \sim \bQ_0}[\ell(\cdot, z)]$ is inf-compact on $\cX$, and the scalar $c_0 := \EE_{z \sim \bQ_0}[\ell(x_0, z)]$ is finite. By definition of inf-compactness, the corresponding sublevel set $\cS_0 = \{x \in \cX : \EE_{z \sim \bQ_0}[\ell(x, z)] \leq c_0\}$ is compact in $\RR^n$ and nonempty since $x_0 \in \cS_0$. Because all sublevel sets of $\EE_{z \sim \bQ_0}[\ell(\cdot, z)]$ are closed, this function is lower semicontinuous on $\cX$. By the generalized Weierstrass theorem, a lower semicontinuous function on a nonempty compact set attains its minimum; thus, it attains a finite minimum $m_0 > -\infty$ on $\cS_0$. Since any $x \in \cX \setminus \cS_0$ strictly yields an objective value greater than $c_0 \geq m_0$, $m_0$ is the global minimum over the entire feasible region $\cX$. Since $\bQ_0 \in \cP$, we have
\begin{align*}
    \sup_{\bQ \in \cP} \EE_{z \sim \bQ}[\ell(x, z)] 
    \geq 
    \EE_{z \sim \bQ_0}[\ell(x, z)] 
    \geq 
    m_0 
    > 
    -\infty, \quad \text{for every }x \in \cX.
\end{align*}
This proves that the primal DRO value is finite. The dual DRO value is finite as well. To see this, first note that weak duality gives
\begin{align*}
    \sup_{\bQ \in \cP}
    \inf_{x \in \cX}
    \EE_{z \sim \bQ}[\ell(x, z)]
    \leq
    \inf_{x \in \cX}
    \sup_{\bQ \in \cP}
    \EE_{z \sim \bQ}[\ell(x, z)]
    <
    \infty,
\end{align*}
while the feasible distribution $\bQ_0 \in \cP$ gives
\begin{align*}
    \sup_{\bQ \in \cP}
    \inf_{x \in \cX}
    \EE_{z \sim \bQ}[\ell(x, z)]
    \geq
    \inf_{x \in \cX}
    \EE_{z \sim \bQ_0}[\ell(x, z)]
    \geq 
    m_0
    >
    -\infty.
\end{align*}

It remains to establish attainment and the minimax identity. Assumption~\ref{asp::regularity}\ref{asp::regular::sets} gives a nonempty inf-compact convex decision set. Assumption~\ref{asp::regularity}\ref{asp::regular::ell} gives convexity in $x$ and upper semicontinuity in~$z$. Assumption~\ref{asp::regularity}\ref{asp::regular::c} guarantees tightness of the OT ambiguity set and the required upper-growth condition. Specifically, for $p > 1$, the superlinear transportation cost dominates the linear growth of the loss; for $p = 1$, the conclusion follows from the stated sublinear growth condition. Moreover, since $c$ is real-valued on $\cZ \times \cZ$ and $\rho > 0$, the Slater conditions required in~\cite[Assumption~7]{shafiee2025nash} are satisfied.
Therefore, the requirements of the minimax result in~\cite[Lemmas~3 \& 4]{shafiee2025nash} hold. Consequently, the primal and dual DRO problems have the same finite optimal value, both extrema are attained, and the corresponding optimizers form a saddle point. $\hfill\qed$


\subsection{Proofs of Lemma~\ref{lem::eval_approx} and Lemma~\ref{lem::evaluate_vi}}
\label{proof::lemma2and3}

To analyze the evaluation of \(V_i^{(k_1,k_2)}(b)\) in~\eqref{eq::U_i^k}, we first fix \(b \in \RR_+\). We then introduce the function \(\Psi_i^{(k_1,k_2)}(\alpha,\beta)\), which denotes the optimal value of~\eqref{eq::U_i^k} over \(v_1\) and \(v_2\), for fixed weights \(\alpha=(\alpha_1,\alpha_2)\) and budget allocations \(\beta=(\beta_1,\beta_2)\). Specifically, we define
\begin{equation}
\label{eq::Psi}
\Psi_i^{(k_1,k_2)}(\alpha, \beta) := 
\max_{v_{1} , v_{2}} \left\{ 
\alpha_{1} \ell_{k_1}\left(\hat{z}_i + v_1 \right) + \alpha_{2} \ell_{k_2} \left(\hat{z}_i + v_2 \right):\ 
\begin{aligned}
&v_j \in \RR^m,\ \hat{z}_i + v_j\in \cZ, ~ \forall j \in [2] \\
&\alpha_j c\left(\hat{z}_i+v_j, \hat{z}_i\right) \leq \beta_j, ~ \forall j \in [2]
\end{aligned} \right\}.
\end{equation}

For any index $k \in [K]$, nominal point $\hat{z} \in \mathcal{Z}$, and radius $u \geq 0$, we define the function $S_k(\cdot\,; \hat{z}): \mathbb{R}_+ \rightarrow \mathbb{R}$ as
\begin{equation}
\label{eq::inner_oracle}
S_k(u; \hat{z}) := \max
\left\{ \ell_k(\hat z+v) : v \in \bR^m, \ \hat z+v\in\cZ, \
c(\hat z+v,\hat z) \leq u \right\}.    
\end{equation}
We note that Assumption~\ref{asp::inner_oracle} provides an approximate oracle for computing $S_k(u; \hat{z})$.
A direct consequence of this formulation is that $S_k(0 ; \hat{z}) = \ell_k(\hat{z})$, and $S_k(u; \hat{z})$ is nondecreasing in $u$. Furthermore, $S_k(\cdot\,; \hat{z})$ is a concave function. To see this, note that $\left\{ (v,u)\in \bR^m\times \bR: \ \hat z+v\in\cZ, \
c(\hat z+v,\hat z) \leq u \right\}$ is a convex set and $\ell_k(\hat z+v)$ is jointly concave with respect to $(v,u)$. Therefore, the concavity of $S_k(\cdot\,; \hat{z})$ directly follows from \cite[Section 3.2.5]{boyd2004convex}.  
Armed with this function, for weights $\alpha \in \{(\alpha_1, \alpha_2) \in \mathbb{R}_{++}^2 : \alpha_1 + \alpha_2 = 1\}$ and budget allocations $\beta \in \{(\beta_1, \beta_2) \in \mathbb{R}_{+}^2 : \beta_1 + \beta_2 = b\}$, we can rewrite $\Psi_i^{(k_1,k_2)}(\alpha, \beta)$ as
\begin{equation}
\label{eq::perspective_decomposition}
\Psi_i^{(k_1,k_2)}(\alpha, \beta) = \alpha_1 S_{k_1}\left(\frac{\beta_1}{\alpha_1}; \hat{z}_i \right) + \alpha_2 S_{k_2}\left(\frac{\beta_2}{\alpha_2}; \hat{z}_i \right).    
\end{equation}
The following lemma establishes the nested concavity of $\Psi_i^{(k_1,k_2)}(\alpha, \beta)$.

\begin{lemma}
\label{lem::dgss}
Fix any $i\in [N]$, $1\leq k_1 < k_2 \leq K$, and $b \in \RR_+$. The function $\Psi_i^{(k_1,k_2)}(\alpha, \beta)$ exhibits nested concavity over $(\alpha, \beta)$:
\begin{enumerate}
    \item For fixed $\alpha$, it is concave in the budget allocation $\beta$.
    \item The partial maximum $\Phi_i^{(k_1,k_2)}(\alpha) := \max_{\beta} \Psi_i^{(k_1,k_2)}(\alpha, \beta)$ is concave in the weight allocation $\alpha$.
\end{enumerate}
Here we restrict $\alpha \in \{(\alpha_1 , \alpha_2): \alpha_1 + \alpha_2 = 1, \alpha_1,\alpha_2 \in \RR_{++} \}$, $\beta \in \{(\beta_1 , \beta_2): \beta_1 + \beta_2 = b, \beta_1,\beta_2 \in \RR_+ \}$.
\end{lemma}
\begin{proof}
Recall the decomposition established in~\eqref{eq::perspective_decomposition}. By construction, \(\Psi_i^{(k_1,k_2)}(\alpha,\beta)\) can be expressed as the sum of two perspective functions. Since the perspective operation preserves concavity, it follows that \(\Psi_i^{(k_1,k_2)}(\alpha,\beta)\) is jointly concave in \((\alpha,\beta)\) over its domain, which establishes the first statement. The second statement follows directly from~\cite[Section 3.2.5]{boyd2004convex}.
\end{proof}

The following helper lemma shows that any one-dimensional concave function satisfies the fundamental subadditivity property.

\begin{lemma}
\label{lem::subadditivity}
Let $h: \RR_+ \to \RR$ be a concave function, then for any $0\leq y \leq x$, we have
\[
 h(x) - h(y) \leq h\left( x-y \right) - h(0).
\]
\end{lemma}
\begin{proof}
Since \(h(\cdot)\) is concave, Jensen's inequality implies that, for any \(a,b \geq 0\),
\begin{align*}
h(a) &\geq \frac{a}{a+b} h(a+b) + \frac{b}{a+b} h(0),\\
h(b) &\geq \frac{b}{a+b} h(a+b) +\frac{a}{a+b} h(0).
\end{align*}
Summing the two inequalities yields
\[
h(a)+h(b)\geq h(a+b)+h(0).
\]
Now, for any \(0 \leq y \leq x\), let \(a=x-y\) and \(b=y\). Then,
\[
h(x-y)+h(y)\geq h(x)+h(0) \implies h(x)-h(y)\leq h(x-y)-h(0).
\]
This completes the proof.
\end{proof}

In the subsequent analysis, let
$G_i$
denote the largest dual norm among all supergradients of \(\ell_k\) at \(\hat z_i\). Since each function \(\ell_k\) is locally Lipschitz continuous at \(\hat z_i\), it follows that \(G_i\) is finite.
The following lemma characterizes the continuity of $S_k(\cdot\,; \hat{z}_i)$ at the origin.

\begin{lemma}
\label{lem::continuity_subroutine}
Suppose Assumptions~\ref{asp::regularity} and~\ref{asp::inner_oracle} hold. Fix any $k\in [K]$ and $i\in [N]$. For any $u \geq 0$, we have
$$
S_k(u; \hat{z}_i) - S_k(0; \hat{z}_i) \leq 
\left\{
\begin{aligned}
& C_i + g u^{r} && p = 1 \\
& G_i u^{1/p} && p \geq 1
\end{aligned}
\right.,
$$
where $C_i := \max_{k \in [K]} \big\{ g + g\norm{\hat{z}_i - \hat{z}_0}^r - \ell_k(\hat{z}_i) \big\}$ and $G_i := \max_{k \in [K]} \big\{ \big\| g_i^{(k)} \big\|_* : g_i^{(k)} \in \partial \ell_k(\hat z_i) \big\}$ are constants; the constant $g>0$, the reference point $\hat{z}_0$ and the exponent $r \in (0,1)$ are specified in Assumption~\ref{asp::regularity}\ref{asp::regular::c}.
\end{lemma}
\begin{proof}
We analyze the cases $p=1$ and $p \ge 1$ separately. Although the $p \ge 1$ regime inherently includes $p=1$, we isolate the latter to provide a more refined analysis. When $p=1$, Assumption~\ref{asp::regularity}\ref{asp::regular::c} indicates that there exists a constant $g>0$, a reference point $\hat{z}_0$ and an exponent $r \in (0,1)$ such that
\[
\ell_k(\hat{z}_i + v) \leq g\left( 1 + \norm{\hat{z}_i + v - \hat{z}_0}^r \right),
\]
for any $\hat{z}_i+v\in \cZ$.
Because $(x+y)^r \leq x^r + y^r$ for $r \in (0,1)$ and any $x, y \in \RR_+$, we have
\[
\ell_k(\hat{z}_i + v) \leq g\left( 1 + \norm{v}^r + \norm{\hat{z}_i - \hat{z}_0}^r \right),
\]
which implies
\[
\ell_k(\hat{z}_i + v) - \ell_k(\hat{z}_i) \leq \left(g + g\norm{\hat{z}_i - \hat{z}_0}^r - \ell_k(\hat{z}_i)\right) + g \norm{v}^r.
\]
Let $C_i$ be a constant as defined in the lemma statement. For any $v \in \RR^m$ such that $\hat{z}_i + v\in \cZ$ and $c(\hat{z}_i + v, \hat{z}_i) \leq u$, by Assumption~\ref{asp::regularity}\ref{asp::regular::c}, we have $\norm{v} \leq c(\hat{z}_i + v, \hat{z}_i) \leq u$. Taking maximization over $v$ on both sides yields:
\[
S_k(u; \hat{z}_i) - S_k(0; \hat{z}_i) \leq C_i + g u^{r}.
\]
Next, we consider $p \geq 1$. Due to the concavity of $\ell_k(\cdot)$, for any perturbation $v$ such that $\hat{z}_i + v\in \cZ$ and $c(\hat{z}_i + v, \hat{z}_i) \leq u$, we have
$$
\ell_k(\hat{z}_i+v) - \ell_k(\hat{z}_i) \leq \langle g_i^{(k)}, v \rangle \leq G_i \|v\|,
$$
where $g_i^{(k)} \in \partial \ell_k(\hat z_i)$.
From Assumption~\ref{asp::regularity}\ref{asp::regular::c}, $\|v\| \leq c(\hat{z}_i+v, \hat{z}_i)^{1/p} \leq u^{1/p}$. Therefore, taking maximization over $v$ on both sides yields:
$$
S_k(u; \hat{z}_i) - S_k(0; \hat{z}_i) \leq G_i u^{1/p},
$$
completing the proof.
\end{proof}

Building upon Lemma~\ref{lem::continuity_subroutine}, we are now equipped to establish the H\"{o}lder continuity of $\Psi_i^{(k_1,k_2)}(\alpha, \beta)$ with respect to the budget allocation $\beta$ for a fixed weight $\alpha$. This result is formally stated in the following lemma.

\begin{lemma}
\label{lem::continuity_beta}
Suppose Assumptions~\ref{asp::regularity} and~\ref{asp::inner_oracle} hold. Fix any $i\in [N]$, $1\leq k_1 < k_2 \leq K$, $b \in \RR_+$, and a weight $\alpha \in \RR_{+}^2$ such that $\alpha_1 + \alpha_2 = 1$. For any two budget allocations $\beta, \beta'\in \RR^2_{+}$ such that $\beta_1+\beta_2 = \beta_1'+\beta_2' = b$, we have:
$$
\left| \Psi_i^{(k_1,k_2)}(\alpha, \beta) - \Psi_i^{(k_1,k_2)}(\alpha, \beta') \right| \leq 2 G_i \left| \beta_1 - \beta_1' \right|^{1/p},
$$
where the constants $G_i$ is specified in Lemma~\ref{lem::continuity_subroutine}.
\end{lemma}
\begin{proof}
For any fixed $k \in [K]$, the function $S_k(u; \hat{z}_i)$ is concave for $u \geq 0$. Let $u_1, u_2 \geq 0$ be two arbitrary radii. Assuming without loss of generality that $u_2 \geq u_1$, we obtain:
$$
S_k(u_2; \hat{z}_i) - S_k(u_1; \hat{z}_i) \leq S_k(u_2 - u_1; \hat{z}_i) - S_k(0; \hat{z}_i)\leq G_i |u_2 - u_1|^{1/p},
$$
where the first and second inequalities follow from Lemma~\ref{lem::subadditivity} and Lemma~\ref{lem::continuity_subroutine}, respectively.
Due to the symmetric nature of the above inequality, it holds for any $u_1, u_2 \geq 0$ that $|S_k(u_2; \hat{z}_i) - S_k(u_1; \hat{z}_i)| \leq G_i |u_2 - u_1|^{1/p}$. This implies:
$$
\begin{aligned}
&\left| \Psi_i^{(k_1,k_2)}(\alpha, \beta) - \Psi_i^{(k_1,k_2)}(\alpha, \beta') \right|\\ 
&\leq \alpha_1 \left| S_{k_1}\left(\frac{\beta_1}{\alpha_1}; \hat{z}_i\right) - S_{k_1}\left(\frac{\beta'_1}{\alpha_1}; \hat{z}_i\right) \right|+\alpha_2 \left| S_{k_2}\left(\frac{\beta_2}{\alpha_2}; \hat{z}_i\right) - S_{k_2}\left(\frac{\beta'_2}{\alpha_2}; \hat{z}_i\right) \right| \\
&\leq \alpha_1 G_i \left| \frac{\beta_1}{\alpha_1} - \frac{\beta'_1}{\alpha_1} \right|^{1/p}+\alpha_2 G_i \left| \frac{\beta_2}{\alpha_2} - \frac{\beta'_2}{\alpha_2} \right|^{1/p} \\
&= G_i \left(\alpha_1^{1 - 1/p} \left| \beta_1 - \beta'_1 \right|^{1/p}+\alpha_2^{1 - 1/p} \left| \beta_2 - \beta'_2 \right|^{1/p}\right)\\
&\leq 2 G_i \left| \beta_1 - \beta_1' \right|^{1/p},
\end{aligned}
$$
thereby completing the proof.
\end{proof}

Similarly, we establish the H\"{o}lder continuity of the partial maximum $\Phi_i^{(k_1,k_2)}(\alpha) := \max_{\beta} \Psi_i^{(k_1,k_2)}(\alpha, \beta)$ with respect to $\alpha$ in the following lemma.

\begin{lemma}
\label{lem::continuity_alpha}
Suppose Assumptions~\ref{asp::regularity} and~\ref{asp::inner_oracle} hold. Fix any $i\in [N]$, $1\leq k_1 < k_2 \leq K$, and $b \in \RR_+$. For any two weight allocations $\alpha, \alpha' \in \RR_{++}$ such that $\alpha_1+\alpha_2 = \alpha_1'+\alpha_2' = 1$, we have:
$$
\left| \Phi_i^{(k_1,k_2)}(\alpha) - \Phi_i^{(k_1,k_2)}(\alpha') \right| \leq 
\left\{
\begin{aligned}
& (R_i + 2 C_i) |\alpha_1 - \alpha_1'| + 2 g b^r |\alpha_1 - \alpha'_1|^{1 - r} && p = 1\\
& R_i |\alpha_1 - \alpha'_1| + 2G_i b^{1/p}|\alpha_1 - \alpha'_1|^{1 - 1/p} &&  p> 1
\end{aligned}
\right. ,
$$
where $R_{i} := \max_{1\leq k_1 < k_2 \leq K} \left| \ell_{k_1}(\hat{z}_i) - \ell_{k_2}(\hat{z}_i) \right|$ is a constant; the constants $C_i$, $G_i$, $g>0$, and $r \in (0,1)$ are specified in Lemma~\ref{lem::continuity_subroutine}.
\end{lemma}
\begin{proof}
Let $\beta^\star(\alpha) = \argmax_\beta \Psi_i^{(k_1,k_2)}(\alpha, \beta)$ and $\beta^\star(\alpha') = \argmax_\beta \Psi_i^{(k_1,k_2)}(\alpha', \beta)$ be the optimal budgets for fixed $\alpha$ and $\alpha'$, respectively. By definition, $\Phi_i^{(k_1,k_2)}(\alpha) = \Psi_i^{(k_1,k_2)}(\alpha, \beta^\star(\alpha))$ and $\Phi_i^{(k_1,k_2)}(\alpha') \geq \Psi_i^{(k_1,k_2)}(\alpha', \beta^\star(\alpha))$, which implies $\Phi_i^{(k_1,k_2)}(\alpha)-\Phi_i^{(k_1,k_2)}(\alpha')\leq \Psi_i^{(k_1,k_2)}(\alpha, \beta^\star(\alpha))-\Psi_i^{(k_1,k_2)}(\alpha', \beta^\star(\alpha))$. Similarly, one can obtain $\Phi_i^{(k_1,k_2)}(\alpha')-\Phi_i^{(k_1,k_2)}(\alpha)\leq \Psi_i^{(k_1,k_2)}(\alpha', \beta^\star(\alpha'))-\Psi_i^{(k_1,k_2)}(\alpha, \beta^\star(\alpha'))$. Combining these inequalities yields:
\begin{align}\label{eq::Phi-equiv}
    \left| \Phi_i^{(k_1,k_2)}(\alpha) - \Phi_i^{(k_1,k_2)}(\alpha') \right| \leq \max_{\beta \in \{\beta^\star(\alpha), \beta^\star(\alpha')\}} \left| \Psi_i^{(k_1,k_2)}(\alpha, \beta) - \Psi_i^{(k_1,k_2)}(\alpha', \beta) \right|.
\end{align}
We now bound the right-hand side for any fixed feasible budget $\beta$. For convenience, define $H_j(u) := S_{k_j}(u; \hat{z}_i) - S_{k_j}(0; \hat{z}_i)$. Because $S_{k_j}(\cdot\,; \hat{z}_i)$ is concave, $H_j(\cdot)$ is concave, nondecreasing and satisfies $H_j(0) = 0$. Noting that $S_{k_j}(0; \hat{z}_i) = \ell_{k_j}(\hat{z}_i)$, we rewrite \eqref{eq::perspective_decomposition} using $H_j$:
\begin{align}\label{eq::Psi-equiv}
    \Psi_i^{(k_1,k_2)}(\alpha, \beta) = \sum_{j=1}^2 \alpha_j \ell_{k_j}(\hat{z}_i) + \sum_{j=1}^2 \underbrace{\alpha_j H_j\left(\frac{\beta_j}{\alpha_j}\right)}_{:= h_j(\alpha_j)}.
\end{align}
Let $\Delta := |\alpha_1 - \alpha'_1| = |\alpha_2 - \alpha'_2|$ and fix $j\in [2]$. As the perspective function of a concave function, $h_j(\alpha_j)$ is concave and nonnegative for $\alpha_j \geq 0$, with $h_j(0) = 0$. Therefore, by Lemma~\ref{lem::subadditivity}, for any $\alpha_j, \alpha_j' \geq 0$:
$$
\left| h_j(\alpha_j) - h_j(\alpha'_j) \right| \leq \left| h_j(|\alpha_j - \alpha'_j|) - h_j(0) \right| = h_j(\Delta)=\Delta H_j\left(\frac{\beta_j}{\Delta}\right).
$$
We now divide the analysis based on the value of $p$ to bound $H_j(\beta_j / \Delta)$, utilizing the upper bounds established in Lemma~\ref{lem::continuity_subroutine}.

$\diamond$ When $p > 1$, from Lemma~\ref{lem::continuity_subroutine}, we have $H_j(u) \leq G_i u^{1/p}$, which leads to
$$
\left| h_j(\alpha_j) - h_j(\alpha'_j) \right|\leq h_j(\Delta)=\Delta H_j\left(\frac{\beta_j}{\Delta}\right) \leq \Delta \left[ G_i \left( \frac{\beta_j}{\Delta} \right)^{1/p} \right] \leq G_i b^{1/p} \Delta^{1 - 1/p}.
$$
Combining this bound with \eqref{eq::Psi-equiv} yields
$$
\begin{aligned}
\left| \Psi_i^{(k_1,k_2)}(\alpha, \beta) - \Psi_i^{(k_1,k_2)}(\alpha', \beta) \right| 
&\leq \left| \ell_{k_1}(\hat{z}_i) - \ell_{k_2}(\hat{z}_i) \right| \Delta + \sum_{j=1}^2 h_j(\Delta) 
\leq R_i \Delta + 2 G_i b^{1/p} \Delta^{1 - 1/p}.
\end{aligned}
$$
where $R_i$ is a constant defined in the lemma statement.

$\diamond$ When $p = 1$, from Lemma~\ref{lem::continuity_subroutine}, we have $H_j(u) \leq C_i + g u^r$, which leads to
$$
\left| h_j(\alpha_j) - h_j(\alpha'_j) \right|\leq h_j(\Delta)=\Delta H_j\left(\frac{\beta_j}{\Delta}\right) \leq \Delta \left[ C_i + g \left( \frac{\beta_j}{\Delta} \right)^{r} \right] \leq C_i \Delta + g b^r \Delta^{1 - r}.
$$
Combining this bound with \eqref{eq::Psi-equiv} yields
$$
\begin{aligned}
\left| \Psi_i^{(k_1,k_2)}(\alpha, \beta) - \Psi_i^{(k_1,k_2)}(\alpha', \beta) \right| 
&\leq \left| \ell_{k_1}(\hat{z}_i) - \ell_{k_2}(\hat{z}_i) \right| \Delta + \sum_{j=1}^2 \left( C_i \Delta + g b^r \Delta^{1 - r} \right) \\
&\leq  (R_i + 2 C_i)\Delta + 2 g b^r \Delta^{1 - r}.
\end{aligned}
$$
Since both cases hold for any feasible \(\beta\), substituting them into~\eqref{eq::Phi-equiv} completes the proof.   
\end{proof}

We are now ready to prove Lemma~\ref{lem::eval_approx}.

\begin{proof}[Proof of Lemma~\ref{lem::eval_approx}] 
For convenience, for fixed $i \in [N]$, $1 \le k_1 < k_2 \le K$, $b \ge 0$, we abbreviate $\Psi_i^{(k_1,k_2)}(\alpha,\beta)$ by $\Psi(\alpha,\beta)$ and abbreviate $\Phi_i^{(k_1,k_2)}(\alpha)$ by $\Phi(\alpha)$ throughout this proof. 
We also define
\begin{align}
\label{eq:constants}
    C := \max_{i \in [N]} ~ C_i, 
    \qquad
    G := \max_{i \in [N]} ~ G_i, 
    \qquad
    R := \max_{i \in [N]} ~ R_i,
\end{align}
where the local constants $C_i$, $G_i$ and $R_i$ are defined in Lemmas~\ref{lem::continuity_subroutine} and \ref{lem::continuity_alpha}.
The nested concavity of $\Psi(\alpha,\beta)$ and $\Phi(\alpha)$ established in Lemma~\ref{lem::dgss} guarantees that the objective is unimodal along any line segment, which is the key property underlying the nested golden-section search algorithm (Algorithms~\ref{alg::evaluate_pairwise} and~\ref{alg::solve_inner}).

Let $(\alpha^\star,\beta^\star)$ be an optimal solution to \eqref{eq::U_i^k} with objective value $V_i^{(k_1, k_2)}(b)$. Under Assumption~\ref{asp::inner_oracle}, for any fixed $\alpha_j \in \RR_{++}$ and $\beta_j \in \RR_{+}$, we set the radius to $u_j = \beta_j/\alpha_j$, for $j\in [2]$. By querying the oracle with accuracy $\epsilon > 0$, we obtain a perturbation $\hat{v}_j \in \RR^m$ in time $\mathsf{Cost}_{k_j, \epsilon}$ such that $\hat{z}_j + v_j \in \cZ$, $c(\hat{z}_i + \hat{v}_j, \hat{z}_i) \le u_j$ and
$$
\ell_{k_j}(\hat{z}_i + \hat{v}_j) \ge S_{k_j}(u_j; \hat{z}_i) - \epsilon, \quad j \in [2].
$$
We define our approximate evaluation of $\Psi(\alpha,\beta)$ as:
$$
\hat{\Psi}(\alpha, \beta) := \alpha_1 \ell_{k_1}(\hat{z}_i + \hat{v}_1) + \alpha_2 \ell_{k_2}(\hat{z}_i + \hat{v}_2).
$$
Since its exact value is $\Psi(\alpha, \beta) = \alpha_1 S_{k_1}(u_1; \hat{z}_i) + \alpha_2 S_{k_2}(u_2; \hat{z}_i)$, the approximation error is bounded by
$$
\Psi(\alpha, \beta) - \hat{\Psi}(\alpha, \beta) \le \alpha_1 \epsilon + \alpha_2 \epsilon = \epsilon.
$$
Thus, we can compute $\Psi(\alpha, \beta)$ to tolerance $\epsilon$ in time $\mathsf{Cost}_{k_1,\epsilon} + \mathsf{Cost}_{k_2,\epsilon}$.

Fix $\alpha$ in the standard simplex. The inner golden-section search in Algorithm~\ref{alg::solve_inner} narrows the search interval for $\beta$ by comparing evaluations $\hat{\Psi}(\alpha,\beta)$ and $\hat{\Psi}(\alpha,\beta')$. We claim that the algorithm discards the correct subinterval whenever the true difference in function values exceeds $2 \epsilon$. Assume without loss of generality that $\Psi(\alpha,\beta) > \Psi(\alpha,\beta') + 2\epsilon$. Then we have
$$
\hat{\Psi}(\alpha,\beta) \ge \Psi(\alpha,\beta) - \epsilon > \Psi(\alpha,\beta') + \epsilon \ge \hat{\Psi}(\alpha,\beta'),
$$
ensuring the algorithm correctly maintains the optimal budget $\beta^\star(\alpha)$ within the interval $[L_{\beta}, U_{\beta}]$. Conversely, the algorithm may only discard a subinterval containing the true maximizer $\beta^\star(\alpha)$ if the function values are indistinguishable within the $2\epsilon$ tolerance. As a result, the suboptimality introduced by potentially incorrect discards is bounded by $2\epsilon$. The search terminates when the interval length reduces to $\eta$ in Algorithm~\ref{alg::solve_inner}. By Lemma~\ref{lem::continuity_beta}, $\Psi(\alpha, \cdot)$ is $(1/p)$-H\"{o}lder continuous, meaning this final interval length introduces a continuous resolution error of at most $2 G_i \eta^{1/p}$. Thus, the overall gap to the true inner maximum is bounded by:
$$
\Psi(\alpha,\beta^\star(\alpha)) - \Psi(\alpha,\hat{\beta}(\alpha)) \le \max\left\{2\epsilon, 2 G_i \eta^{1/p}\right\}\leq 2\epsilon+ 2 G \eta^{1/p}.
$$ 
Algorithm~\ref{alg::solve_inner} achieves this precision in time:
$$
O\left( \left( \mathsf{Cost}_{k_1,\epsilon} + \mathsf{Cost}_{k_2,\epsilon} \right) 
\log\!\left(\frac{b}{\eta}\right) \right).
$$
The outer golden-section search in Algorithm \ref{alg::evaluate_pairwise} operates on the partial maximum function $\Phi(\alpha) = \Psi(\alpha, \beta^\star(\alpha))$. For any queried $\alpha$, the algorithm uses the inner  search to compute the approximate maximum $\hat{\Phi}(\alpha) := \hat{\Psi}(\alpha, \hat{\beta}(\alpha))$. The total error of this evaluation incorporates both the inner optimization error and the oracle evaluation error:
$$
|\Phi(\alpha) - \hat{\Phi}(\alpha)| \le |\Psi(\alpha, \beta^\star(\alpha)) - \Psi(\alpha, \hat{\beta}(\alpha))| + |\Psi(\alpha, \hat{\beta}(\alpha)) - \hat{\Psi}(\alpha, \hat{\beta}(\alpha))| \le  3\epsilon + 2 G \eta^{1/p}.
$$
Similar to the inner search, Algorithm \ref{alg::evaluate_pairwise} correctly discards subintervals whenever the true difference in function values between two weights $\alpha, \alpha'$ strictly exceeds twice the evaluation error $6 \epsilon + 4 G \eta^{1/p}$. Conversely, the algorithm may mistakenly discard a subinterval containing the true optimal weight $\alpha^\star$ only if the function values are indistinguishable within this tolerance. Consequently, the suboptimality introduced by potentially incorrect discards is bounded by $6 \epsilon + 4 G \eta^{1/p}$.

The search safely terminates when the interval length reduces to $\eta$ in Algorithm~\ref{alg::evaluate_pairwise}. This final interval length introduces an additional continuous resolution error, which can be controlled by the H\"{o}lder continuity of $\Phi(\alpha)$ established in Lemma~\ref{lem::continuity_alpha}. Therefore, the total suboptimality gap must account for both the accumulated error from the discard threshold and this final resolution error. Adding these components together and using the bounds $G_i \leq G$, $R_i \leq R$, and $C_i \leq C$, the gap is bounded~by:
$$
\Phi(\alpha^\star) - \Phi(\hat{\alpha}) \le \left( 6 \epsilon + 4 G \eta^{1/p} \right) + 
\begin{cases} 
R \eta + 2 G b^{1/p} \eta^{1 - 1/p} & \text{if } p > 1, \\ 
(R + 2C) \eta + 2 g b^r \eta^{1 - r} & \text{if } p = 1. 
\end{cases}
$$
The outer search requires $O(\log(1/\eta))$ iterations to reach the final interval $\eta$. Combining the nested loops, the total runtime of the algorithm is:
$$
O \left( \left( \mathsf{Cost}_{k_1,\epsilon} + \mathsf{Cost}_{k_2,\epsilon} \right) 
\log\!\left(\frac{b}{\eta}\right) \log\!\left(\frac{1}{\eta}\right) \right).
$$
The final evaluation returned by the algorithm is $\hat{\Phi}(\hat{\alpha}) = \hat{\Psi}(\hat{\alpha}, \hat{\beta}(\hat{\alpha}))$. Its deviation from the true global optimum $V^{(k_1,k_2)}_i(b) = \Phi(\alpha^\star)$ satisfies:
\begin{align*}
|\hat{\Phi}(\hat{\alpha}) - \Phi(\alpha^\star)|
&\le |\hat{\Phi}(\hat{\alpha}) - \Phi(\hat{\alpha})| + |\Phi(\hat{\alpha}) - \Phi(\alpha^\star)| \\
&\le \left( 9 \epsilon + 6 G \eta^{1/p} \right) + 
\begin{cases} 
R \eta + 2 G b^{1/p} \eta^{1 - 1/p} & \text{if } p > 1, \\ 
(R + 2 C) \eta + 2 g b^r \eta^{1 - r} & \text{if } p = 1. 
\end{cases}
\end{align*}
By defining the required target tolerance appropriately, the nested golden-section search correctly and efficiently computes the approximate evaluation. This completes the proof.
\end{proof}

\begin{proof}[Proof of Lemma~\ref{lem::evaluate_vi}] 
The proof readily follows from \eqref{eq::V_i} and the result of Lemma~\ref{lem::eval_approx}.
\end{proof}


\subsection{Proof of Lemma~\ref{lem::Bi_eval}}
\label{proof::lem::Bi_eval}

Before presenting the proof of Lemma~\ref{lem::Bi_eval}, we first need three intermediate results.

\begin{lemma}
\label{lem::concave_V}
Under Assumptions~\ref{asp::regularity} and~\ref{asp::inner_oracle}, the function $V_i(b)$ is concave on $b \in \RR_+$ for each $i \in [N]$.
\end{lemma}

\begin{proof}
Recall from Section~\ref{sec::worst_case} that the full local utility $V_i(b)$ can be expressed as the partial maximization of a joint objective over the probability weights $\alpha$ and the spatial perturbations $v$. By applying the same perspective function transformation used in \eqref{eq::perspective_decomposition}, the joint objective can be rewritten as a sum of perspective functions, which is jointly concave in the weights $\alpha$ and the component budgets $\beta$. Because the constraints $\sum_{k=1}^K \alpha_k = 1$ and $\sum_{k=1}^K \beta_k \le b$ define a convex feasible region in $(\alpha, \beta, b)$, $V_i(b)$ is the partial maximization of a jointly concave function over a convex set. The concavity of $V_i(b)$ thus follows directly from~\cite[Section 3.2.5]{boyd2004convex}.
\end{proof}

Our next lemma establishes the H\"{o}lder continuity of $V_i$.

\begin{lemma}
\label{lem::continuity_V}
Suppose Assumptions~\ref{asp::regularity} and~\ref{asp::inner_oracle} hold. Fix $i\in [N]$. For any two budget allocations $b, b' \in \RR_+$, we have:
\begin{align*}
    \left| V_i(b) - V_i(b') \right| \le G_i \left| b - b'\right|^{1/p},
\end{align*}
where $G_i$ is a constant specified in Lemma~\ref{lem::continuity_subroutine}.
\end{lemma}
\begin{proof}
For any budget $\delta > 0$ and arbitrarily small $\epsilon > 0$, there exists a feasible solution $\{\alpha_{k}^{\delta, \epsilon}\}_{k\in [K]}$, $\{v_{k}^{\delta, \epsilon}\}_{k\in [K]}$ to \eqref{eq::full_V_i} for $V_i(\delta)$ such that the budget constraint $\sum_{k=1}^K \alpha_k^{\delta, \epsilon} c(\hat{z}_i + v_k^{\delta, \epsilon}, \hat{z}_i) \leq \delta$ holds, and
$$
0\leq V_i(\delta) - \sum_{k=1}^{K} \alpha_k^{\delta, \epsilon} \ell_k(\hat{z}_i + v_k^{\delta, \epsilon}) < \epsilon.
$$
From the assumption $c(z, \hat{z}_i) \ge \|z - \hat{z}_i\|^p$, we know that $\|v_k\| \le c(\hat{z}_i + v_k, \hat{z}_i)^{1/p}$. Because $p \ge 1$, the function $x \mapsto x^{1/p}$ is concave for $x \ge 0$. We can apply Jensen's inequality to bound the weighted sum of norms:$$\sum_{k=1}^K \alpha_k^{\delta, \epsilon} \|v_k^{\delta, \epsilon}\| \leq \sum_{k=1}^K \alpha_k^{\delta, \epsilon} c(\hat{z}_i + v_k^{\delta, \epsilon}, \hat{z}_i)^{1/p} \leq \left( \sum_{k=1}^K \alpha_k^{\delta, \epsilon} c(\hat{z}_i + v_k^{\delta, \epsilon}, \hat{z}_i) \right)^{1/p} \leq \delta^{1/p}.$$Now, we bound the objective difference from $0$:
$$
\begin{aligned}
V_i(\delta) - V_i(0) &\leq \sum_{k=1}^{K} \alpha_k^{\delta, \epsilon} \ell_k(\hat{z}_i + v_k^{\delta, \epsilon}) + \epsilon  - \max_{k\in [K]}\ell_k(\hat{z}_i) \\
&\leq \sum_{k=1}^{K} \alpha_k^{\delta, \epsilon} \left[ \ell_k(\hat{z}_i + v_k^{\delta, \epsilon}) - \ell_k(\hat{z}_i)\right] + \epsilon \\
&\leq \sum_{k=1}^{K} \alpha_k^{\delta, \epsilon} \norm{g_i^{(k)}}_* \norm{v_k^{\delta, \epsilon}} + \epsilon \\
&\leq G_i \sum_{k=1}^K \alpha_k^{\delta, \epsilon} \norm{v_k^{\delta, \epsilon}} + \epsilon \\
&\leq G_i \delta^{1/p} + \epsilon.
\end{aligned}
$$
Taking the limit as $\epsilon \rightarrow 0^+$, we arrive at $V_i(\delta) - V_i(0) \leq G_i \delta^{1/p}.$
Last, by invoking Lemma~\ref{lem::subadditivity} and Lemma~\ref{lem::concave_V}, for any $b, b' \in \RR_+$, we have
\[
\left| V_i(b) - V_i(b') \right| 
\le \left| V_i(|b - b'|) - V_i(0) \right| 
\le G_i \left| b - b'\right|^{1/p}.
\]
This completes the proof.
\end{proof}

Building on the H\"{o}lder continuity established in Lemma~\ref{lem::continuity_V}, the following lemma demonstrates that the search space for the optimal dual multiplier can be rigorously restricted to a finite interval $[0, U_\lambda]$ for an explicit value of $U_\lambda$.

\begin{lemma}
\label{lem::U_lambda}
Suppose Assumptions~\ref{asp::regularity} and~\ref{asp::inner_oracle} hold. There exists an optimal dual multiplier $\lambda^\star$ satisfying $\lambda^\star\in [0, U_\lambda]$, where 
\[
U_\lambda := \rho^{\frac{1-p}{p}} G.
\]
Here, $G := \max_{i \in [N]} G_i$, and the constants $G_i > 0$ are specified in Lemma~\ref{lem::continuity_subroutine}.
\end{lemma}

\begin{proof}
Let $\lambda^\star$ be the smallest non-negative scalar such that $\frac{1}{N}\sum_{i=1}^N b_i^\star(\lambda^\star)\leq \rho$. In light of \eqref{eq::dual-opt}, this choice of $\lambda^\star$ is indeed optimal.

For any dual multiplier $\lambda > 0$ and index $i \in [N]$, the optimal budget $b_i^\star(\lambda)$ maximizes the penalized objective $V_i(b) - \lambda b$ over $\RR_+$. Since $b = 0$ is a trivially feasible choice, the optimal objective value must be at least that of the zero-budget allocation:
\[
V_i(b_i^\star(\lambda)) - \lambda b_i^\star(\lambda) \ge V_i(0).
\]
Rearranging this inequality and invoking the H\"{o}lder continuity established in Lemma~\ref{lem::continuity_V}, we obtain:
\[
\lambda b_i^\star(\lambda) \le V_i(b_i^\star(\lambda)) - V_i(0) \le G_i \left(b_i^\star(\lambda)\right)^{1/p}
\]
If $p=1$, the inequality simplifies to $\lambda b_i^\star(\lambda) \le G_i b_i^\star(\lambda)$. Thus, for any $\lambda \ge G_i$, we must have $b_i^\star(\lambda) = 0 \le \rho$. 
If $p>1$, we can divide by $\left(b_i^\star(\lambda)\right)^{1/p}$ (assuming $b_i^\star > 0$, as otherwise the bound holds trivially) and solve for $b_i^\star(\lambda)$ to yield $b_i^\star(\lambda) \le ({G_i}/{\lambda})^{\frac{p}{p-1}} \leq ({G}/{\lambda})^{\frac{p}{p-1}},$
for every $i \in [N]$. Evaluating this bound at $\lambda = U_\lambda = \rho^{\frac{1-p}{p}} G$ leads to:
\[
b_i^\star(U_\lambda) \leq \left(\frac{G}{\rho^{\frac{1-p}{p}} G}\right)^{\frac{p}{p-1}} = \rho\implies \frac{1}{N} \sum_{i=1}^{N} b_i^\star(U_\lambda) \leq \rho
\]
for all $p \geq 1$. Recalling that $\lambda^\star$ is defined as the smallest non-negative scalar satisfying $\frac{1}{N}\sum_{i=1}^N b_i^\star(\lambda^\star)\leq \rho$, we obtain $\lambda^\star \leq U_\lambda$.
\end{proof}

We are now ready to prove Lemma~\ref{lem::Bi_eval}.

\begin{proof}[Proof of Lemma~\ref{lem::Bi_eval}]
Due to the concavity and non-decreasing properties of $V_i(b_i)$ established in Lemma~\ref{lem::concave_V}, for any given dual variable $\lambda \ge 0$, the objective function $U(b_i) := V_i(b_i) - \lambda b_i$ is concave and unimodal on the feasible budget domain $[0, \rho N]$. Because $\argmax_{b_i \in [0, \rho N]} \{ V_i(b_i) - \lambda b_i \}$ is a closed set, we let $b_i^\star(\lambda)$ denote the smallest element of this optimal set. In the following, we use the constants $(C,G,R)$ defined in~\eqref{eq:constants} to establish the lemma. 

By Lemma~\ref{lem::U_lambda}, we restrict the search space of the dual multiplier $\lambda$ to $[0, U_\lambda]$. For any fixed $i \in [N]$ and dual candidate $\lambda \in [0, U_{\lambda}]$, let $\hat{b}_i(\lambda)$ be the output of the golden-section search presented in Algorithm~\ref{alg::budget_eval}. Let $b_i$ and $b_i'$ be two distinct budget points evaluated by the algorithm. By Lemma~\ref{lem::evaluate_vi}, the nested evaluation $\hat{V}_i(b_i)$ has a bounded approximation error. Therefore, the evaluated dual objective $\hat{U}(b_i) := \hat{V}_i(b_i) - \lambda b_i$ satisfies:
$$
\left| \hat{U}(b_i) - U(b_i) \right| \le 
\mathcal{E}_{\text{eval}} :=
\left( 9 \epsilon + 6 G \eta^{1/p} \right) + 
\begin{cases} 
R \eta + 2 G (\rho N)^{1/p} \eta^{1 - 1/p} & \text{if } p > 1, \\ 
(R + 2C) \eta + 2 g (\rho N)^r \eta^{1 - r} & \text{if } p = 1. 
\end{cases}
$$
Mirroring the established logic, the algorithm correctly discards suboptimal subintervals whenever the true difference in function values strictly exceeds twice the maximum evaluation error, $2\mathcal{E}_{\text{eval}}$. To see this, suppose without loss of generality that $U(b_i) > U(b_i') + 2\mathcal{E}_{\text{eval}}$. Then, the $\mathcal{E}_{\text{eval}}$-accurate evaluations guarantee
$$ 
\hat{U}(b_i) \geq U(b_i) - \mathcal{E}_{\text{eval}} > U(b_i') + \mathcal{E}_{\text{eval}} \geq \hat{U}(b_i'), 
$$
ensuring the algorithm correctly shrinks the interval while retaining the true optimal solution $b_i^\star(\lambda)$. Conversely, an inexact discard can only occur if the values are indistinguishable within this tolerance, introducing an algorithmic suboptimality gap of at most $2\mathcal{E}_{\text{eval}}$.

The search terminates and returns a final interval of length $\eta$ containing both the returned point $\hat{b}_i(\lambda)$ and the exact maximum $b_i^\star(\lambda)$. By Lemma~\ref{lem::continuity_V}, $V_i(b_i)$ is $(1/p)$-H\"{o}lder continuous for $p \ge 1$. Thus, the objective function variation within this final interval introduces an additional continuous resolution error bounded by:
$$
G_i \eta^{1/p} + \lambda \eta \le G \eta^{1/p} + U_{\lambda} \eta. 
$$
Crucially, the total suboptimality gap is given by the maximum of the discard error and the final resolution error. For simplicity, we upper bound it by their sum, yielding
$$ 
U(b_i^\star(\lambda)) - U(\hat{b}_i(\lambda)) \le 
18\epsilon + 13 G \eta^{1/p} + U_{\lambda} \eta + 
\begin{cases}
2R \eta + 4 G (\rho N)^{1/p} \eta^{1 - 1/p} & \text{if } p > 1, \\ 
2(R + 2C) \eta + 4 g (\rho N)^r \eta^{1 - r} & \text{if } p = 1. 
\end{cases} 
$$
Consequently, the number of golden-section iterations to achieve this error is $O(\log(\rho N / \eta))$. Combined with the complexity of the nested evaluations derived in Lemma~\ref{lem::evaluate_vi}, Algorithm~\ref{alg::budget_eval} runs in time:
$$
\begin{aligned}
&O\!\left( K^2 \cdot \max_{k\in [K]}\mathsf{Cost}_{k,\epsilon} \cdot \log\!\left(\frac{\rho N}{\eta}\right) \cdot \log\!\left(\frac{1}{\eta} \right) \cdot \log\!\left(\frac{\rho N}{\eta} \right) \right) \\
=\, &O\!\left(K^2 \cdot \mathsf{Cost}_{\epsilon} \cdot \log^2\!\left(\frac{\rho N}{\eta}\right) \cdot \log\!\left(\frac{1}{\eta} \right) \right).
\end{aligned}
$$
The proof is completed by bounding the final evaluation error:
$$
\begin{aligned}
&\quad\, \left| \left(\hat{V}_i(\hat b_i (\lambda)) - \lambda \hat b_i (\lambda)\right) - \left(V_i(b_i^\star(\lambda)) - \lambda b_i^\star(\lambda)\right) \right|\\
&\leq  \left|\hat{V}_i(\hat b_i (\lambda)) - V_i(\hat b_i (\lambda))\right| + \left| U(\hat b_i (\lambda)) - U(b_i^\star(\lambda)) \right|\\
&\leq 27\epsilon + 19 G \eta^{1/p} + U_{\lambda} \eta + 
\begin{cases}
3R \eta + 6 G (\rho N)^{1/p} \eta^{1 - 1/p} & \text{if } p > 1, \\ 
3(R + 2C) \eta + 6 g (\rho N)^r \eta^{1 - r} & \text{if } p = 1. 
\end{cases} 
\end{aligned}
$$
This completes the proof.
\end{proof}

\subsection{Proof of Theorem~\ref{thm::error_analysis}}
\label{proof::thm::error_analysis}

Let $\lambda$ be a dual candidate evaluated during the bisection search in Algorithm~\ref{alg::master_eval}. For each sample $i \in [N]$, Algorithm~\ref{alg::master_eval} queries the inner allocation oracle, which returns an approximate budget $\hat{b}_i(\lambda)$. From the proof of Lemma~\ref{lem::Bi_eval}, we have
$$
V_i(\hat{b}_i(\lambda)) - \lambda \hat{b}_i(\lambda) \ge \max_{b \ge 0} \{V_i(b) - \lambda b\} - \mathcal{E}_{\text{inner}},
$$
where 
$$
\mathcal{E}_{\text{inner}} = 18\epsilon + 13 G \eta^{1/p} + U_{\lambda} \eta + 
\begin{cases}
2R \eta + 4 G (\rho N)^{1/p} \eta^{1 - 1/p} & \text{if } p > 1, \\ 
2(R + 2C) \eta + 4 g (\rho N)^r \eta^{1 - r} & \text{if } p = 1. 
\end{cases} 
$$
Denoting the average approximate budget $\hat{r}(\lambda) := \frac{1}{N} \sum_{i=1}^N \hat{b}_i(\lambda)$, we obtain
$$
\frac{1}{N} \sum_{i=1}^N V_i(\hat{b}_i(\lambda)) - \lambda \hat{r}(\lambda) \ge \frac{1}{N} \sum_{i=1}^N \max_{b \ge 0} \{V_i(b) - \lambda b\} - \mathcal{E}_{\text{inner}}.
$$
Defining the exact dual function $g(\lambda) := \lambda \rho + \frac{1}{N} \sum_{i=1}^N \max_{b \ge 0} \{V_i(b) - \lambda b\}$, the above inequality leads to the following primal suboptimality bound:
\begin{equation}
\label{eq::practical_suboptimality_gap}
\frac{1}{N} \sum_{i=1}^N V_i(\hat{b}_i(\lambda)) \ge g(\lambda) - \lambda(\rho - \hat{r}(\lambda)) - \mathcal{E}_{\text{inner}}.
\end{equation}
The outer bisection search narrows the interval $[\lambda^{(1)}, \lambda^{(2)}]$ based on whether the average budget $\hat{r}(\lambda)$ exceeds $\rho$. Let $\lambda_l$ and $\lambda_u$ be the final bounds satisfying $\lambda_u - \lambda_l \le \eta$. Because the algorithm dynamically assigns the bounds based on the threshold $\rho$, the final iterations guarantee $\hat{r}(\lambda_u) \le \rho \le \hat{r}(\lambda_l)$. Algorithm~\ref{alg::master_eval} defines the interpolation weight $\theta := \frac{\rho - \hat{r}(\lambda_u)}{\hat{r}(\lambda_l) - \hat{r}(\lambda_u)}$ and outputs the convex combination $\bar{b}_i := \theta \hat{b}_i(\lambda_l) + (1-\theta) \hat{b}_i(\lambda_u)$. Noting
$$
\frac{1}{N} \sum_{i=1}^N \bar{b}_i = \theta \hat{r}(\lambda_l) + (1-\theta) \hat{r}(\lambda_u) = \rho,
$$
the final output $\{\bar{b}_i\}_{i=1}^{N}$ is indeed feasible. From Lemma~\ref{lem::concave_V}, the true value function $V_i$ is concave. Applying Jensen's inequality to the interpolated allocation and averaging over all $N$ samples yields
$$
\frac{1}{N} \sum_{i=1}^N V_i(\bar{b}_i) \ge \theta \left( \frac{1}{N} \sum_{i=1}^N V_i(\hat{b}_i(\lambda_l)) \right) + (1-\theta) \left( \frac{1}{N} \sum_{i=1}^N V_i(\hat{b}_i(\lambda_u)) \right).
$$
Substituting the primal suboptimality bound~\eqref{eq::practical_suboptimality_gap} for both $\lambda_l$ and $\lambda_u$ gives
\begin{align}\label{eq::error}
    \frac{1}{N} \sum_{i=1}^N V_i(\bar{b}_i) \ge \theta \big[ g(\lambda_l) - \lambda_l(\rho - \hat{r}(\lambda_l)) \big] + (1-\theta) \big[ g(\lambda_u) - \lambda_u(\rho - \hat{r}(\lambda_u)) \big] - \mathcal{E}_{\text{inner}}.
\end{align}
On the other hand, letting
$$
T := \theta(\hat{r}(\lambda_l) - \rho) = \left(\frac{\rho - \hat{r}(\lambda_u)}{\hat{r}(\lambda_l) - \hat{r}(\lambda_u)}\right)(\hat{r}(\lambda_l) - \rho) = (1-\theta)(\rho - \hat{r}(\lambda_u)),
$$
it follows that $T \in [0, \rho]$. Substituting $T$ in \eqref{eq::error} simplifies the bound to
$$
\frac{1}{N} \sum_{i=1}^N V_i(\bar{b}_i) \ge \theta g(\lambda_l) + (1-\theta) g(\lambda_u) - T(\lambda_u - \lambda_l) - \mathcal{E}_{\text{inner}}.
$$
By weak duality, the dual function $g(\lambda)$ provides a direct upper bound on the optimal primal value, i.e., $g(\lambda) \ge \frac{1}{N}\sum_{i=1}^N V_i(b_i^\star)$ for any $\lambda \ge 0$. Applying this inequality to $g(\lambda_l)$ and $g(\lambda_u)$, and noting $T \le \rho$, we obtain
$$
\frac{1}{N} \sum_{i=1}^N V_i(\bar{b}_i) \ge \frac{1}{N}\sum_{i=1}^N V_i(b_i^\star) - \rho(\lambda_u - \lambda_l) - \mathcal{E}_{\text{inner}}.
$$
By invoking the specific bisection tolerance $\lambda_u - \lambda_l \le \eta$, and substituting in $V^\star = \frac{1}{N}\sum_{i=1}^N V_i(b_i^\star)$ as the optimal value of~\eqref{eq::master}, we obtain the final objective guarantee:
$$
\frac{1}{N} \sum_{i=1}^N V_i(\bar{b}_i) \ge V^\star - \rho \eta - \mathcal{E}_{\text{inner}}.
$$
The required number of bisection iterations to achieve the above guarantee is $O(\log(U_{\lambda}/ \eta))$.

In each iteration, the algorithm calls the subproblem allocation (Algorithm~\ref{alg::budget_eval}) for all $N$ samples. Since the budget constraint limits any individual allocation to $0\leq b_i \le \rho N$, the length of the inner search interval is bounded by $\rho N$. Multiplying the outer bisection steps by the $N$ sample calls and the runtime of the inner allocation yields the final stated complexity. 

It remains to establish the second claim, namely the feasibility and objective guarantee of the constructed $2N$-point distribution $\QQ_{2N}$. Algorithm~\ref{alg::master_eval} performs a final call to \texttt{LocalEval} with the interpolated budgets $\bar{b}_i$, which returns the approximate values $\hat{V}_i(\bar{b}_i)$ and the local optimizers $\hat{s}_i$. By construction, the distribution $\QQ_{2N}$ allocates mass $\hat{\alpha}_{ij}$ and perturbations $\hat{v}_{ij}$ (for $j \in \{1, 2\}$) to exactly match these local optimizers. Therefore, its expected loss is exactly $\mathbb{E}_{z \sim \QQ_{2N}} [\ell(z)] = \frac{1}{N}\sum_{i=1}^N \hat{V}_i(\bar{b}_i)$, and its optimal transport cost satisfies 
$$
\frac{1}{N}\sum_{i=1}^N \left( \hat{\alpha}_{i\hat{k}_{i1}} c(\hat{z}_i + \hat{v}_{i\hat{k}_{i1}}, \hat{z}_i) + \hat{\alpha}_{i\hat{k}_{i2}} c(\hat{z}_i + \hat{v}_{i\hat{k}_{i2}}, \hat{z}_i) \right) \leq \frac{1}{N}\sum_{i=1}^N \bar{b}_i = \rho.
$$
This confirms that $\QQ_{2N} \in \cP$. Finally, applying Lemma~\ref{lem::evaluate_vi} to the final \texttt{LocalEval} call guarantees that $\hat{V}_i(\bar{b}_i) \ge V_i(\bar{b}_i) - \mathcal{E}_{\text{eval}}$. Substituting our lower bound on the true interpolated values yields:
$$
\mathbb{E}_{z \sim \QQ_{2N}} [\ell(z)] = \frac{1}{N} \sum_{i=1}^N \hat{V}_i(\bar{b}_i) \ge \frac{1}{N} \sum_{i=1}^N V_i(\bar{b}_i) - \mathcal{E}_{\text{eval}} \ge V^\star - \rho \eta - \mathcal{E}_{\text{inner}} - \mathcal{E}_{\text{eval}}.
$$
Since $\mathcal{E}_{\text{eval}}$ and $\mathcal{E}_{\text{inner}}$ both satisfy the asymptotic scaling in the theorem statement, the total suboptimality gap of the returned distribution $\QQ_{2N}$ is bounded as claimed. $\hfill\qed$

\subsection{Proof of Lemma~\ref{lem::N+n+1_lower_bound}}
\label{app:N+n+1-proof}

To prove Lemma~\ref{lem::N+n+1_lower_bound}, it suffices to construct instances of the dual DRO problem where the number of loss components $K$ is sufficiently large (specifically, $K \ge n+2$) such that $\min\{N+n+1, KN\} = N+n+1$, and then demonstrate that any optimal least-favorable distribution requires at least $N+n+1$ unique atoms. 
The core geometric intuition behind our construction is to design a scenario where the dual equilibrium condition (i.e., the stationarity of the inner primal decision) forces the adversary to shatter a single empirical sample into exactly $n+2$ distinct points. To achieve the $N+n+1$ bound for a general number of samples $N$, we must design the domain and the loss function such that the remaining $N-1$ empirical samples are completely rigid. That is, any transport of their mass must either strictly suboptimize the objective or violate the domain boundaries.

\begin{lemma}
\label{lem::lower_bound_N3}
Consider the $n$-dimensional setting with the feasible set $\mathcal{X} = \{x \in \mathbb{R}^n : \norm{x}_2 \leq 1\}$ and the support set $\mathcal{Z} = \{z \in \mathbb{R}^n : \norm{z}_2 \leq 1\}$. Let $N\ge 3$ and suppose the reference distribution is $\hat{\mathbb{P}}_N = \frac{1}{N}\sum_{i=1}^N \delta_{\hat{z}_i}$, where $\hat{z}_1 = 0$ and $\{\hat{z}_i\}_{i=2}^{N}$ are distinct boundary points satisfying $\norm{\hat{z}_i}_2 = 1$ and $\sum_{i=2}^N \hat{z}_i = 0$. Set the radius $\rho = 1 / (2N)$, and the transport cost $c(z, \hat{z}) = \norm{z - \hat{z}}_2$. Let $\{v_1, \dots, v_{n+1}\}\subset \mathbb{R}^n$ denote the vertices of a regular $n$-simplex centered at the origin with $\norm{v_k}_2 = 1$ for $1\leq k \leq n+1$. Assume the boundary samples $\{\hat{z}_i\}_{i=2}^N$ and the vertices $\{v_k\}_{k=1}^{n+1}$ are disjoint, that is, their maximum inner product $M := \max_{i,\, k} \hat{z}_i^\top v_k$ satisfies $M < 1$. Choose constants $\gamma$ and $\epsilon$ satisfying $M < \gamma < 1$ and $0 < \epsilon < \min\{1-\gamma,\,\gamma-M\}$. Define the $N+n+1$ components as
\begin{align*}
\ell_k(x, z) &= \norm{x}_2^2 - 2x^\top z + v_k^\top z - \epsilon, \quad k = 1, 2, \dots, n+1,\\
\ell_{n+1+k}(x, z) &= \norm{x}_2^2 - 2x^\top z + \gamma \hat{z}_k^\top z, \quad k = 1, 2, \dots, N,
\end{align*}
and let $\ell(x, z) = \max_{1 \leq k \leq N+n+1} \ell_k(x, z)$. Then, the least-favorable distribution $\QQ^\star_{\mathrm{dual}}$ of the dual DRO problem~\eqref{eq::dual_DRO} requires a support of at least $N+n+1$ distinct atoms. Moreover, the unique optimal least-favorable distribution attaining this minimum support is
\begin{align}\label{eq::Qdual-example}
    \QQ^\star_{\mathrm{dual}} = \frac{1}{N} \underbrace{\left[ \frac{1}{2}\delta_{0} + \frac{1}{2(n+1)} \sum_{k=1}^{n+1} \delta_{v_k} \right]}_{\mathrm{Split\,of}\,\hat{z}_1} + \frac{1}{N} \sum_{i=2}^N \delta_{\hat{z}_i}.
\end{align}
\end{lemma}
\begin{proof}
Define the surrogate component functions $g_k: \mathcal{Z} \rightarrow \mathbb{R}$ as 
\begin{align*}
g_k(z) &= v_k^\top z - \epsilon, \quad k = 1, 2, \dots, n+1,\\
g_{n+1+k}(z) &= \gamma \hat{z}_k^\top z, \quad k = 1, 2, \dots, N,
\end{align*}
and let $g(z) = \max_{1\leq k \leq N+n+1} g_k(z)$.

Consider the inner minimization of the dual DRO problem~\eqref{eq::dual_DRO} for the loss function $\ell(x,z) = \norm{x}_2^2 - 2x^\top z + g(z)$. For any distribution $\mathbb{Q}$, the unconstrained minimizer of $\EE_{z\sim \QQ}[\ell(x,z)]$ with respect to $x$ is $x^\star = \mathbb{E}_{z \sim \mathbb{Q}}[z]$. Because the support of $\mathbb{Q}$ is restricted to $\mathcal{Z}$ and $\mathcal{Z} = \mathcal{X}$, this expectation lies within $\mathcal{X}$, and is hence feasible. Substituting $x^\star$ yields the maximization problem:
\begin{equation}
\label{eq::maximin_example}
\max_{\bQ \in \cP} \left\{ \min_{x \in \mathcal{X}} \mathbb{E}_{z \sim \mathbb{Q}} [\ell(x,z)] \right\} = \max_{\bQ \in \cP} \left\{ \mathbb{E}_{z \sim \mathbb{Q}} [g(z)] - \norm{\EE_{z\sim\bQ}[z]}_2^2 \right\}   
\end{equation}
To solve the above maximization problem, we prove the existence of a distribution that jointly achieves the maximum of $\mathbb{E}_{z \sim \mathbb{Q}} [g(z)]$ and the minimum of $\norm{\EE_{z\sim\bQ}[z]}_2^2$, i.e., $\norm{\EE_{z\sim\bQ}[z]}_2^2=0$. 

We first focus on maximizing $\mathbb{E}_{z \sim \mathbb{Q}} [g(z)]$.
To identify the optimal transport plan that achieves this maximum, we evaluate the objective gain per unit of transport cost (the OT efficiency) for moving mass away from the empirical atoms.
\begin{enumerate}
    \item Efficiency of transporting mass from $\hat{z}_1 = 0$: Since $g(\hat z_1) = g(0) = 0$, by the Cauchy-Schwarz inequality, the transport efficiency for any target location $z \neq 0$ is
    $$
    \frac{g(z) - g(\hat z_1)}{\norm{z}_2} = \frac{\max\left\{\max_{k}\left\{v_k^\top z - \epsilon \right\},\, \max_{j} \{\gamma \hat{z}_j^\top z\} \right\} }{\norm{z}_2} \leq \max\left\{1 - \frac{\epsilon}{\norm{z}_2},\, \gamma \right\}.$$Since $z \in \mathcal{Z}$, we have $\norm{z}_2 \leq 1$. Combined with our assumption $\gamma < 1 - \epsilon$, this implies:
    $$
    \max\left\{1 - \frac{\epsilon}{\norm{z}_2},\, \gamma \right\} \leq 1 - \epsilon.
    $$
    This bound is tight. The maximum efficiency of exactly $1 - \epsilon$ is attained if and only if $z = v_k$, for any $1\leq k\leq n+1$.
    \item Efficiency of transporting mass from $\hat{z}_i$ ($2 \le i \le N$): We first establish the baseline objective value $g(\hat{z}_i)$. Because $\norm{\hat{z}_i}_2 = 1$ and $v_k^\top \hat{z}_i \leq M$ for all $1\leq k\leq n+1$, we have
    $$
    g(\hat{z}_i) = \max\left\{\max_k \left\{ v_k^\top \hat{z}_i - \epsilon\right\},\, \max_j \left\{\gamma \hat{z}_j^\top \hat{z}_i\right\}\right\} \leq \max\left\{M - \epsilon,\, \gamma\right\}.
    $$
    By our choice of $\epsilon < \gamma - M$, it holds that $M - \epsilon < \gamma$. Furthermore, evaluating the $(n+1+i)$-th component at $\hat z_i$ yields exactly $g_{n+1+i}(\hat{z}_i) = \gamma \hat{z}_i^\top \hat{z}_i = \gamma$. Thus, $g(\hat{z}_i) = \gamma$. For any target location $z \neq \hat{z}_i$, we have 
    $$
    \begin{aligned}
    g(z) - \gamma &\leq \max\left\{\max_k \left\{ v_k^\top (z - \hat{z}_i) + v_k^\top \hat{z}_i - \epsilon \right\},\, \max_j \left\{ \gamma \hat{z}_j^\top (z - \hat{z}_i) + \gamma \hat{z}_j^\top \hat{z}_i \right\} \right\} - \gamma \\
    &\leq \max\left\{\norm{z - \hat{z}_i}_2 + M - \epsilon,\, \gamma\norm{z - \hat{z}_i}_2 + \gamma \right\} - \gamma\\
    &= \max\left\{\norm{z - \hat{z}_i}_2 + M - \gamma -\epsilon,\, \gamma\norm{z - \hat{z}_i}_2 \right\}.
    \end{aligned}
    $$
    Dividing by the transport cost $\norm{z - \hat{z}_i}_2$ yields the following transport efficiency from $\hat z_i$:
    $$
    \frac{g(z) - g(\hat{z}_i)}{\norm{z-\hat{z}_i}_2} \leq \max\left\{1 - \frac{\gamma + \epsilon - M}{\norm{z - \hat{z}_i}_2},\, \gamma \right\}.
    $$
    Because the maximum diameter of $\mathcal{Z}$ is 2, $\norm{z - \hat{z}_i}_2 \leq 2$. Therefore, the first term in the maximum is bounded by $1 - \frac{1}{2}(\gamma + \epsilon - M)$. By our strict condition $\epsilon < \gamma - M$, this value is strictly less than $1 - \epsilon$. Consequently, the transport efficiency from any boundary sample $\hat{z}_i$ is strictly less than $1 - \epsilon$.
\end{enumerate}

This establishes that the objective gain of moving mass from the origin to $\{v_k\}_{k=1}^{n+1}$ strictly dominates the gain from moving any mass away from the boundary samples $\{\hat{z}_i\}_{i=2}^N$. Therefore, to maximize \eqref{eq::maximin_example}, two conditions must be satisfied simultaneously:
\begin{itemize}
    \item[1.] Maximizing $\mathbb{E}_{z \sim \mathbb{Q}} [g(z)]$: The total transport budget $\rho = 1/(2N)$ must be exhausted entirely on moving mass from $\hat{z}_1 = 0$ to $\{v_k\}_{k=1}^{n+1}$, leaving the remaining empirical samples $\{\hat{z}_i\}_{i=2}^N$ rigidly anchored.
    \item[2.] Minimizing $\norm{\EE_{z\sim\bQ}[z]}_2^2$: To avoid the non-positive penalty $-\norm{\mathbb{E}_{z\sim \mathbb{Q}}[z]}_2^2$ in~\eqref{eq::maximin_example}, the least-favorable distribution must satisfy $\mathbb{E}_{z\sim \mathbb{Q}} [z] = 0$.
\end{itemize}
The proposed distribution $\mathbb{Q}^\star_{\mathrm{dual}}$ in \eqref{eq::Qdual-example} trivially satisfies these conditions and attains a support size of $N+n+1$. To prove it is the unique distribution satisfying these two conditions, consider an arbitrary least-favorable distribution $\mathbb{Q}$. Due to its optimality, this distribution must satisfy both Conditions 1 and 2. Due to Condition 1, it must take the form
$$
\mathbb{Q} = \left[ q_0\delta_0 + \sum_{k=1}^{n+1}q_k \delta_{v_k} \right] + \frac{1}{N} \sum_{i=2}^N \delta_{\hat{z}_i},
$$
where $q_k \geq 0$ for $k=0,1,\dots, n+1$. Conservation of mass at the origin requires $q_0 + \sum_{k=1}^{n+1}q_k = \frac{1}{N}$. The exhausted transport budget requires $\sum_{k=1}^{n+1} q_k \norm{v_k}_2 = \sum_{k=1}^{n+1} q_k = \rho = \frac{1}{2N}$.

Applying Condition 2 and utilizing the assumption that $\sum_{i=2}^{N}\hat{z}_i = 0$ yields 
$$
\mathbb{E}_{z \sim \mathbb{Q}}[z] = q_0 \hat{z}_1 + \sum_{k=1}^{n+1} q_k v_k + \frac{1}{N} \sum_{i=2}^{N} \hat{z}_i = \sum_{k=1}^{n+1} q_k v_k = 0.
$$
where we use the assumptions that $\frac{1}{N} \sum_{i=2}^{N} \hat{z}_i=0$ and $\hat z_1=0$. 
Because $\{v_k\}_{k=1}^{n+1}$ are vertices of a regular $n$-simplex centered at the origin, they are affinely independent and sum to zero. Thus, there exists a constant $\bar{q} > 0$ such that $q_k = \bar{q}$ for all $1 \leq k \leq n+1$. From the budget constraint $\sum_{k=1}^{n+1} q_k = \frac{1}{2N}$, it follows that $\bar{q} = \frac{1}{2N(n+1)}$. The remaining mass anchoring at the origin is $q_0 = \frac{1}{N} - \frac{1}{2N} = \frac{1}{2N}$.
Thus, $\mathbb{Q}^\star_{\mathrm{dual}}$ is the unique least-favorable distribution attaining the optimal value of dual DRO problem~\eqref{eq::dual_DRO}.
\end{proof}


\section{Piecewise Affine Case}
\label{sec::piecewise-linear}
For the special case \(p=1\), Assumption~\ref{asp::regularity}\ref{asp::regular::c} implies that either \(\cZ\) must be compact or the loss function \(\ell(x,z)\) must exhibit sublinear growth. Although this assumption is essential for guaranteeing the existence of a saddle point (cf.\ Lemma~\ref{lem::existence_saddle}), it inevitably excludes an important and widely studied class of DRO problems in which the loss is piecewise affine and the transportation cost is induced by the 1-Wasserstein distance. The following proposition provides a (exact or asymptotic) closed-form characterization of the worst-case distribution. The proof is omitted for
brevity as it follows the same path as that in \cite[Theorem~9]{shafieezadeh2019regularization}.

\begin{proposition}
\label{lem::piecewise_linear}
Under Assumptions~\ref{asp::regularity}\ref{asp::regular::sets} and~\ref{asp::regularity}\ref{asp::regular::ell}, suppose that \(\ell_k(z)=a_k^\top z\) with \(a_k \in \RR^m\) for all \(k \in [K]\), let the support set be \(\cZ=\RR^m\), and let the transportation cost be \(c(z,\hat z)=\norm{z-\hat z}\). Define
\[
\cK := \argmax_{k \in [K]} \norm{a_k}_*, 
\qquad
\cK_i := \argmax_{k \in [K]} a_k^\top \hat z_i,
\quad i \in [N],
\]
where \(\norm{\cdot}_*\) denotes the dual norm of \(\norm{\cdot}\). Moreover, for every \(k \in [K]\), let \(v_k\) satisfy \(\norm{v_k}=1\) and \(a_k^\top v_k=\norm{a_k}_*\).
Then, the following statements hold:
\begin{enumerate}
    \item Suppose that there exists a sample index \(i' \in [N]\) and a component \(k' \in [K]\) such that $k' \in \cK \cap \cK_{i'}$.
    Then, the worst-case expectation problem~\eqref{eq::worst_case} admits an optimal distribution, given by
\begin{align}\label{eq::worst_case_explicit}
    \QQ^\star
    =
    \frac{1}{N}\delta_{\hat z_{i'} + N\rho v_{k'}}
    +
    \frac{1}{N}\sum_{i \neq i'} \delta_{\hat z_i}.
\end{align}

    \item Suppose that $\cK \cap \left(\bigcup_{i=1}^N \cK_i\right)=\emptyset$.
    Then, the worst-case expectation problem~\eqref{eq::worst_case} does not admit an optimal distribution. However, for any \(k' \in \cK\), the sequence of distributions
    \begin{align}\label{eq::worst_case_asymp}
        \QQ_\iota
    =
    \frac{1}{N}\sum_{i=1}^N
    \left(1-\frac{1}{\iota}\right)\delta_{\hat z_i}
    +
    \frac{1}{N}\sum_{i=1}^N
    \frac{1}{\iota}\delta_{\hat z_i + \iota c_i v_{k'}},
    \end{align}
    where \(c_i \geq 0\) satisfy \(\frac{1}{N}\sum_{i=1}^N c_i=\rho\), achieves the optimal value asymptotically as \(\iota \to \infty\).
\end{enumerate}
\end{proposition}


\section{Efficient Worst-Case Oracles for Special Cases}
\label{sec::efficient_oracles}

Beyond implementing the worst-case oracle $\mathsf{WCO}_{k}^{\epsilon}(\hat z,u)$ via generic solvers, we highlight special cases where geometric structure yields highly efficient semi-closed-form updates. We focus on \emph{prox-friendly} settings, where the generalized proximal operator can be evaluated efficiently for any $\lambda\geq 0$:
\begin{equation*}
    z^\star(\lambda)
    \;=\;
    \mathrm{prox}_{\lambda,\,c}^{-\ell_k}(\hat{z})
    \;\in\;
    \argmin_{z \in \mathcal{Z}}
    \Bigl\{-\ell_k(z) + \lambda\, c(z, \hat{z})\Bigr\}.
\end{equation*}
We highlight two highly relevant families and refer to \cite{parikh2014proximal} for a broader treatment:
\begin{itemize}[label=$\diamond$,leftmargin=*]
    \item \textbf{Squared Euclidean cost.} When $c(z,\hat{z}) = \tfrac{1}{2}\|z-\hat{z}\|_2^2$ (e.g., the $2$-Wasserstein cost), the proximal operator admits an efficient, often closed-form, solution under general conditions on $-\ell_k$.
    \item \textbf{Bregman divergence cost.} For $c(z,\hat{z}) = D_{\phi}(z,\hat{z}) := \phi(z) - \phi(\hat{z}) - \nabla\phi(\hat{z})^\top(z-\hat{z})$ with strongly convex $\phi$, the operator is tractable if $\nabla\phi$ is explicitly invertible and $-\ell_k$ has matching algebraic structure. A prominent example is the Kullback--Leibler (KL) divergence, where $\phi(z) = \sum_i z_i \log z_i$.
\end{itemize}

Table~\ref{table::prox-operators} collects representative prox-friendly instances. Their relevance to our framework is formalized by the following lemma.

\begin{table}[H]
\centering
\small
\begin{tabular}{@{}lllll@{}}
\toprule
$\mathcal{Z}$ & $-\ell_k(z)$ & $c(z,\hat{z})$ &
  $\mathrm{prox}_{\lambda,c}^{-\ell_k}(\hat{z})$ & Cost \\
\midrule
$\mathbb{R}^m$
  & $\tfrac{1}{2}z^\top Q z + q^\top z$ (convex quad.)
  & $\tfrac{1}{2}\|z-\hat{z}\|_2^2$
  & $(Q + \lambda I)^{-1}(\lambda\hat{z} - q)$
  & $O(m^3)$/$O(m)$\rule{0pt}{2.6ex} \\[4pt]
$\mathbb{R}^m$
  & $\|z\|_1$ (1-norm)
  & $\tfrac{1}{2}\|z-\hat{z}\|_2^2$
  & $z_i = \mathcal{S}_{1/\lambda}(\hat{z}_i)$
  & $O(m)$\rule{0pt}{2.6ex} \\[4pt]
$\Delta^m$
  & $\sum_i z_i \log z_i$ (neg. entropy)
  & $\tfrac{1}{2}\|z-\hat{z}\|_2^2$
  & $z_i=\frac{1}{\lambda}
 W_0\!\left(
 \lambda e^{\lambda\widehat z_i-1-\nu}
 \right),$\;
    $\nu$ from $\textstyle\sum_i z_i = 1$
  & $O\!\left(m\log\tfrac{1}{\varepsilon}\right)$\rule{0pt}{2.6ex} \\[4pt]
$\mathbb{R}_{++}^m$
  & $-\sum_i \log z_i$ (log-barrier)
  & $\tfrac{1}{2}\|z-\hat{z}\|_2^2$
  & $z_i = \dfrac{\hat{z}_i + \sqrt{\hat{z}_i^2 + 4/\lambda}}{2}$
  & $O(m)$\rule{0pt}{2.6ex} \\[6pt]
$\Delta^m$
  & $a^\top z$ (linear)
  & $\mathrm{KL}(z\,\|\,\hat{z})$
  & $z_i = \dfrac{\hat{z}_i\,e^{-a_i/\lambda}}
                 {\sum_j \hat{z}_j\,e^{-a_j/\lambda}}$
  & $O(m)$\rule{0pt}{2.6ex} \\[6pt]
$\mathbb{R}_{++}^m$
  & $\sum_i z_i \log z_i$
  & $\mathrm{KL}(z\,\|\,\hat{z})$
  & $z_i = e^{-1/(1+\lambda)}\hat{z}_i^{\,\lambda/(1+\lambda)}$
  & $O(m)$\rule{0pt}{2.6ex} \\[6pt]
$\mathbb{R}_{++}^m$
  & $-\sum_i \log z_i$
  & $\mathrm{KL}(z\,\|\,\hat{z})$
  & $z_i = \dfrac{1}{\lambda W_0(1/(\lambda \hat z_i))}$
  & $O(m)$\rule{0pt}{2.6ex} \\[4pt]
\bottomrule
\end{tabular}
\caption{Selected prox-friendly instances.
\emph{Notation:}
$[\,\cdot\,]_+ := \max(\cdot,0)$;\;
$\mathcal{S}_\tau(t) := \operatorname{sign}(t)\,[|t|-\tau]_+$ (soft-thresholding);\; $W_0$ is the principal branch of the Lambert--$W$ function;\;
$\Delta^m := \{z \ge 0 : \sum_i z_i = 1\}$ (probability simplex);\; $\mathrm{KL}(z\|\hat z):=\sum_i z_i\log\tfrac{z_i}{\hat z_{i}}-z_i+\hat z_{i}$ (Kullback-Leibler (KL) divergence);\; $\mathbb{R}_{++}^m := \{z: z_i>0, i\in [m]\}$. For the negative-entropy entry, $\nu$ is the normalization multiplier and can be found by one-dimensional bisection.}
\label{table::prox-operators}
\end{table}

\begin{lemma}\label{lem::prox-lagrangian}
Suppose Assumption~\ref{asp::regularity} holds, $c(\cdot,\hat{z})$ is strongly convex, and $c(\cdot,\hat{z}) \geq 0$ with equality at $z = \hat{z}$. For any $u > 0$ and $\hat{z} \in \mathcal{Z}$, the optimal solution to
    \begin{equation*}
        z^\star \in \argmax_{z \in \mathcal{Z}}
        \left\{ \ell_k(z) : c(z, \hat{z}) \leq u \right\}
    \end{equation*}
    is given as follows:
    \begin{itemize}
        \item If $c\left(\mathrm{prox}_{0,\,c}^{-\ell_k}(\hat{z}), \hat{z}\right) \leq u$, then $z^\star = \mathrm{prox}_{0,\,c}^{-\ell_k}(\hat{z})$.

        \item If $c\left(\mathrm{prox}_{0,\,c}^{-\ell_k}(\hat{z}), \hat{z}\right) > u$, then $z^\star = \mathrm{prox}_{\lambda^\star,\,c}^{-\ell_k}(\hat{z})$, 
            where $\lambda^\star > 0$ is the unique solution to 
            $c\!\left(\mathrm{prox}_{\lambda^\star,\,c}^{-\ell_k}(\hat{z}),\,\hat{z}\right) = u$.
    \end{itemize}
\end{lemma}
\begin{proof}
    The proof follows from standard KKT conditions \cite{parikh2014proximal, beck2017first}. The uniqueness of $\lambda^\star$ holds because the strong convexity of $c(\cdot, \hat z)$ ensures $\lambda \mapsto c\!\left(\mathrm{prox}_{\lambda,\,c}^{-\ell_k}(\hat{z}),\,\hat{z}\right)$ is continuous and strictly decreasing for $\lambda \geq 0$.
\end{proof}

\begin{remark}
    Strong convexity of $c(\cdot,\hat{z})$ is satisfied globally by the squared Euclidean cost, and by the KL divergence on any compact subset of the interior of $\mathcal{Z}$.
\end{remark}

Thus, implementing $\mathsf{WCO}_{k}^{\epsilon}(\hat{z}, u)$ reduces to a one-dimensional root-finding problem for $\lambda \geq 0$. If the unconstrained maximizer $\mathrm{prox}_{0,\,c}^{-\ell_k}(\hat{z})$ is infeasible, we find $\lambda^\star$ such that
\begin{equation*}
    g(\lambda) := c\!\left(\mathrm{prox}_{\lambda,\,c}^{-\ell_k}(\hat{z}),\,\hat{z}\right) - u = 0.
\end{equation*}
Since $g$ is continuous and strictly decreasing, with $g(0) > 0$ and $\lim_{\lambda \to \infty} g(\lambda) = -u < 0$, a valid bracketing interval exists and bisection converges geometrically. In Table~\ref{table::prox-operators}, evaluating the proximal operator and $c(\cdot,\hat{z})$ takes $O(m)$ time. The only exception is the quadratic case, which incurs $O(m^3)$ naively but reduces to $O(m)$ with eigenvalue pre-factorization of $Q$. Consequently, the overall oracle complexity via bisection is $O\!\left(m\log(1/\epsilon)\right)$.

\end{document}